\documentclass[12pt]{amsart}
\usepackage{amsthm}
\usepackage{amsmath,amsxtra,amssymb,amsfonts,latexsym, mathrsfs, dsfont, bm, mathtools, comment}
\usepackage[initials,alphabetic]{amsrefs}
\usepackage[usenames]{xcolor}
\usepackage[all]{xy}
\usepackage[hyperindex]{hyperref}
\usepackage{filecontents}

\hypersetup{colorlinks=true}
\hypersetup{
	linkcolor=purple,
	citecolor = cyan,
          }

\makeatletter
\newcommand{\opnorm}{\@ifstar\@opnorms\@opnorm}
\newcommand{\@opnorms}[1]{%
  \left|\mkern-1.5mu\left|\mkern-1.5mu\left|
   #1
  \right|\mkern-1.5mu\right|\mkern-1.5mu\right|
}
\newcommand{\@opnorm}[2][]{%
  \mathopen{#1|\mkern-1.5mu#1|\mkern-1.5mu#1|}
  #2
  \mathclose{#1|\mkern-1.5mu#1|\mkern-1.5mu#1|}
}
\makeatother

\newtheorem{lemma}{Lemma}[section]
\newtheorem{prop}[lemma]{Proposition}
\newtheorem{theorem}[lemma]{Theorem}
\newtheorem{cor}[lemma]{Corollary}

\numberwithin{equation}{section}

\theoremstyle{remark}
\newtheorem{rem}[lemma]{Remark}

\theoremstyle{definition}
\newtheorem{defn}[lemma]{Definition}

\newcommand{\lge}{\langle}
\newcommand{\rge}{\rangle}

\newcommand{\ax}{\mathcal{A}}
\newcommand{\bx}{\mathcal{B}}
\newcommand{\cx}{\mathcal{C}}
\newcommand{\dx}{\mathcal{D}}

\newcommand{\hx}{\mathcal{H}}
\newcommand{\kx}{\mathcal{K}}
\newcommand{\lx}{\mathcal{L}}
\newcommand{\mx}{\mathcal{M}}
\newcommand{\nx}{\mathcal{N}}
\newcommand{\rx}{\mathcal{R}}
\newcommand{\sx}{\mathcal{S}}
\newcommand{\ux}{\mathcal{U}}
\newcommand{\vx}{\mathcal{V}}
\newcommand{\ox}{\mathcal{O}}

\newcommand{\fz}{\mathbb{F}}
\newcommand{\nz}{\mathbb{N}}
\newcommand{\rz}{\mathbb{R}}
\newcommand{\cz}{\mathbb{C}}

\newcommand{\tz}{\mathbb{T}}
\newcommand{\zz}{\mathbb{Z}}

\newcommand{\hz}{\mathbb{H}}
\newcommand{\qz}{\mathbb{Q}}

\newcommand{\Ga}{\Gamma}

\newcommand{\La}{\Lambda}

\newcommand{\al}{\alpha}
\newcommand{\de}{\delta}
\newcommand{\si}{\sigma}
\newcommand{\ga}{\gamma}
\newcommand{\la}{\lambda}
\newcommand{\bt}{\beta}
\newcommand{\ta}{\theta}
\newcommand{\om}{\omega}

\newcommand{\eps}{\varepsilon}
\newcommand{\vph}{\varphi}
\newcommand{\vsi}{\varsigma}
\newcommand{\vpi}{\varpi}
\newcommand{\vta}{\vartheta}

\newcommand{\8}{\infty}
\newcommand{\td}{\tilde}
\mathchardef\dash="2D

\newcommand{\id}{\operatorname{id}}
\newcommand{\cb}{\operatorname{cb}}
\newcommand{\supp}{\operatorname{supp}}

\newcommand{\ii}{\mathrm{i}}
\newcommand{\dd}{\mathrm{d}}

\newcommand{\pl}{\hspace{.1cm}}
\newcommand{\lel}{\pl = \pl}
\newcommand{\kl}{\pl \le \pl}

\newcommand{\ten}{\otimes}
\newcommand{\kla}{\left ( }
\newcommand{\mer}{\right ) }

\allowdisplaybreaks[1]
\title[quantum Gromov--Hausdorff convergence]{Harmonic analysis approach to Gromov--Hausdorff convergence for noncommutative tori}
\author{Marius Junge}
\address{(M.J.) Department of Mathematics, University of Illinois, Urbana, IL 61801}
\email{mjunge@illinois.edu}
\author{Sepideh Rezvani}
\address{(S.R.) Department of Mathematics, University of Illinois, Urbana, IL 61801}
\email{rezvani2@illinois.edu}
\author{Qiang Zeng}
\address{(Q.Z.) Mathematics Department, Northwestern University, 2033 Sheridan Road, Evanston, IL  60208}
\email{qzeng.math@gmail.com}
\date{\today}

\begin{document}
\maketitle
\begin{abstract}
  We show that the rotation algebras are limit of matrix algebras in a very strong sense of convergence for algebras with additional Lipschitz structure. Our results generalize to higher dimensional noncommutative tori and operator valued coefficients. In contrast to previous results by Rieffel, Li, Kerr, and Latr\'emoli\`ere we use Lipschitz norms induced by the `carr\'e du champ' of certain  natural dynamical systems, including the heat semigroup.
\end{abstract}

\tableofcontents

\section{Introduction}

The theory of quantum metric spaces, inspired by Connes' work in noncommutative geometry \cite{Co94} and results from string theory, aims to capture the notion of convergence for noncommutative algebras. From the examples in quantum field theory (see e.g. \cites{GKP, CDS, Sc98, Kr99, SW99, Ma99, KS00, HL01}), it is clear that the notion of convergence is  motivated through the notion of distance in the commutative situation. It is fair to say that Rieffel and his collaborators have come a long way, before a nice conceptual framework was created which satisfies mathematical principles and includes the expected examples. The aim of this paper is to use tools from harmonic analysis and operator space theory to show that rotation algebras provide natural examples for even the strongest form of matricial approximation known as of now.

Rotation algebras provide paradigmatic examples in noncommutative geometry. These examples are close enough to their commutative counterparts, classical higher dimensional tori, and still exhibit some genuine noncommutative features. In order to talk about geometry,
it is necessary to identify the additional structure  concerning derivatives on top of identifying the algebra of noncommutative functions. Motivated by examples from string theory, Rieffel decided to focus on the additional structure of Lipschitz functions. Let us be more precise and consider a spectral triple $(A,H,D)$ given by a $^*$-algebra $A$, a self-adjoint operator $D$ and a Hilbert space $H$ so that
 \[ \|a\|_{Lip} \lel \| [D,a]\| \pl.\]
For Spin-manifolds, $D$ is given by the Dirac operator, and then the Lipschitz norm reflects the Riemannian distance, i.e.
 \[ d(p,q) \lel \sup_{\|f\|_{Lip}\le 1} |f(p)-f(q)| \pl .\]
Starting from this point of view, Rieffel was able to make sense of certain statements from physics, stating that there are  \emph{matrix algebras converging to spheres}; see \cite{Ri04}. In this paper we want to prove that similarly \emph{rotation algebras of even dimensional tori are limits of matrix algebras}.

Similar results have been obtained by Latr\'emoli\`ere \cite{La05}, Kerr and Li \cite{KL09}. The starting point of the theory is to consider the Gromov--Hausdorff distance
 \[
 d_{GH}(X,Y) \lel \inf_{X\subset Z,Y\subset Z} \max\{\sup_{x} d(x,Y), ~\sup_{y}d(X,y)\}
 \]
where the infimum is taken over isometric embeddings $X\subset Z$, $Y\subset Z$ in a common metric space and define for C$^*$-algebras $A$ and $B$
 \[
 d_{GH}((A, \|\cdot \|_{Lip}), (B, \|\cdot \|_{Lip})) \lel d_{GH}((S(A), d_L),(S(B), d_L)) \pl,
 \]
where the state space is equipped with the Lip-metric
 \[
 d_L(\phi,\psi) \lel \sup_{\|a\|_{Lip}\le 1} |\phi(a)-\psi(a)| \pl .
 \]
However, for C$^*$-algebras and C$^*$-algebraists, it is more natural to consider the objects themselves. It appears that finding the `correct' notion of distance for $C^*$-algebras with a `Lipschitz norm' has gone through a variety of different stages, making certain compromises between practicality in terms of examples and validity  of conceptual properties. Even during the time of finalizing this paper, it was necessary to shoot for a moving target. However, the principal tools for rotation algebras, based on harmonic analysis, are very concrete. The starting point of our work is the definition of distance for compact  order-unit spaces given by Li \cite{Li06}
 \[ d_{oq}(A,B) \lel d'_{GH}(D_R(A),D_R(B)) \pl ,\]
where $D_R(A)=\{a\in A: \|a\|_{Lip}\le 1,\|a\|\le R\}$. We refer to \cite{Li06} for additional restriction on which type of embeddings  and superspaces $Z$ are permitted in realizing this distance $d'_{GH}$ in the space of order unit spaces.

\begin{theorem}(Li) $d_{GH}((A, \|\cdot \|_{Lip}),(B, \|\cdot \|_{Lip}))\kl \frac52 d_{oq}(A,B)$. Moreover, for a continuous field of compact quantum metric spaces $\ax$ on $\nz\cup \{\8\}$ with coordinate maps $\pi_n$, the following criterion implies convergence. Let $R$ be larger that the diameter of the space $S(\pi_n(\ax))$ for all $n$. For all $\eps>0$, there exists $n_0$ and $x_1,...,x_N \in \ax$ such that for $x\in D_R(\pi_n(\ax))$ and $n>n_0$ there exists $j\in \{1,... , N\}$ such that $\pi_\8(x_j)\in D_R(\pi_\8(\ax))$ and $\|\pi_n(x_j)-x\|<\eps$.
\end{theorem}

In this criterion, the continuity of the family $\ax$ is required in order to ensure  that $\pi_n(x)$ is well-defined for nice elements $x\in \pi_\8(\ax)$. It is important to note that in establishing this criterion Li has to leave the category of C$^*$-algebras and work within ordered unit spaces. Nevertheless it is desirable to work with a notion of distance so that $d(A,B)=0$ implies that the corresponding C$^*$-algebras are isomorphic, and it appears that this leads to rather involved definitions of propinquity, and dual propinquity, see  \cites{La16, La15a, La15}. In fact, even in our writing process, we have been caught up with this changing set of definitions, and hence proved our results in different setups. Let us start with the easiest example and the easiest form of convergence. Our target algebra is $C(\mathbb{T})$ with
$D=\ii\frac{\dd}{\dd t}$ and
 \[ \|f\|_{Lip} \lel \|f'\| \pl .\]

In contrast to \cite{La05} our approach towards Lipschitz norms is dynamical in nature. More precisely, we start with a semigroup $T_t=e^{-tA}$ of completely positive trace-preserving maps and gradient form
 \[ 2 \Gamma^A(x,y) \lel A(x^*)y+x^*A(y)-A(x^*y) \pl .\]
We refer to \cite{CS03} for the domain issues here. Our Lipschitz norm is given by
 \[ \|x\|_{Lip_A} \lel \max\{\|\Gamma^A(x,x)\|^{1/2},\|\Gamma^A(x^*,x^*)\|^{1/2}\} \pl .\]
For $C(\mathbb{T})$ we may consider the heat semigroup $T_t(e^{2\pi \ii k\cdot})=e^{-k^2 t}e^{2\pi \ii k\cdot}$ and the Poisson semigroup $P_t(e^{2\pi \ii k\cdot})=e^{-|k|t}e^{2\pi \ii k\cdot}$. For the heat semigroup we find the usual Lipschitz structure, while the Poisson structure is a little more exotic. Since  $C(\mathbb{T})=C^*(\zz)$ is a group algebra, we should expect the approximating examples $B_n=C^*(\zz_n)$ to be of the same nature, where $\zz_n=\zz/n\zz$ is the finite cyclic group. Indeed, the Poisson semigroup given by $P_t(e^{\frac{2\pi \ii k \cdot}{n}})=e^{-t|k|_n}e^{\frac{2\pi \ii k \cdot}{n}}$, where $|k|_n=\min(k,n-k)$, gives a natural finite-dimensional analogue. For the heat semigroup we have to work with the semigroup  $T_t (e^{\frac{2\pi \ii k \cdot}{n}})=e^{-\frac{tn^2}{2\pi^2} (1-\cos(2\pi k/n))}e^{\frac{2 \pi \ii k \cdot}{n}}$. Our first observation is that we have a natural $^*$-homomorphism $\pi_n:C^*(\zz)\to C^*(\zz_n)$ given by $\pi_n(e^{2\pi \ii \cdot})=e^{\frac{2\pi \ii\cdot}{n}}$ such that
  \begin{equation}\label{convv} \lim_{n\to\8} \|\pi_n(x)\|_{C^*(\zz_n)} \lel \|x\|_{C^*(\zz)}\pl ,\quad \lim_{n\to\8} \|\pi_n(x)\|_{Lip_n}
  \lel \|x\|_{Lip}  \end{equation}
holds for trigonometric polynomials. Here $Lip$ norm refers to the Poisson semigroup in both cases, or the heat semigroup in both cases.

\begin{theorem} $(C^*(\zz_n),Lip_n)$ converges to $(C^*(\zz),Lip)$ with respect to $d_{oq}$ and also with respect to the quantum Gromov--Hausdorff distance.
\end{theorem}

Indeed, we verify Li's criterion by working with a finite set of Fourier coefficients  $(\pi_n(e^{2\pi \ii k\cdot}))_{-m\le k\le m}$ using tools from harmonic analysis, inspired from \cite{JM10}.
The argument for  rotation algebras $\ax_{\theta}$ generated by unitaries $u,v$ such that
 \[ uv\lel e^{2\pi \ii \theta}vu \pl ,\]
is very similar.   Indeed, we consider the finite dimensional version $\ax_{\theta_n}(n)$ generated by $u(n),v(n)$ such that
 \[ u(n)v(n)\lel e^{2\pi \ii \theta_n}v(n)u(n)\pl,\pl u(n)^n=1=v(n)^n \pl .\]
In case $\theta_n=p/n$ and $(p,n)=1$ are relatively prime we obtain the full matrix algebras. We use the Lipschitz norm transferred from $\zz_n^2$ by
 \[ T_t(u(n)^kv(n)^l) \lel e^{-t(\psi_n(k)+\psi_n(l))}u(n)^kv(n)^l \]
where $\psi_n(k)=\frac{n^2}{2\pi^2}(1-\cos(2\pi k/n))$ is the generator for the heat semigroup and $\psi_n(k)=|k|_n$ for the Poisson semigroup. Using the maps
$\pi_n(u^kv^l)=u(n)^kv(n)^l$ we are able to verify the compactness and continuity requirements in Li's criterion and obtain:
\begin{theorem}\label{main} Let $0\le \theta \le 1$. Then there exists a sequence of matrix algebras $\ax_{\theta_{n_j}}(n_j)$ which converge to $\ax_{\theta}$ in the quantum Gromov--Hausdorff distance, with respect to the heat or Poisson Lipschitz norm.
\end{theorem}

In fact our results hold for all semigroups for which $u^kv^l$ are eigenfunctions and the choice of the matrices are independent of the specific semigroup.  We can also generalize this result for higher dimensional tori $\ax_{\Theta}^{d}$ generated by a sequence of unitaries $u_k$ satisfying
 \[ u_ku_l\lel e^{2\pi i \theta_{kl}}u_lu_k \pl .\]
Using the heat semigroup for the Lipschitz norm we deduce that
  \[ \|x\|_{Lip} \sim \max_k \|\delta_k(x)\|
 \]
where the derivations are determined by $\delta_k(u_j)=\varDelta_{kj}$ and $\varDelta_{kj}$ is the Kronecker delta function. It is clear that $\de_k$ is an analogue of directional derivative in the commutative setting. This Lipschitz norm has been identified, at least for $d=2$, in Connes' approach via spectral triples (see \cite{Co94}). Note here that the continuity assumption is based on extension of a result of Rieffel \cite{R89}, refined  by Haagerup and R\o rdam \cite{HR}, on the continuous field given by the family $\ax_{\Theta}^{d}$. Using Blanchard's extension of Kirchberg's exact embedding theory into $C(X;\ox_2)$ (see \cite{B97}), one can considerably improve the convergence properties in this case. The great advantage of Blanchard's argument for convergence of nuclear algebras has been discovered in \cite{KL09}. In order to formulate our results in this setting, we will define a Lipschitz structure for tensor products. Indeed, in our situation we can identify a derivation $\delta:\ax_{\Theta}^d \to M$ and a conditional expectation $E:M\to (\ax_{\Theta}^d)''$ such that
 \[ \Gamma^A(x,y) \lel E(\delta(x)^*\delta(y)) \pl .\]
This allows us to define for $x\in \kx\otimes \ax_\Theta$,
 \[
 \|x\|_{Lip} \lel \max\{\|\id\ten E[(\id\ten\delta(x))^*(\id\ten \delta(x))]\|^{1/2},
 \|\id\ten E[(\id\ten\delta(x^*))^*(\id\ten \delta(x^*))]\|^{1/2}\} \pl .
 \]
 Here $\kx$ is the space of compact operators on $\ell_2$.

\begin{theorem} There exist a sequence $B_n=\ax_{\Theta_n}^{2d}(k_n)$ of matrix algebras, and $^*$-homomorphisms $\pi_n:B_n\to \bx(H)$, $\pi_{\infty}:\ax^{2d}_{\Theta}\to \bx(H)$ with the following convergence property. For  every $\eps>0$ there exists an $n_{\eps}$ such that for all $n>n_\eps$
 \begin{enumerate}
 \item[I)] for every $x\in \kx\ten \ax^{2d}_{\Theta}$ with $\|x\|_{Lip}\le 1$, there exists a $y\in \kx\ten B_n$ with  $\|y\|_{Lip}\le 1$ and
      \[ \|\pi_\8(x)-\pi_{n}(y)\|\le \eps \pl ,\]
 \item[II)] for every $x\in \kx\ten B_n$ with $\|x\|_{Lip}\le 1$, there exists a $y\in \kx\ten \ax^{2d}_\Theta$ with $\|y\|_{Lip}\le 1$ and
      \[ \|\pi_n(x)-\pi_{\infty}(y)\|\le \eps \pl .\]
\end{enumerate}
\end{theorem}
A very similar  criterion is used in Latr\'emoli\`ere \cite{La15} to show convergence in the propinquity sense. Starting from I) and II) it is very easy to construct an operator space $X$ which witnesses the fact that the Gromov--Hausdorff distance between $S(\ax_{\Theta})$ and $S(B_n)$ is $\le \eps$. Indeed, fix operator space embeddings $v_n:(B_n,\|\cdot \|_{Lip})\to \bx(H_1)$, $v_{\infty}:(\ax_{\Theta},\|\cdot \|_{Lip})\to \bx(H_2)$ and consider the operator space
     \[ X = \left \{ \kla \begin{array}{ccc}  v_n(a) &0&0 \\
      0& v_{\infty}(b) &0 \\
      0& 0& \frac{1}{\eps}(\pi_\infty(a)-\pi_n(b))
      \end{array}\mer : a\in \ax_{\Theta}, b\in B_n \right \} \pl .\]
Then the quotient maps $q_{\ax_\Theta}: X\to (\ax_{\Theta},\|\cdot\|_{Lip})$ and $q_{n}:X\to (B_n,\|\cdot\|_{Lip})$ are completely surjective. It turns out that now $q_{\ax_\Theta}^*(S(\ax_{\Theta}))$ and $q_n^*(S(B_n))$ are both embedded isometrically in the unit ball of $X^*$ and are $\eps$-close; see \cite[Theorem 6.3]{La16}. Moreover, this easily extends to matrix-valued states. It seems that dual spaces of operator spaces are natural candidates for realizing the Gromov--Hausdorff distance for state spaces; see Section \ref{s:prop} for additional structure of the space $X$ above. In this sense the cb-propinquity criterion is stronger, and probably also more transparent, than Li's criterion for order-unit spaces.
Our operator space embedding for the Lipschitz norms are fairly explicit.

The extension to matrix-valued coefficients requires some conceptual improvements, and therefore is postponed to Section \ref{cb qgh}, \ref{higher cb qgh} and \ref{s:prop}. It appears that even for $d=1$ and the torus this strong notion of convergence has been overlooked. The convergence for matrix-valued coefficients is one of the benefits of our method, which is based on noncommutative harmonic analysis and operator space theory. Note that harmonic analysis on noncommutative tori from the perspective of operator space theory has been studied extensively in recent years; see \cite{CXY, XXY, XXX} and the references therein.

The paper is organized as follows. In Section \ref{s:cnl}, we define semigroups associated to conditionally negative length functions on groups which will be used repeatedly. Then some general estimates from noncommutative harmonic analysis are given in Section \ref{anaest}. We prove quantum Gromov--Hausdorff convergence for $C(\tz)$ and for the noncommutative two tori via Li's criteria in Section \ref{ct} and \ref{A_theta}, respectively. In Section \ref{cb qgh} and \ref{higher cb qgh}, we propose a completely bounded version of quantum Gromov--Hausdorff convergence and extend the previous approximation results to two and higher (even) dimensions in this sense. Finally, we show that matrix algebras converge to the noncommutative $2d$ tori in the c.b. version of Gromov--Hausdorff propinquity in Section \ref{s:prop}.

\subsection*{Acknowledgements}
M.J.~was partially supported by NSF grant DMS-1501103. Q.Z.~would like to thank the support from the UIUC graduate college dissertation fellowship, CMSA at Harvard University, MSRI, and the mathematics department of Northwestern University during the preparation of this paper. We are grateful to Florin Boca for helpful conversations and suggestions on a draft of this paper. Finally, we would like to thank the anonymous referees for the careful reading and many constructive suggestions, which have helped to improve this paper.

\section{Conditionally negative length functions on groups}\label{s:cnl}
Although the objects we study in this paper are $\mathrm{C}^*$-algebras (and order unit spaces), we will use various estimates in noncommutative $L_p$ spaces. To this end, we need to work in the context of von Neumann algebras. We refer to e.g. \cites{BO08,JMP10,JZ12} and the references therein for the unexplained facts in the following. Let $(\nx,\tau)$ be a noncommutative $W^*$ probability space. Here $\nx$ is a finite von Neumann algebra and $\tau$ is a normal faithful tracial state. Let $(T_t)_{t\ge0}$ be a point-wise weak* continuous semigroup acting on $(\nx,\tau)$ such that every $T_t$ is unital, normal, completely positive and self-adjoint in the sense that $\tau(T_t(f)g)=\tau(fT_t (g))$ for every $f,g\in \nx$. We will call a semigroup satisfying these conditions a noncommutative symmetric Markov semigroup. One can extend $T_t$ to a strongly continuous semigroup of contractions on $L_2(\nx,\tau)$ (actually on $L_p(\nx,\tau)$ for all $1\le p<\8$). Here the noncommutative $L_p(\nx,\tau)$ space is the closure of $\nx$ in the norm $\|f\|_p=[\tau(f^*f)^{p/2}]^{1/p}$ for $1\le p<\8$ and $\|f\|_\8=\|f\|=\|f\|_\nx$, the operator norm. We denote by $A$ the infinitesimal generator of $T_t$, i.e. $T_t = e^{-tA}$. We define the gradient form associated to $A$ (Meyer's `carr\'e du champ') by
\begin{equation}\label{carre}
\Ga^A(f,g)=\frac12[A(f^*)g+f^*A(g)-A(f^*g)],
\end{equation}
for $f,g$ in the domain of $A$. Our major examples involve groups with conditionally negative length functions.

Let $G$ be a countable discrete group. Let $\la: G\to \bx(\ell_2(G))$ be the left regular representation of $G$ given by $\la(x)\de_y= \de_{xy}$ for $x,y\in G$, where $(\de_x)_{x\in G}$ are the natural unit vectors of $\ell_2(G)$, the natural Hilbert space associated to $G$. Let $C_r^*(G)$ and $LG$ denote the reduced $\mathrm{C}^*$-algebra and von Neumann algebra of $G$, respectively. They are the norm closure and weak* closure of  the linear span of $\la(G)$ in $\bx(\ell_2(G))$, respectively. There is a canonical normal faithful tracial state $\tau_G$ on $C^*_r(G)$ and $LG$ given by $\tau_G(f)=\lge \de_e, f \de_e\rge$, where $\lge\cdot,\cdot\rge$ is the inner product on $\ell_2(G)$ and $e$ is the identity of $G$. We will also consider the group algebra $\cz(G)$ in the following, which will be regarded as a dense subalgebra of $C^*_r(G)$ and $LG$ in the respective topology. A function $\psi: G\to \rz_+$ is called a length function if $\psi(e)=0$ and $\psi(x)=\psi(x^{-1})$. A length function $\psi$ is said to be conditionally negative if $\sum_x \bt_x=0$ implies that $\sum_{x,y}\bar\bt_x\bt_y\psi(x^{-1}y) \le 0$. Here the sums are over finite indices. By Schoenberg's theorem, a conditionally negative length function $\psi$ gives rise to a completely positive semigroup $(T_t)_{t\ge0}$ acting on $LG$, which is defined by $T_t\la(x)=e^{-t\psi(x)}\la(x)$. It is well known that $T_t$ thus defined is a noncommutative symmetric Markov semigroup and its generator is given by $A\la(x)=\psi(x)\la(x)$; see e.g. \cite{JZ12}. The Gromov form $K$ in this context is defined as
\[
K(x,y)=\frac12[\psi(x)+\psi(y)-\psi(x^{-1}y)], \mbox{ for }x,y\in G.
\]
It is well known that $\psi$ is conditionally negative if and only if $K$ is positive semidefinite as a matrix; see e.g. \cite{BO08}*{Appendix D} for more details. We can write the gradient form as
\begin{equation}\label{gagro}
\Ga^\psi(f,g) = \sum_{x,y} \bar {\hat f}(x)K(x,y)\hat g(y)\la(x^{-1}y)
\end{equation}
for $f=\sum_x \hat f(x)\la(x)\in LG$ and $g=\sum_y \hat g(y)\la(y)\in LG$ being finite linear combinations. In the following, we will frequently ignore the superscript $A$ and $\psi$ in the notation of gradient form for short.

In this paper, we will mainly work with $G=\zz^d$ or $G=\zz_n^d=(\zz/n\zz)^d$. In this paragraph we write $\la$ and $\la_n$ for the left regular representations of $\zz$ and $\zz_n$, respectively. Using the Fourier transform, we can identify $\la(k)$ with $e^{2\pi \ii k \cdot}$ for $k\in \zz$, identify $C_r^*(\zz)$ with $C(\tz)$, the continuous functions on the torus $\tz=\hat\zz$, and identify $L\zz$ with $L_\8(\tz)$. Since the dual group of $\zz_n$ is $\zz_n$, we can identify $\la_n(j)$ with $\exp(\frac{2\pi \ii j\cdot}n)$ for $j\in \zz_n$ and $C^*_r(\zz_n)=L(\zz_n)\simeq L_\8(\zz_n)=\ell_\8(n)$. The elements in the group algebras $\cz(\zz)$ and $\cz(\zz_n)$ will be referred to as trigonometric polynomials in the following. Here the induced trace on $\ell_\8(n)$ is the normalized trace on the $n\times n$ matrix algebra $M_n$ where $\ell_\8(n)$ is regarded as the diagonal subalgebra of $M_n$. In other words, $\la_n(j)$  is identified with
\begin{equation}\label{ujn}
u_j(n)= \left( \begin{array}{ccccc}
1 & & &&  \\
& e^{\frac{2\pi \ii j}{n}} &&& \\
 &  & e^{\frac{2\pi \ii j2}{n}}&& \\
 &&&\ddots&\\
 &&&&e^{\frac{2\pi \ii j(n-1)}{n}}\\
 \end{array} \right)\in \ell_{\infty}(n)
\end{equation}
We will consider two types of conditionally negative length functions on $\zz$ and $\zz_n$. The first type is the word length function on $\zz$ and $\zz_n$, respectively, i.e. $\psi(k)=|k|$ for $k\in\zz$ and
\begin{equation}\label{cnl0}
  \psi_n(k)=|k|_n=\min\{k, n-k\}, \mbox{ for } k\in\zz_n=\{0,1,...,n-1\}.
\end{equation}
It is known that the word length functions are conditionally negative; see e.g. \cites{JZ13, J3P13}. To unify our notation, we will write $\zz=\zz_\8$, $\psi=\psi_\8$ and $\overline\nz = \nz\cup\{+\8\}$. We will call the semigroup generated by $\psi_n$ the Poisson semigroup on $C^*_r(\zz_n)$ (or $L(\zz_n)$) for $n\in \overline\nz$. This corresponds to the semigroup generated by $(-\dd^2/\dd x^2)^{1/2}$ on $C(\tz)$ in Fourier analysis. A more natural operator to consider is $-\dd^2/\dd x^2$, the 1-dimensional Laplacian. The corresponding conditionally negative length function on $\zz$ is $\psi(k)=k^2$ for $k\in\zz$. On $\zz_n$ for $n\in \nz$, it is tempting to consider $\psi(k)=k^2$ for $|k|\le n/2$. (Note that here and in what follows we may replace $k$ by $k-n$ if $k>n/2$.) However, it is easy to check that this length function on $\zz_n$ for finite $n$ is not conditionally negative. Instead, we consider
\begin{equation}\label{cnl1}
\psi_n(k)=\frac{n^2}{2\pi^2}[1-\cos(\frac{2\pi k}{n})], \mbox{ for } k\in \zz_n, n\in\nz.
\end{equation}
One can check that $\psi_n$ defined in \eqref{cnl1} is conditionally negative by noting that $\exp(\frac{2\pi \ii \cdot}{n})$ is a positive semidefinite function on $\zz_n$. Since $\lim_{n\to \8}\psi_n(k)=k^2$ for any fixed $k$, we have
\begin{equation}\label{heatas}
  \psi_n(k)\sim k^2 \text{ for } |k|\le n/2.
\end{equation}
Here and in the following $a_k\sim b_k$ for two sequences $(a_k)$ and $(b_k)$ means that there exists an absolute constant $C\ge 1$ such that $C^{-1}\le a_k/b_k \le C$. We also define $\psi_\8(k)=k^2$ for $k\in\zz$ and call the semigroup generated by $\psi_n$ defined by \eqref{cnl1} the heat semigroup on ${C}^*_r(\zz_n)$ (or $L(\zz_n)$) for $n\in \overline\nz$. Once we know $\psi_n$ for $n\in\overline\nz$, we write $\Ga=\Ga^{\psi_\8}$ and $\Ga^n=\Ga^{\psi_n}$. We also denote by $\|\cdot\|_\8$ the supremum norm on both $C(\tz)$ and $C^*_r(\zz_n)$.

Let us now introduce the terminology and notation of compact quantum metric spaces. Our references here are \cites{Rie,Li06}. Given a unital $\mathrm{C}^*$-algebra $\mathscr{A}$, we denote by $\mathscr{A}_{sa}$ the set of self-adjoint elements in $\mathscr{A}$. Then $\mathscr{A}_{sa}$ is an order-unit space in the sense of \cite{Li06} with the identity of $\mathscr{A}$ as its order unit. Let $L$ be a (densely) defined Lip-norm on $\mathscr{A}_{sa}$ and let $\ax$ be a dense order-unit subspace of $\mathscr{A}_{sa}$ such that $L$ is finite on $\ax$. By definition, $(\ax,L)$ is a compact quantum metric space; see \cite{Li06}.  Let $S(\ax)$ denote the state space of $\ax$. For $\phi, \psi\in S(\ax)$, we define
\begin{align}\label{rhol}
\rho_{L}(\phi,\psi)=\sup\{|\phi(a)-\psi(a)|:L(a)\le 1, a\in \ax \}.
\end{align}
In general, $\rho_L$ can be defined on the state space of a C$^*$-algebra where a semi-norm $L$ is well-defined. The radius of $(S(\ax), \rho_L)$ is given by
$$ r_{\ax}=\frac12\sup\{\rho_L(\phi,\psi): \phi,\psi\in S(\ax)\}.
$$
For $r\ge0$, let
$$
\dx_r(\ax)=\{a\in \ax: L(a)\le 1, \|a\|\le r\}.
$$
Let $\bar\ax$ be the completion of $\ax$ for the order-unit norm inherited from $\mathscr{A}_{sa}$ so that $\bar\ax=\mathscr{A}_{sa}$ in our setting. Extend the Lip-norm $L$ to $\bar\ax$ by
\begin{align}\label{lbarx}
\bar L(x) :=\inf\{ \liminf_{n\to \8} L(x_n): x_n\in \ax, \lim_{n\to\8} x_n=x\}.
\end{align}
Clearly, $\bar L(x)$ may be infinite for $x\in\bar \ax$. Consider a subspace of $\bar\ax$ defined by
\[
\ax^c= \{x\in\bar \ax: \bar L(x)<\8\}.
\]
The closure of $L$ is defined as the restriction of $\bar L$ to $\ax^c$ and is denoted by $L^c$. Then $L^c$ is a Lip-norm on $\ax^c$. 

\begin{rem}\label{r:domain}
In practice, there seems to be some freedom to specify what elements are in $\ax$ as long as $\ax$ contains a dense subspace of $\mathscr{A}_{sa}$ on which the Lip-norm $L$ is finite.  In this paper, we simply choose $\ax$ to be the order-unit space of finite linear combinations of (powers of) generators of $\mathscr{A}_{sa}$, which may be regarded as `smooth elements' of $\mathscr{A}_{sa}$ or polynomials of generators. More precisely, when $\mathscr{A}= C(\tz)$, we take $\ax=\cz(\zz)_{sa}$; and when $\mathscr{A}$ is the rotation algebra $\ax_\Theta$ with $d$ generators $u_1,...,u_d$, we choose $\ax$ to be the self-adjoint finite linear combinations of $u_1^{k_1}\cdots u_d^{k_d}$, $(k_1,...,k_d)\in \zz^d$.

In this way, the gradient form $\Ga$ is always well-defined on $\ax$ whatever semigroup we use. Therefore, our argument of the approximation result is unified for different choices of length functions.
\end{rem}

\subsection*{Notation} For a (separable) Hilbert space $H$, we write $H^c$ and $H^r$ for its associated column and row operator space, respectively. We denote by $S_p^m$ (resp. $S_p$) the Schatten $p$ class on $\ell_2^m$ (resp. $\ell_2$). So $S_\8^m = M_m$ is the algebra of $m\times m$ matrices. We use $\otimes_{\min}$ to denote the spatial tensor product of operator spaces. The same notation is also used for the minimal tensor product of C$^*$-algebras. We usually omit the subscript if one of the C$^*$-algebras is nuclear so that there is only one tensor product; see \cite{BO08}*{Chapter 3}. In particular, $M_m(A)=M_m\otimes A = M_m\otimes_{\min} A$ for a C$^*$-algebra $A$. In addition, $M_m\otimes B$ or $M_m(B)$ will be used to denote the algebraic tensor product when $B$ is an algebra (but not a C$^*$-algebra). It should be clear from context in which category we are working.

We use the convention that $n\in \overline\nz$ means $n\ge 2$ or $n=\8$, as $\zz_1$ is trivial.

\section{Some analytic estimates}\label{anaest}

In this section we collect some analytic estimates which we will need later. Let us define
\[
L_p^0(\nx) = \{f\in L_p(\nx):\lim_{t\to \8} T_t f=0\}
\]
for $1\le p\le \8$. Here the limit is taken in $\|\cdot \|_p$ for $1\le p<\8$ and in the weak* topology for $p=\8$. For $x\in L_p(\nx)$ we write the mean-zero part of $x$ as $\mathring{x}= x - \lim_{t\to \8} T_t x$.  Equivalently, $\mathring{x}=x-E_{\rm Fix} x$, where ${\rm Fix} =\{f\in \nx: T_t f = f, \forall t\ge 0\}$ is the fixed point von Neumann subalgebra of $\nx$ and $E_{\rm Fix}: \nx\to {\rm Fix}$ is the conditional expectation which extends to a complete contraction on $L_p(\nx)$. See e.g. \cites{JX07, JLMX06} for these facts. Following \cite{JM10}, we define the (mean-zero) Lorentz spaces $L_{r,s}^0(\nx) = [L_p^0(\nx), L_q^0(\nx)]_{\ta,s}$, where $\frac1r = \frac{1-\ta}p +\frac{\ta}q$. See e.g. \cite{BL76,PX03} for the interpolation spaces. Note that in our case for the generator $A$ of the semigroup $(T_t)_{t \geq 0}$, we have Ker$(A^{1/2}) = \cz1$, which is equivalent to the ergodicity of the semigroup.

\begin{prop}\label{albd}
Let $T_t=e^{-tA}$ be a noncommutative symmetric Markov semigroup on $(\nx,\tau)$. Suppose
\begin{equation}\label{rm}
\|T_t: L_1^0(\nx,\tau)\to L_\8(\nx,\tau)\|_{\cb}\le Ct^{-m/2} \ \ \text{for all } t>0.
\end{equation}
Then $\| A^{-\al}: L_p^0(\nx,\tau)\to L_\8^0(\nx)\|_{\cb}\le C(m,\al)$ for $\al>\frac{m}{2p}$, where $C(m,\al)<\8$ only depends on $m$ and $\al$.
\end{prop}
\begin{proof}
The argument modifies from \cite{JM10}; see also \cite{JZ12}*{Corollary 4.22}. Let $\al=\frac{m}{2s}$. The argument in \cite{JM10}*{Lemma 1.1.3} can be trivially generalized to prove the complete boundedness. Hence, we have
\[
\|A^{-\al}: L_{s,1}^0(\nx) \to L_\8(\nx)\|_{\cb} \le C(m,\al).
\]
We know from the interpolation theory that $L_p^0(\nx) \hookrightarrow L_{s,1}^0(\nx)$ if $p>s$. The assertion follows.
\end{proof}

Let us consider the rotation $\mathrm{C}^*$-algebra $\ax_\Theta$, where $\Theta=(\theta_{ij})$ is a $d\times d$ skew symmetric matrix with $|\ta_{ij}|\in [0,1)$. By definition, $\ax_\Theta$ is the universal $\mathrm{C}^*$-algebra generated by unitaries $u_1,...,u_d$ with the commutation relations
\[
u_k u_l = e^{2\pi \ii \ta_{kl}} u_l u_k, \quad k,l=1,..., d.
\]
It is well known that $\ax_\Theta$ admits a faithful canonical tracial state $\tau$ such that $\tau(u_1^{k_1}\cdots u_d^{k_d}) = 1$ if and only if $k_1=\cdots=k_d=0$; see e.g. \cite{Ri90}. In order to work with noncommutative $L_p$ spaces of von Neumann algebras, we recall that $\rx_\Theta = \ax_\Theta''$ is the rotation von Neumann algebra associated to $\Theta$, which is the weak* closure of $\ax_\Theta$ acting on the GNS Hilbert space $L_2(\ax_\Theta,\tau)$. The linear combinations of $u_1^{k_1}\cdots u_d^{k_d}$ form a weakly dense subspace of $\rx_\Theta$. We will frequently use the following $^*$-homomorphism:
\begin{align}\label{comult}
\pi: \rx_\Theta\to L(\zz^d)\overline\otimes \rx_\Theta, \quad \pi(u_1^{k_1}\cdots u_d^{k_d}) = e^{2\pi \ii \lge \vec k,\cdot\rge} u_1^{k_1}\cdots u_d^{k_d}.
\end{align}
Note that $\pi$ is trace-preserving. Let $\psi$ be a conditionally negative length function on $\zz^d$ and $\td T_t$ the semigroup on $L(\zz^d)$ generated by $\psi$. We define a semigroup on $\rx_\Theta$ by
\[
T_t (u_1^{k_1}\cdots u_d^{k_d})= e^{-t\psi(k_1,...,k_d)}u_1^{k_1}\cdots u_d^{k_d}.
 \]
Then $(\td T_t\otimes \id) \circ \pi = \pi\circ  T_t$. Thanks to Schoenberg's Theorem, $T_t$ is a completely positive map for each $t\ge0$. We see that $T_t$ is a noncommutative symmetric Markov semigroup on $\rx_\Theta$; see also \cite{JZ12}*{Proposition 5.10}.

\begin{cor}\label{cbsob}
Let $A$ be the infinitesimal generator of $T_t$ defined as above. Assume that there exist $c\in(0,1],D>0$ and $\ga\ge 0$ such that $\psi(\zz^d)\setminus\{0\}\subset \cup_{j=1}^\8 [cj, c^{-1} j]$ and
\[
\#\{\vec{k}\in\zz^d: c j  \le \psi(\vec{k}) \le c^{-1} j\} \le Dj^\gamma, \text{ for all } j\in \zz_{>0}.
\]
Then
\[
  \| T_t: L_1^0(\rx_\Theta) \to L_\8(\rx_\Theta) \|_{\cb} \le Ct^{-(\ga+1)},
  \]
where $C$ only depends on $c, D$ and $\ga$.
Therefore, $A^{-\al}: L_p^0(\rx_\Theta)\to L_\8^0(\rx_\Theta)$ is completely bounded for $\al>\frac{\ga+1}{p}$. In particular, if $\psi(\vec{k}) \sim |k_1|+\cdots+ |k_d|$, we can choose $\ga = d$; and if $\psi(\vec{k}) \sim |k_1|^2+\cdots+ |k_d|^2$, we can choose $\ga=\frac{d+1}2$.
\end{cor}
\begin{proof}
Let $x=\sum_{\psi(\vec{k})> 0} a_{\vec{k}}\otimes u_1^{k_1}\cdots u_d^{k_d}\in M_m(L_1^0(\rx_\Theta))$ be a finite linear combination. Then $(\id\otimes T_t) (x)=\sum_{\vec{k}} e^{-t\psi(\vec{k})} a_{\vec{k}}\otimes u_1^{k_1}\cdots u_d^{k_d}$. Consider the linear functional
\[
\phi: L_1(\rx_\Theta) \to \cz,\quad
\phi(f)=\tau(f\cdot(u_1^{k_1}\cdots u_d^{k_d})^*).
\]
We have $\|\phi\|_{\cb} = \|\phi\|$ and thus $\|a_{\vec{k}}\|_{M_m}\le \|x\|_{M_m(L_1)}.$ It follows that
\begin{align*}
  \|(\id\otimes T_t) x\|_{M_m(\rx_\Theta)} &\le \sum_{\psi(\vec{k})> 0} \| a_{\vec{k}}\|_{M_m}e^{-t\psi(\vec{k})}\|u_1^{k_1}\cdots u_d^{k_d}\| {\le } \|x\|_{M_m(L_1)}\sum_{\psi(\vec{k})>0} e^{-t\psi(\vec{k})} \nonumber\\
  &\le D\|x\|_{M_m(L_1)} \sum_{j=1}^\8 j^\ga e^{-ct j}.\label{e:idttx}
\end{align*}
Consider the function $f(x)= x^\ga e^{-ctx}$, which attains its maximum at $x=\frac{\ga}{ct}$. Then
\[
\sum_{j=1}^\8 j^\ga e^{-ct j} \le \sum_{j=1}^{[\frac{\ga}{ct}]-1}j^\ga e^{-ct j} +2f(\frac{\ga}{ct}) +\sum_{j=[\frac{\ga}{ct}]+2}^{\8}j^\ga e^{-ct j} .
\]
For $t\le 1$, we consider $\sum_{j=1}^{[\frac{\ga}{ct}]-1}j^\ga e^{-ct j} $ (resp.~$\sum_{j=[\frac{\ga}{ct}]+2}^{\8}j^\ga e^{-ct j}$) as the left (resp.~right) endpoint approximation of the integral $\int_1^{[\frac{\ga}{ct}]} s^\ga e^{-cts}\dd s$ (resp.~$\int_{[\frac{\ga}{ct}]+1}^\8 s^\ga e^{-cts}\dd s$). It follows that
\[
\sum_{j=1}^\8 j^\ga e^{-ct j}\le 2\int_0^\8 s^\ga e^{-cts}\dd s+2(\frac{\ga}{ct})^\ga e^{-\ga}= \frac{2\Ga(\ga+1)}{(ct)^{\ga+1}}+\frac{2\ga^\ga e^{-\ga}}{(ct)^\ga}\le C_{\ga,c} t^{-\ga-1}
\]
where the constant $C_{\ga,c} =\frac{2\Ga(\ga+1)}{c^{\ga+1}} + \frac{2\ga^\ga}{(ce)^\ga}$ depends only on $\ga$ and $c$. For $t>1$ and $j\ge 1$, we note that $e^{-ctj/2}\le  e^{-c j/2}$, $e^{-ctj/2}\le e^{-ct/2}\le C'_{\ga,c} t^{-\ga-1}$ for some constant $C'_{\ga,c}$ and that the series $\sum_{j=1}^\8 j^\ga e^{-c j/2}$ only depends on $\ga$ and $c$. It follows that
\[
\sum_{j=1}^\8 j^\ga e^{-ct j}= \sum_{j=1}^\8 j^\ga e^{-ct j/2} e^{-ctj/2}\le e^{-ct/2}\sum_{j=1}^\8 j^\ga e^{-c j/2}\le C''_{\ga,c}t^{-\ga-1}
\]
for some constant $C''_{\ga,c}$ depending only on $\ga,c$. Therefore, we have proved that
\[
 \|(\id\otimes T_t) x\|_{M_m(\rx_\Theta)}  \le C \|x\|_{M_m(L_1)} t^{-(\ga+1)}
\]
for some constant $C$ depending only on $D,\ga,c$. Since the set of finite linear combinations of $u_1^{k_1}\cdots u_d^{k_d},\psi(\vec k)>0$ is dense in $L_1^0(\rx_\Theta)$, the first assertion follows by density. We deduce from Proposition \ref{albd} with $m=2(\ga+1)$ that $A^{-\al}: L_p^0(\rx_\Theta)\to L_\8(\rx_\Theta)$ is completely bounded for $\al>\frac{\ga+1}{p}$.

It remains to check the value of $\ga$. Since the set defined by $\{x\in \rz^d: \sum_{i=1}^d |x_i| \le j\}$ is convex, it is well known that (see e.g. \cite{Hor90}*{Theorem 7.7.16})
\[
\#\{\vec{k}\in \zz^d: |k_1|+\cdots + |k_d|=j\}\le D j^{d-1}.
\]
The same argument gives $\#\{\vec{k}\in \zz^d: k_1^2+\cdots + k_d^2=j\}\le D j^{\frac{d-1}2}$. For $\psi(\vec{k}) \sim |k_1|+\cdots+ |k_d|$, we have $\frac{\psi(\vec{k})}{\sum_{i=1}^d|k_i| }\in [c,c^{-1}]$ for some $c\in(0,1]$ and thus
\begin{align*}
\#\{\vec{k}:cj\le  \psi(\vec{k})\le c^{-1}j\}&\le \#\{\vec k\in\zz^d: c^2 j\le \sum_{i=1}^d |k_i|\le c^{-2}j \}\\
&\le D (c^{-2}-c^2)j   (c^{-2}j)^{d-1}\le D' j^d.
\end{align*}
Similarly, for $\psi(\vec{k}) \sim |k_1|^2+\cdots+ |k_d|^2$, we have
\[
\#\{\vec{k}:cj\le  \psi(\vec{k})\le c^{-1}j\} \le D (c^{-2}-c^2)j   (c^{-2}j)^{\frac{d-1}2}\le D' j^{\frac{d+1}2}.
\]
The proof is now complete.
\end{proof}

For notational convenience, let us introduce the following norms for $2\le p\le \8$. Let $\nx$ be a von Neumann algebra with a trace $\tau$ and $H$ a separable Hilbert space. We will always identify $H$ as $\ell_2$ by fixing an orthonormal basis $(e_i)_{i=1}^\8$. Recall from \cite{Pi03} that $H^c[p] =(H^c, H^r)_{1/p}$ and $H^r[p] =(H^r, H^c)_{1/p}$ with the interpolation notation. Let $(e_{ij})$ be the matrix units of $\bx(H)=\bx(\ell_2)$. We define $L_p(\nx, H^c[p])=L_p(\bx(H)\overline \otimes \nx) (e_{11}\otimes 1)$ to be the subspace of $L_p(\bx(H)\overline \otimes \nx )$ given by
\[
L_p(\nx, H^c[p]) = \{ y(e_{11}\otimes 1): y\in L_p(\bx(H)\overline \otimes \nx )\}.
\]
The norm of $L_p(\nx, H^c[p]) $ is inherited from $L_p(\bx(H)\overline \otimes \nx )$.
Equivalently, each $x\in L_p(\nx, H^c[p])$ can be written as a column vector $x=(x_1, x_2,...)^{\mathsf T}$ and the norm of $L_p(\nx, H^c[p]) $ is given by
\[
\|x\|_{L_p(\nx, H^c[p])} = \Big\| \Big(\sum_{i=1}^\8 x_i^* x_i\Big)^{1/2} \Big\|_{L_p(\nx,\tau)}.
\]
Similarly, we define $L_p(\nx, H^r[p]) = (e_{11}\otimes1)L_p(\bx(H)\overline \otimes \nx) $ to be the subspace of $L_p(\bx(H)\overline \otimes \nx)$ which consists of elements of the form $(e_{11}\otimes 1)y$ for $y\in L_p(\bx(H)\overline \otimes \nx)$. The inherited norm is given by
\[
\|x\|_{L_p(\nx, H^r[p])} = \|x^*\|_{L_p(\nx, H^c[p])}.
\]
The elements of $L_p(\nx, H^r[p])$ can be thought of as row vectors. Note that $L_\8(\nx, H^c)=H^c  \otimes_{\min}\nx $, $L_\8(\nx, H^r)= H^r \otimes_{\min} \nx$. We define $L_p(\nx, H^c[p]\cap H^r[p]) = L_p(\nx, H^c[p]) \cap L_p(\nx, H^r[p])$. Thus for $x\in L_p(\nx, H^c[p]\cap H^r[p])$, we may write $x=(x_1,x_2,...)$ and
\begin{align*}
\|x\|_{L_p(\nx, H^c[p]\cap H^r[p])} &= \max\{\|x\|_{L_p(\nx, H^c[p])}, \ \|x\|_{L_p(\nx, H^r[p])} \} \\
& = \max\Big\{ \Big\| \Big(\sum_{i=1}^\8 x_i^* x_i\Big)^{1/2} \Big\|_{L_p(\nx,\tau)} ,\  \Big\| \Big(\sum_{i=1}^\8 x_i x_i^*\Big)^{1/2} \Big\|_{L_p(\nx,\tau)} \Big\}.
\end{align*}
See \cite{PX97}*{Section 1} for an elementary discussion on these norms (where different notation was used). Let us turn to the group case. Let $\psi$ be a conditionally negative length function on $G$. Recall that $\psi$ determines a 1-cocycle $b: G\to H_\psi$ with values in a real unitary representation $(\al,H_\psi)$. Here $H_\psi$ is a real Hilbert space and $\lge b(g),b(h)\rge_{H_\psi} = K(g,h)$. One has
\[
b(gh)=b(g)+\al_g(b(h))\quad \text{and} \quad \psi(g)=\|b(g)\|^2,
\]
 for $g,h\in G$. See e.g. \cite{BO08}*{Page 468} (and also \cite{Ze13}*{Section 2.4} for an explicit construction). We define $\hx=H_\psi\otimes LG$ to be the right Hilbert $LG$-module with the $LG$-valued inner product
\[
\lge a\otimes x, b\otimes y \rge_{LG} = \lge a,b\rge_{H_\psi} x^*y
\]
and the right action $(\sum_{i} a_i\otimes x_i) y= \sum_{i} a_i\otimes x_i y$ for $x_i,y\in LG$ and $a_i\in H_\psi$; see \cite{Lan95}*{Page 5}.  We define a left $\cz(G)$-action on $\hx$ by
\[
\la(g)(b(h)\otimes \la(s)) = \al_g(b(h))\otimes \la(gs), \quad g,h,s\in G,
 \]
and extending linearly to $\cz(G)$. Let $\de: \cz(G)\to \hx$ be defined by
\begin{equation}\label{deri0}
\de(\la(g)) = b(g)\otimes \la(g).
\end{equation}
One can check that $\de$ is a derivation on $\cz(G)$. Moreover, we have
\[
\Ga(x,y)=\lge \de(x),\de(y)\rge_{LG}
\]
for $x,y\in \cz(G)$. One can naturally extend $\de$ to $M_m(\cz(G))$ by defining $\de(a_g\otimes \la(g))= b(g)\otimes a_g\otimes\la(g)$ for $a_g\in M_m$. In terms of Hilbert C$^*$-modules, we may think of $\de$ taking values in $ H_\psi\otimes M_m(LG)$. Extending the semigroup generated by $\psi$ to matrix levels, we can define the gradient form $\Ga$ on $M_m(\cz(G))$. Then we have
\[
\|\Ga(x,x)^{1/2}\|_{L_p(M_m(LG))} = \|\de(x)\|_{L_p(M_m(LG), H_\psi^c[p])}
\]
for $x\in M_m(\cz(G))$.

\begin{rem}
The domain of $\de$ (and $\Ga$) is a delicate issue even in the commutative theory. In consistence with Remark \ref{r:domain}, it suffices to restrict the domain of $\de$ to $\cz(G)$, which is dense in all $L_p(LG)$ for $1\le p<\8$. When $G=\zz$ and $\psi(k)=k^2$, this corresponds to restricting the domain of the (densely-defined) differentiation operator on $L_2(\tz)$ to the subspace of trigonometric polynomials, which is dense in all $L_p(\tz)$ for $1\le p<\8$.
\end{rem}

Note that $L_p(M_m(LG)) = S_p^mL_p(LG)$. For our later c.b. estimates of the Riesz transform, we wish to completely embed $L_\8(LG,H_\psi^c)$ into $L_p(LG, H_\psi^c[p])$. To this end, we have to consider $H_\psi^c\cap H_\psi^r$ and $H_\psi^c[p]\cap H_\psi^r[p]$.

\begin{lemma}\label{gahr}
  If $G$ is abelian, then $ \|\Ga(x^*,x^*)^{1/2}\|_{L_p(M_m(LG))}=\|\de(x)\|_{L_p(M_m(LG), H_\psi^r[p])}$ for $x\in M_m(\cz(G))$.
\end{lemma}
\begin{proof}
Let $x=\sum_{g}a_g\otimes \la(g)$ where $a_g\in M_m$. We define a linear map
\[
J: H_\psi\to H_\psi,\quad J(b(g)) =b(g^{-1}).
\]
Then thanks to commutativity,
\begin{align*}
  \lge b(g), b(h)\rge &= K(g,h)=\frac12[\psi(g)+\psi(h)-\psi(g^{-1}h)]\\
  &=\frac12[\psi(g^{-1})+\psi(h^{-1})-\psi(gh^{-1})] = \lge b(g^{-1}), b(h^{-1})\rge.
\end{align*}
Namely, $J$ preserves the inner product of $H_\psi$. Note that
\[
\de(x^*)^*= \Big(\sum_{g} b(g^{-1})\otimes a_g^*\otimes  \la(g^{-1}) \Big)^*= \sum_{g} b(g^{-1})^*\otimes  a_g\otimes \la(g)=( J\otimes \id\otimes\id) \de(x).
\]
Here we used $b(g^{-1})^*$ to specify that we view $b(g^{-1})$ as a row vector.
Since $J$ is an isometry, we have
\begin{align*}
  \| &\Ga(x^*, x^*)^{1/2}\|_{L_p(M_m(LG))} = \|\de(x^*)\|_{L_p(M_m(LG), H_\psi^c[p])} = \|\de(x^*)^*\|_{L_p(M_m(LG), H_\psi^r[p])}\\
  &=\|( J\otimes \id\otimes\id) \de(x)\|_{L_p(M_m(LG), H_\psi^r[p])} = \| \de(x)\|_{L_p(M_m(LG), H_\psi^r[p])}.\qedhere
\end{align*}
\end{proof}

Let us return to the rotation von Neumann algebra $\rx_\Theta$. Let $\ax_\Theta^\8$ denote the subalgebra of $\rx_\Theta$ which consists of finite linear combinations of $u_1^{k_1}\cdots u_d^{k_d}, (k_1,...,k_d)\in\zz^d$. Recall the homomorphism $\pi$ as defined in \eqref{comult}. Let $\de:\cz( \zz^d) \to H_\psi\otimes L(\zz^d)$ be the derivation given in \eqref{deri0}. Considering $(\id\otimes\de)\circ \pi$, we extend the derivation $\de$ to $M_m(\ax^\8_\Theta)$ by
\begin{align}\label{derext}
  \de(a_{\vec{k}}\otimes u_1^{k_1}\cdots u_d^{k_d}) =  b(\vec{k})\otimes a_{\vec{k}} \otimes u_1^{k_1}\cdots u_d^{k_d} .
\end{align}
By definition, the left action of $\ax^\8_\Theta$ on the Hilbert $\rx_\Theta$-module $H_\psi\otimes \rx_\Theta$ is given by
\begin{align}\label{leftact}
(u_1^{i_1}\cdots u_d^{i_d}) \de( u_1^{k_1}\cdots u_d^{k_d})  = \al_{(i_1,...,i_d)} (b(k_1,...,k_d)) \otimes (u_1^{i_1}\cdots u_d^{i_d})(u_1^{k_1}\cdots u_d^{k_d})
\end{align}
and the right action is $\de( u_1^{k_1}\cdots u_d^{k_d}) (u_1^{i_1}\cdots u_d^{i_d})=b(k_1,...,k_d) \otimes (u_1^{k_1}\cdots u_d^{k_d})(u_1^{i_1}\cdots u_d^{i_d})$. Note that the derivation is constructed so that the following diagram commutes at the matrix levels:
\[\xymatrixcolsep{5pc}
\xymatrix{
\ax^\8_\Theta \ar@{.>}[d]^\de \ar[r]^\pi &\cz(\zz^d) \otimes \ax^\8_\Theta\ar[d]^{\de\otimes\id  }\\
H_\psi \otimes \rx_\Theta   \ar[r]^{\id \otimes \pi}          & H_\psi \otimes L(\zz^d)  \otimes\rx_\Theta }
\]
Extending $T_t$ to $\id_{M_m} \otimes T_t$ on $M_m(\rx_\Theta)$, we can define the gradient form $\Ga$ on $M_m(\ax^\8_\Theta)$ associated to the generator $\id_{M_m}\otimes A$. Then we have $\Ga(x,y) = \lge \de(x), \de(y)\rge_{M_m(\rx_\Theta)}$ for $x,y\in M_m(\ax_\Theta^\8)$, where $\lge\cdot, \cdot \rge_{M_m(\rx_\Theta)}$ is the $M_m(\rx_\Theta)$-valued inner product of the Hilbert $M_m(\rx_\Theta)$-module. It follows that
\begin{align}\label{e:gade}
\|\Ga(x,x)^{1/2}\|_{L_p(M_m(\rx_\Theta))} =\|\lge \de(x),\de(x)\rge^{1/2}_{M_m(\rx_\Theta)}\|_{L_p(M_m(\rx_\Theta))} = \|\de(x)\|_{L_p(M_m(\rx_\Theta), H_\psi^c[p])}
\end{align}
for $x\in M_m(\ax_\Theta^\8)$. Since $\Ga(x,x)\in M_m(\ax_\Theta)$ for $x\in M_m(\ax_\Theta^\8)$, we may write $\|\Ga(x,x)^{1/2}\|_{M_m(\ax_\Theta)} = \|\Ga(x,x)^{1/2}\|_{M_m(\rx_{\Theta})}$ for any $m\in\nz$. Using similar argument to that of Lemma \ref{gahr}, we have the following result.

\begin{lemma}\label{gahr2}
Let $x=\sum_{\vec{k}\in\zz^d} a_{\vec{k}}\otimes u_1^{k_1}\cdots u_d^{k_d}$ be a finite sum where $a_{\vec{k}}\in M_m$. Then
\[
\|\Ga(x^*,x^*)^{1/2}\|_{L_p(M_m(\rx_\Theta))} = \|\de(x)\|_{L_p(M_m(\rx_\Theta),H_\psi^r[p])}.
\]
\end{lemma}
\begin{proof}
Observing \eqref{derext}, we may define for clarity,
\begin{align*}
\de^c(u_1^{k_1}\cdots u_d^{k_d}) &= b(\vec{k})\otimes u_1^{k_1}\cdots u_d^{k_d}\in H_\psi^c \otimes \rx_\Theta, \\
\de^r(u_1^{k_1}\cdots u_d^{k_d}) &= b(\vec{k})\otimes u_1^{k_1}\cdots u_d^{k_d}\in H_\psi^r \otimes \rx_\Theta.
\end{align*}
As in \eqref{derext}, we may extend $\de^c$ and $\de^r$ to matrix levels. Then
$$
\de^c(x^*) = \sum_{\vec{k}} b(-\vec{k})\otimes a_{\vec k}^* \otimes (u_1^{k_1}\cdots u_d^{k_d})^*.
$$
Since $\lge b(-\vec k),b(-\vec{k'})\rge_{H_\psi} = \lge b(\vec k),b(\vec{k'})\rge_{H_\psi}$, we have
\begin{align*}
  &\quad \|\Ga(x^*,x^*)^{1/2}\|_{L_p(M_m(\rx_\Theta))} = \|\de^c(x^*)\|_{L_p(M_m(\rx_\Theta), H_\psi^c[p])}  \\
  &= \|\sum_{\vec{k},\vec{k'}} \lge b(-\vec k),b(-\vec{k'})\rge_{H_\psi} a_{\vec k} a_{\vec{k'}}^* \otimes (u_1^{k_1}\cdots u_d^{k_d}) (u_1^{k'_1}\cdots u_d^{k'_d})^* \|_{p/2}^{1/2}\\
  &= \|\sum_{\vec{k},\vec{k'}} \lge b(\vec k),b(\vec{k'})\rge_{H_\psi} a_{\vec k} a_{\vec{k'}}^* \otimes (u_1^{k_1}\cdots u_d^{k_d}) (u_1^{k'_1}\cdots u_d^{k'_d})^* \|_{p/2}^{1/2}\\
  &= \|\de^r(x)\|_{L_p(M_m(\rx_\Theta), H_\psi^r[p])}.\qedhere
\end{align*}
\end{proof}

Let us introduce more notations to formulate our complete embedding results. For $2\le p\le \8$, we consider the semi-norm defined by
\[
\|x\|_{\nabla_p(\rx_\Theta)} = \|\de(x)\|_{L_p(\rx_\Theta, H_\psi^c[p] \cap H_\psi^r[p])}, \quad x\in \ax_\Theta^\8 \cap L_p^0(\rx_\Theta),
\]
and let $\nabla_p(\rx_\Theta)$ denote the completion of $\ax_\Theta^\8 \cap L_p^0(\rx_\Theta)$   with respect to $\|\cdot\|_{\nabla_p(\rx_\Theta)}$. In particular, the elements of $\nabla_p(\rx_\Theta)$ are mean-zero. Then by Lemma \ref{gahr2} we have
\begin{equation}\label{gahr2-2}
\|x\|_{S_p^m(\nabla_p(\rx_\Theta))} =  \max\{\|\Gamma(x,x)^{1/2}\|_p, \|\Gamma(x^*,x^*)^{1/2}\|_p\}
\end{equation}
for any $x$ in  $M_m(\ax_\Theta^\8)\cap M_m(L_p^0(\rx_\Theta))$.

For notational convenience, let us define for a finite sum $x=\sum_{k} a_{k}\otimes x_k$ in $M_m\otimes \ax_\Theta^\8$,
\begin{align}\label{lipmat}
\opnorm{x}_m= \max\{\|\de^c(x)\|_{M_m\otimes_{\min} \rx_\Theta\otimes_{\min} H_\psi^c}, \|\de^r(x)\|_{M_m\otimes_{\min} \rx_\Theta\otimes_{\min} H_\psi^r}\} = \Big\|\sum_{k} a_k\otimes \mathring{x_k}\Big\|_{M_m(\nabla_\8(\rx_\Theta))},
\end{align}
where $\mathring{x_k}$ is the mean-zero part of $x_k$ defined in the beginning of Section \ref{anaest}. Then $\opnorm{x}_1$ is a Lip-norm; see \cites{JM10,JMP10}. But $ \opnorm{x}_m$ vanishes on $M_m\otimes 1$ as well so that it is not a Lip-norm for $m\ge 2$; see Section \ref{s:prop} for more discussion about this. We usually ignore the subscript $m$ and write $\opnorm{x}$ if the underlying space is clear from context. We will also use frequently the notation $L(x):=\opnorm{x}$, especially when we consider a continuous field of compact quantum metric spaces.

\begin{prop}\label{dep8}
With the notation above, we have $\|\id: \nabla_\8(\rx_\Theta) \to \nabla_p(\rx_\Theta)\|_{\cb} \leq C_p$ for some constant $C_p$.
\end{prop}

\begin{proof}
Writing c.c. and c.b. for completely contractive and completely bounded isomorphisms, respectively, we consider the following diagram:
\[
\xymatrix{
\nx\otimes_{\min}(H_\psi^c\cap H_\psi^r) \ar@{.>}[d] \ar@{^{(}->}[r]^{\text{c.b.}} &\nx\otimes_{\min}L(\fz_\8)\ar@{^{(}->}[r]^{\text{c.c.}} &\nx\overline \otimes L(\fz_\8)\ar[d]^{\text{c.c. }}\\
L_p(\nx, H_\psi^c[p] \cap H_\psi^r[p]) \ar@{^{(}->}[rr]^{\text{c.b.}}& & L_p(\nx \overline \otimes L(\fz_\8))}
\]
Here $\nx$ can be any finite von Neumann algebra. In particular we take $\nx=\rx_\Theta$. From \cite{Pi03}*{Theorem 9.7.1}, we know that $H_\psi^c\cap H_\psi^r\hookrightarrow L(\fz_\8)$ completely isomorphically and the first line of the diagram follows. Also, by Corollaries 9.7.2 and 9.8.8 in \cite{Pi03}, $H_\psi^c[p] \cap H_\psi^r[p]$ completely embeds into $L_p(L(\fz_\8))$ and the second line of the diagram follows. {But $\nx\overline \otimes L(\fz_\8)\hookrightarrow L_p(\nx\overline \otimes L(\fz_\8))$ is completely contractive.} We deduce that there is a complete contraction from $\nx\otimes_{\min}(H_\psi^c\cap H_\psi^r)$ to $L_p(\nx, H_\psi^c[p] \cap H_\psi^r[p])$. Combining this with the definition of $\nabla_p(\rx_\Theta)$, we find that $\nabla_\8(\rx_\Theta)$ completely embeds into $\nabla_p(\rx_\Theta)$.
\end{proof}

\begin{rem}\label{mndlipemb}
The above procedure works not only for $\nx=\rx_\Theta$, it also works for $\nx=M_{n^d}$, the $n^d\times n^d$ dimensional matrix algebra, by choosing $2d$ generators of $M_{n^d}$ with order $n$. To see this, we simply define the homomorphism $\pi$ as in \eqref{comult} and the derivation $\de$ as in \eqref{derext} using $L(\zz_n^d)$ instead of $L(\zz^d)$. The notation $\nabla_{p}(\nx)$ will be used to represent $\nabla_p(\rx_\Theta)$ or $\nabla_p(M_{n^d})$.
\end{rem}

Suppose the semigroup $T_t=e^{-tA}$ on $\nx$ satisfies $\Ga_2\ge 0$, where $\Ga_2(f,g)=\frac12[\Ga(Af, g)+\Ga(f,Ag)-A\Ga(f,g)]$. Then $\id_{M_m}\otimes T_t$ also satisfies $\Ga_2\ge 0$; see \cites{JM10, JZ12} for more detailed discussion on this condition. Hence, we deduce from \cite{JM10} the complete boundedness of Riesz transforms
\begin{equation}\label{cbriesz0}
  \|A^{1/2}: \nabla_p(\nx) \to L_p^0(\nx)\|_{\cb} \le K_p.
\end{equation}
Combining this with Proposition \ref{dep8}, we obtain the following crucial ingredient in our argument for approximation in cb Gromov--Hausdorff convergence. Recall that we may take $\nx=\rx_\Theta$ or $\nx=M_{n^d}$ as in Remark \ref{mndlipemb}.

\begin{cor}\label{ade8p}
Suppose $T_t$ satisfies $\Ga_2\ge 0$ on $\nx$. Then we have
\[
\|A^{1/2}: \nabla_\8 (\nx) \to L_p^0(\nx)\|_{\cb}\le C_p
\]
for some constant $C_p$.
\end{cor}

Recall that for a given function $\vph: G\to \cz$, the Fourier multiplier $T_\vph$ on $LG$ is defined by extending $T_\vph(\la(s))=\vph(s)\la(s)$ for $s\in G$. $\vph$ is called a Herz--Schur multiplier if $T_\vph$ is completely bounded; see e.g. \cite{BO08}.
\begin{lemma}\label{cbga}
Let $\vph$ be a Herz--Schur multiplier on $G$ and $\Ga$ be the gradient form associated to $\id\otimes A$ as defined in \eqref{gagro}. If $f\in M_m(\cz(G))$, then
  \[
  \|\Ga((\id\otimes T_\vph) f,(\id\otimes T_\vph) f)\|_{M_m(LG)}\le \|T_\vph\|_{\cb}^2 \|\Ga(f,f)\|_{M_m(LG)}.
  \]
 Moreover, if $\vph$ is a Herz--Schur multiplier on $\zz^d$, then for any finite sum $f=\sum_{\vec{k}} a_{\vec{k}}\otimes u_1^{k_1}\cdots u_d^{k_d}\in M_m(\rx_\Theta)$, we have
  \[
  \|\Ga((\id\otimes T_\vph) f,(\id\otimes T_\vph) f)\|_{M_m(\rx_\Theta)}\le \|T_\vph\|_{\cb}^2 \|\Ga(f,f)\|_{M_m(\rx_\Theta)}.
  \]
\end{lemma}
\begin{proof}
For $f=\sum_s a_s\otimes \la(s)$ in the domain of $\id\otimes A$, since the multiplier commutes with the generator $A$, we have
\begin{align*}
  \|\Ga((\id\otimes T_\vph) f, &(\id\otimes T_\vph) f)^{1/2}\|_{M_m(LG)}= \|\de[(\id\otimes T_\vph) f]\|_{L_\8(M_m(LG),H_\psi^c)} \\
  &= \|(\id_{M_m}\otimes\id_{H_\psi} \otimes T_\vph ) \de(f)\|_{L_\8(M_m(LG), H_\psi^c)}\le \|T_\vph\|_{\cb}\|\de(f)\|_{L_\8(M_m(LG),H_\psi^c)}.
\end{align*}
We get the first assertion. The `moreover' part follows the same argument using the trace-preserving $^*$-homomorphism given in \eqref{comult}.
\end{proof}

\begin{rem}\label{cbga2}
Similar to Remark \ref{mndlipemb}, by considering $G=\zz_n^d$ and using the homomorphism \eqref{comult}, we find that Lemma \ref{cbga} still holds if we replace  $\rx_{\Theta}$ by $M_{n^d}$. This shows that $T_\vph: (M_{n^d}, \opnorm{\cdot})\to  (M_{n^d}, \opnorm{\cdot})$ is completely bounded.
\end{rem}

\begin{lemma}\label{cbapg}
  Let $\psi: G\to \zz$ be a conditionally negative length function. Suppose $\psi$ has at most polynomial growth, i.e. $\# \{g\in G: \psi(g)=0\} <\8$ and for all $l\ge 1$, $\#\{g\in G: \psi(g)=l\}\le Dl^\ga$ for some constants $\ga$ and $D\ge1$. Then for any $\eps>0$ and $k\in\nz$, there exists a Herz--Schur multiplier $\vph_{k,\eps}$ and $m=m(k)>k$ such that
\begin{enumerate}
\item[(i)] $\|T_{\vph_{k,\eps}}\|_{\cb}\le 1+\eps$;
\item[(ii)] the image of $T_{\vph_{k,\eps}}$ is contained in ${\rm span }\{\la(g)\in G: \psi(g)\le m\}$;
\item[(iii)] $|\vph_{k,\eps}(g)-1|\le \eps$ for $\psi(g)\le k$;
\item[(iv)] there exists $\eps_0<\eps$ such that for any $r\in\nz$, $1\le p\le q\le \8$, $\eta\in(0,\eps_0)$ and $x=\sum_{g: \psi(g)\le k} a_g\otimes \la(g)\in S_q^r(L_p(LG))$,
\[
\|(\id\otimes T_{\vph_{k,\eta}})(x)-x\|_{S_q^r(L_q(LG))} \le \eps \|x\|_{S_q^r(L_p(LG))}.
\]
Therefore, if we define $P_k(\sum_{g\in G} \hat f_g \la(g)) = \sum_{g: \psi(g)\le k} \hat f_g \la(g)$, then $\|(T_{\varphi_{k,\eta}}-\id)P_k: L_p(LG) \to L_q(LG)\|_{\cb} \le \eps$ for $1\le p\le q \le \8$.
\end{enumerate}
\end{lemma}
\begin{proof}
Let us define
\begin{equation}\label{phial}
  \vph_\al (g) = e^{-\psi(g)/{\al}} 1_{[\psi(g)\le m]},\quad g\in G.
\end{equation}
We know from Schoenberg's theorem, $\phi_\al(g):=e^{- \psi(g)/{\al}}$ gives a completely positive Fourier multiplier $T_{\phi_\al}$ on $LG$.
We have $\|T_{\phi_\al}\|_{\cb}=\|T_{\phi_\al}(1)\|=1$. Given any $x=\sum_{g}a_g \otimes \la(g)\in S_q(L_p(LG))$, we claim that for $1\le p, q\le \8$,
\begin{equation}\label{axpqr}
\|a_g\|_{S_q^r} \le \|x\|_{S_q(L_p(LG))}.
\end{equation}
Indeed, similar to the argument of Corollary \ref{cbsob}, we define
\[
\varrho: L_p(LG)\to \cz,\quad y\mapsto \varrho(y) =\tau_G(y \la(g)^*).
\]
We have $\|\varrho\|_{\cb} = \|\varrho\|\le 1$. By \cite{Pi98}*{Lemma 1.7}, we also have for any $1\le q\le \8$,
\[
\|\varrho\|_{\cb}= \sup_r \|\id \otimes \varrho: S_q^r(L_p(LG)) \to S_q^r\|.
\]
Hence, we have
\[
\|a_g\|_{S_q^r} = \|\id\otimes \varrho(x)\|_{S_q^r}\le \|x\|_{S_q^r(L_p(LG))}.
\]

Using \eqref{axpqr} with $p=q=\8$, we have
\[
\|(\id\otimes T_{\phi_\al})(x)-(\id\otimes T_{\vph_\al})(x)\|_{M_r(LG)} \le \sum_{\psi(g)\ge m} \|a_g\|_{M_r} e^{-\psi(g)/{\al} }\le \eps \|x\|_{M_r(LG)}
\]
for $\al$ large enough and thus $\|T_{\vph_\al}\|_{\cb}\le 1+\eps$. Given $\eps, k$, we can choose $m>k$ and $\al$ large enough in \eqref{phial}, and define $\varphi_{k,\eps}=\varphi_\al$ such that
\begin{equation*}
|\vph_{k,\eps}(g)-1|\le {\eps} \quad \mbox{ for } \quad \psi(g)\le k <m,
\end{equation*}
and ${\supp}~ \vph_{k,\eps} \subset\{g\in G: \psi(g)\le m\}$. Clearly, the image of $T_{\vph_{k,\eps}}$ is contained in ${\rm span}\{\la(g): \psi(g)\le m\}$.  Let $S_k= |\psi^{-1}(0)|+1+2^\ga+\cdots+k^\ga$, where $|\psi^{-1}(0)|$ is the number of zeros of $\psi$, and let $\eps_0=\frac{\eps}{DS_k}$. Using \eqref{axpqr} again, we have for any $\eta\in(0,\eps_0)$ and $x=\sum_{g:\psi(g)\le k} a_g\otimes \la(g)\in S_q^r(L_q(LG))$,
\begin{equation*}
\|(\id\otimes T_{\vph_{k,\eta}})(x) -x\|_{S_q^r(L_q(LG))} \le \sum_{\psi(g)\le k} \|a_g\|_{S_q^r} |\vph_{k,\eta}(g)-1|\le \eps \|x\|_{S_q^r(L_p(LG))} .
\end{equation*}
This inequality implies the last assertion by using \cite{Pi03}*{Lemma 1.7} again.
\end{proof}

The target space $\zz$ of the length function $\psi$ in the above may be replaced by some other countable discrete set, for instance, when we consider the length function \eqref{cnl1}. In particular, the condition in Corollary \ref{cbsob}
\[
\#\{g\in G: c j  \le \psi(g) \le c^{-1} j\} \le Dj^\gamma, \text{ for all } j\in \zz_{>0}
\]
implies the polynomial growth condition $\#\{g\in G: \psi(g)=l\}\le Dl^\ga$ imposed in the above result.

To motivate our following discussion, let us fix a conditionally negative length function $\psi$ on $\zz_n$ for $n\in\overline\nz$. Let $A_n$ denote the generator of the semigroup associated to $\psi$. Recall the notation above for $2\le p\le \8$
\[
\nabla_p(L(\zz_n))=\{x\in L^0_p(L(\zz_n)): \max\{\|\Gamma^n(x,x)^{1/2}\|_p, \|\Gamma^n(x^*,x^*)^{1/2}\|_p \} <\infty\}.
\]
Let $\frac{1}{2}=\alpha+\beta$ for some fixed $\alpha,\beta>0$.  For $2\le p<\8$, we consider the following chain of maps:
\[
\nabla_\8(L(\zz_n))\subset \nabla_p(L(\zz_n)) \xrightarrow{A_n^{1/2}} L_p^0(L\zz_n)\xrightarrow{A_n^{-\beta}} L_p^0(L\zz_n)\xrightarrow{A_n^{-\alpha}} L_{\infty}^0(L\zz_n).
\]
Note that by the boundedness of Riesz transform \eqref{cbriesz0}, $\|A_n^{1/2}:\nabla_p(L(\zz_n))\to L_p^0(L(\zz_n))\|\leq K_p.$
Suppose $A_n$ has a spectral gap. By \cite{JM10}*{Proposition 1.1.5}, we have
\begin{equation}\label{lpbd}
  \|A_n^{-\beta}: L_p^0\to L_p^0\|_{\cb}\le C_p.
\end{equation}
Using Proposition \ref{albd}, we can show that $A_n^{-\alpha}: L_p^0\to L_{\infty}^0$ is bounded for $p>1/\al$. Then
$$
\id=A_n^{-\alpha}\circ A_n^{-\beta} \circ A_n^{1/2}:\nabla_\8(L(\zz_n))\to L_{\infty}^0(L(\zz_n)).$$
It will become clear later that these maps will help us establish crucial norm estimates. In fact we can already draw a conclusion using these maps. Recall from \eqref{lipmat} the Lip-norm $L$ defined on $\ax^\8_\Theta$ and $M_{n^d}$. We understand $\zz_\8 = \zz$ and $M_{\8^d}=\rx_\Theta$ in the following result.

\begin{prop}\label{statebd}
Let $\nx$ be any of $L(\zz_n)$ or $M_{n^d}$, $n\in\overline\nz$, and $\psi$ be a conditionally negative length function which induces a symmetric Markov semigroup on $\nx$ as above. Assume $\psi$ has polynomial growth as in Lemma \ref{cbapg}. Then
\[
\|\id: \nabla_\8(\nx) \to L_\8^0(\nx)\|_{\cb}\le C
\]
for some $C$ independent of $n$.  Consequently, the radii of state spaces $(S(\nx), \rho_{L})$ are bounded uniformly in $n$.
\end{prop}
\begin{proof}
Since $\Ga_2\ge0$ for any conditionally negative length functions on discrete groups, by Corollary \ref{ade8p} we have $\|A^{1/2}: \nabla_\8(\nx)\to L_p^0(\nx)\|_{\cb} \le C_p$. Applying the same argument as that of Corollary \ref{cbsob}, we have $\|A^{-\al}: L_p^0(\nx)\to L_\8^0(\nx)\|_{\cb}\le C(\ga,\al)$ for $p>\frac{\ga+1}{\al}$ where $\ga$ is the exponent in the polynomial growth condition of $\psi$ as in Lemma \ref{cbapg}. Combining with \eqref{lpbd}, we have proved the first assertion. Furthermore, note that
\[
\rho_L(\phi,\phi') \le \sup_{x\in \nx, L(x)\le 1} |\phi(x)-\phi'(x)| \le \sup_{x\in\nx, \|x\|_{\nx}\le C} |\phi(x)-\phi'(x)| \le 2C.
\]
The second assertion follows.
\end{proof}

For $\Delta\subset \zz_n$, we define
\[
L_p^{\Delta} (L(\zz_n))= \{f\in L_p(L(\zz_n)):f=\sum_{k\in \Delta}\hat f(k)\la(k)\}.
\]
For $k\le n/2$ and $n\in \nz$, we define $\La_k = \{0, \pm 1,...,\pm k\}\subset \zz_n$ and $\La_k^c = \{\pm (k+1),...,\pm [\frac{n}2]\}\subset \zz_n$. For $n=\8$, we let $\La_k^c=\{j\in\zz: |j|>k\}$. Let us define the projection
\[
Q_{k}: L_p(L(\zz_n))\to L_p^{\La_k^c}(L(\zz_n)),\quad
Q_{k}\Big(\sum_{j}\hat{f}(j)\la(j)\Big)=\sum_{|j|>k}\hat{f}(j)\la(j).
\]
\begin{lemma}\label{prbd}
  For $1< p<\8$ and $n>2k$ or $n=\8$,
  \[
  \|Q_{k}: L_p(L(\zz_n))\to L_p^{\La_k^c}(L(\zz_n))\|_{\cb} \le C_p
  \]
  for some constant $C_p$ independent of $n,k$.
\end{lemma}
\begin{proof}
It is well known (see e.g. \cites{Bou, PX03}) that every projection $P: L_p(L\zz)\to L_p^{\Delta}(L\zz)$ is completely bounded for any subinterval $\Delta\subset\zz$. The case $n=\8$ follows. Assume $n\in \nz$. Let $tr$ denote the normalized trace on the $n\times n$ matrix algebra $M_n$. It is well known that there exists an injective trace-preserving $^*$-homomorphism $\rho: L(\zz_n)\to (M_n,tr)$ given by
\[
\la(j)\mapsto \left(
\begin{array}{cc}
0 & I_j\\
I_{n-j}&0
\end{array}
\right)
\]
where the first 1 in the first column appears in the $(j+1)^\text{st}$ row, the first 1 in the first row appears in the $(n-j+1)^\text{st}$ column, and the matrix entries are constant along diagonals. Fix $k$ and put
\begin{align*}
  B_1&=\{(i,j): i \ge k+2,j\le i-k\},\\
  B_2&=\{(i,j): j\ge 2,i\le j-1\},\\
  B_3&=\{(i,j): j\ge k+2,i\le j-k\}.
\end{align*}
Let $P_B$ denote the projection on $M_n$ given by
\[
P_B([a_{ij}]_{1\le i,j\le n})= \sum_{(i,j)\in B} a_{ij}\otimes e_{ij}
\]
where the $e_{ij}$ are the matrix units of $M_n$. Then $Q_{k}=P_{B_1}+P_{B_2}-P_{B_3}$. It is well known (see e.g. \cite{Bou}*{Corollary 19}, \cite{PX97}) that for any triangular projection $P_B$ and $1<p<\8$,
\[
\|P_B: S_p\to S_p\|_{\cb}\le C_p.
 \]
The assertion follows immediately.
\end{proof}

\begin{lemma}\label{tail}
  Let $2\le p<\8$. Then
  \[
  \|A_n^{-\bt}: L_p^{\La_k^c}(L(\zz_n))\to L_p(L(\zz_n))\|_{\cb}\le C_p \psi(k)^{-\bt/(p-1)}
  \]
  uniformly for $n>2k$ or $n=+\8$.
\end{lemma}
\begin{proof}
 Let $q=2p$ and $\frac1p=\frac{1-\ta}q+\frac{\ta}2$. Then $\ta=\frac1{p-1}$. By \eqref{lpbd} and Lemma \ref{prbd}, we have
 $$
 \|A_n^{-\bt}Q_{k}:L_q(L(\zz_n))\to L_q(L(\zz_n))\|_{\cb}\le C_p.
 $$
 Since $\|A_n^{-\bt}Q_{k}:L_2 (L(\zz_n))\to L_2(L(\zz_n))\|_{\cb}\le \psi(k)^{-\bt}$, by the Riesz--Thorin theorem we have
 \[
 \|A_n^{-\bt}Q_{k}: L_p(L(\zz_n))\to L_p(L(\zz_n))\|_{\cb}\le C_p \psi(k)^{-\bt\ta},
 \]
 which yields the assertion.
\end{proof}

\section{Approximation for $C(\tz)$}\label{ct}
Unless otherwise specified, in this section we consider the Poisson semigroups on $L(\zz_n)$ defined in Section \ref{s:cnl}; that is, the generator $A_n\la(k)=|k|\la(k)$ for $|k|\le n/2$. Following the notation of \cite{Li06}, for $n\in \nz$, we define
$$L_n(f)=\|\Ga^n(f,f)^{1/2}\|_\8 \text{ for } f\in C_r^*(\zz_n)_{sa}.
$$
We also write $\Ga:=\Ga^\8$ and $L(f):=L_\8(f)$ for $f\in \cz(\zz)$. It was proved in \cites{JM10, JMP10} that $L$ and $L_n$ are Lip-norms\footnote{There are different versions of definitions of compact quantum metric spaces. While $L_n$ defined here satisfies more conditions than the one in \cite{Li06}, our proof of convergence in the quantum Gromov--Hausdorff distance only requires the conditions listed in \cite{Li06}. In Section \ref{s:prop}, we will need the Lip-norms to have the so-called Leibniz property.}. Clearly, $L_n(f)<\8$ for $f\in C_r^*(\zz_n)_{sa}$ for $n\in \nz$. Note that $L$ is only defined in a dense subspace of $C(\tz)$. For simplicity of the presentation, we take $\ax_\8=\cz(\zz)_{sa}$ which can be identified with the self-adjoint trigonometric polynomials on $[0,1]$. We also write $\ax_n=C_r^*(\zz_n)_{sa}$ for short. Then $(\ax_n, L_n), n\in\overline\nz$ are compact quantum metric spaces in the sense of \cites{Ri04,Li06}. Our first task is to check that they form a continuous field of compact quantum metric spaces.

Define $\pi_n: C^*_r(\zz)\to C^*_r(\zz_n)$ to be the linear map sending $\la(k)$ to $\la(k ~\mathrm{ mod } ~n)$. Since $(\zz_n)$ are abelian, their universal $\mathrm{C}^*$-algebras coincide with the reduced $\mathrm{C}^*$-algebras and therefore $\pi_n$ is a $^*$-homomorphism extended from $\la(1)\mapsto \la(1)$ by universality. To describe $\pi_n$ in the function spaces, we have
\[
\pi_n: C(\tz)\to \ell_\8(n), \quad f\mapsto \pi_n(f)= (f(j/n))_{j=0}^n,
\]
and $\pi_n(e^{2\pi \ii k\cdot})(j) = e^{\frac{2\pi \ii kj}{n}}$.
\begin{lemma}\label{eqva}
  Let $f=\sum_{k=-m}^m a_k e^{2\pi \ii k\cdot}$ and $m\le n/2$. Then $\pi_n A f=A_n\pi_n f$. Therefore,
  \[
  \Ga^n(\pi_n f, \pi_n f)=\pi_n\Ga(f,f).
  \]
\end{lemma}
\begin{proof}
  Note that
\[
 \pi_n A f =\pi_n(\sum_{k=-m}^m a_k |k| e^{2\pi \ii k\cdot})= \sum_{k=-m}^m a_k |k| e^\frac{2\pi \ii k\cdot}{n}.
\]
Since $m\le n/2$, we get
\[
A_n\pi_n f =A_n(\sum_{k=-m}^m a_k e^\frac{2\pi \ii k\cdot}{n})= \sum_{k=-m}^m a_k |k| e^\frac{2\pi \ii k \cdot}{n}.
\]
Therefore, $\pi_n A f=A_n\pi_n f$. Now since $\pi_n$ is a $^*$-homomorphism, we have
\begin{align*}
\Ga^n(\pi_n f, \pi_n f) &= \frac{1}{2}(A_n(\pi_n f^*)\pi_nf+\pi_n f^* A_n(\pi_nf)-A_n(\pi_n f^* \pi_n f))\\
                                      &= \frac{1}{2}(\pi_n(Af^*f)+\pi_n(f^*Af)-\pi_n A(f^*f))\\
                                      &=\pi_n\Ga(f,f).\qedhere
\end{align*}
 \end{proof}

\begin{prop}\label{cont}
For any $m\in \nz$, there exists a constant $C_m$ depending only on $m$ such that
  \[
 0\le  \|f\|_\8 - \|\pi_n(f)\|_\8  \le \frac{C_m\|f\|_\8}{n},
  \]
  and
  \[
   0\le  \|\Ga(f,f)\|_\8 - \|\Ga^n(\pi_n f,\pi_n f)\|_\8 \le \frac{C_m\|\Ga(f,f)\|_\8 }{n},
  \]
for all $n>2m$ and all $f=\sum_{k=-m}^m \hat f(k) e^{2\pi \ii k\cdot}$. Consequently, for such $f$ we have
  \[
  \lim_{n\to\8} \|\pi_n f\|_\8= \|f\|_\8 \quad \text{and} \quad  \lim_{n\to\8} \|\Ga^n(\pi_n f,\pi_n f)\|_\8= \|\Ga(f,f)\|_\8.
  \]
\end{prop}
\begin{proof}
   By Lemma \ref{eqva}, when $n$ is large, $\Ga^n(\pi_n f, \pi_n f)=\pi_n\Ga(f,f)$. Let $h=\Ga(f,f)$. Note that since $f$ is a smooth function, so is $h$. By continuity of $h$, there exists $t_0\in [0,1]$ such that $\|h\|_\8=h(t_0)$. Let $j\in \nz$ be such that $|\frac{j}{n}-t_0|\le \frac{1}{2n}$. Using the mean value theorem, we get
  \[
0\le h(t_0) - h(\frac{j}{n}) \leq \|h'\|_\8|\frac{j}{n}-t_0|.
  \]
By \eqref{gagro}, we may assume $h=\sum_{k=-l}^l a_k e^{2\pi \ii k\cdot}$ for some finite $l$ which only depends on $m$. Then $h'(x) = \sum_{k=-l}^l 2\pi \ii k a_k e^{2\pi \ii kx}$ and thus
\[
 \sup_{x\in[0,1]}  |h'(x)| \le \sum_{k=-l}^l 2\pi |k| |a_k| \le C_l \|h\|_1 \le C_m \|h\|_\8,
\]
for some constant $C_m$ only depending on $m$. This proves that
\[
 0\le \|\Ga(f,f)\|_\8 - \|\Ga^n(\pi_n f,\pi_n f)\|_\8 \le \frac{C_m \|\Ga(f,f)\|_\8 }{n}.
\]
The first assertion follows similarly.
\end{proof}

Recall that  $\bar \ax_n =\ax_n=C^*_r(\zz_n)_{sa}$ for $n\in \nz$, $\ax_\8=\cz(\zz)_{sa}$ and $\bar \ax_\8 = C(\tz;\rz)$.  Following \cite{Dix77}*{Section 10.1}, let $\tilde\sx\subset\prod_{n\in\nz} \bar\ax_n$ be the (maximal) set of continuous sections of the continuous field of Banach spaces over $\overline\nz$  with fiber $\bar\ax_n$. Clearly, $1=(1_n)_{n\in\overline\nz}\in\tilde\sx$ where $1_n$ is the identity of $\ax_n$. Then $(\{\bar \ax_n\}_{n\in\overline\nz}, \tilde\sx)$ is a continuous field of order-unit spaces; see \cite{Li06}*{Definition 6.1}. By Proposition \ref{cont}, we have $(\pi_n(x))_{n\in \overline\nz}\in \tilde\sx$ for any $x\in \ax_\8$. Here we understand $\pi_\8=\id$. Moreover, $n\mapsto L_n(\pi_n(x))$ is continuous (and thus upper semi-continuous) at $n=\8$. Let $\sx=\{(\pi_n(x))_{n\in\overline \nz}:x\in \ax_\8 \}$. For simplicity, here we use the convention that $L_\8(x)=+\8$ for $x\in \bar\ax_\8 \setminus \ax_\8$. It follows that $\sx$ is a subset of
\[
\tilde\sx_{n_0}^L = \{f\in \tilde\sx: \text{ the function } n\mapsto L_n(f_n) \text{ is upper semi-continuous at }n_0 \text{ and } L_n(f_n)<\8 ~\forall n\} 
\]
for all $n_0\in \overline\nz$. Also, $\{\pi_n(x):x\in \ax_\8\} = \{f_n: f\in\sx\}$. To illustrate our idea, let us check $(\{\ax_n, L_n\}_{n\in\overline\nz}, \tilde\sx)$ is a continuous field of compact quantum metric spaces in  the sense of \cite{Li06}*{Definition 6.4}. Recall the notation $L^c$ and $\ax^c$ from Section \ref{s:cnl}. We need to show that $L_n$ restricted to $\{f_n: f\in \tilde\sx_{n}^L\}$ determines $L_n^c$, the closure of $L_n$. However, if $L_n$ restricted to $\{f_n: f\in\sx\}$ determines $L_n^c$, then since $\sx$ is a subset of $\tilde\sx_n^L$, $L_n$ restricted to $\{f_n: f\in \tilde\sx_{n}^L\}$ determines $L_n^c$ automatically. Therefore, it suffices to check $L_n$ restricted to $\{f_n:f\in\sx\}$ determines $L_n^c$. This will follow from the fact that $\{\pi_n(x): x\in\ax_\8\}=\ax_n$ is dense in $\bar \ax_n$ for $n\in\overline\nz$ thanks to our choice of $\ax_\8$. In fact, since $\bar \ax_n =\ax_n$ for $n\in\nz$ and for every $x\in\ax_\8$, $n\mapsto L_n(\pi_n(x))$ is continuous at $n=\8$, we only need to verify $L$ defined on $\ax_\8$ determines $L^c:=L_\8^c$.  But for any $x\in \ax_\8^c$, we have $L^c(x)<\8$.  Let $\eps>0$. By the definition of $L^c$ as in \eqref{lbarx}, there exists $x'\in \ax_\8$ such that $L(x')<L^c(x)+\eps$ and $\|x-x'\|_\8 < \eps$. Therefore, we have shown that $(\{\ax_n, L_n\}_{n\in\overline\nz}, \tilde\sx)$ is a continuous field of compact quantum metric spaces.

 In practice, it may be more convenient to work with a subset of $\tilde \sx$. Here we choose the set $\sx$ to be our set of continuous sections. Then $\sx_{n}^L=\sx$ for all $n\in\overline\nz$ by Proposition \ref{cont}. Due to the choice of $\sx$, if we want to prove $(\{\ax_n, L_n\}_{n\in\overline\nz}, \sx)$ is a continuous field of compact quantum metric spaces, we can directly go to the essence of the above argument: To check $L_n$ restricted to $\{f_n:f\in\sx\}$ determines $L_n^c$ for $n\in\overline\nz$ becuase $\sx_{n}^L=\sx$. From the above discussion, it should be clear that we could save some energy in minor technical issues by choosing $\ax_\8$ to be a dense subset of $C(\tz;\rz)$ where the Lip-norm is well-defined and choosing the continuous sections to be of the form $(\pi_n(x))_{n\in\overline \nz}, x\in\ax_\8$. In the following sections when we work with two and higher dimensions, we will make similar choices, which should not affect the essence of the argument if one makes different choices as explained above. Let us record what we have proved.

\begin{prop}\label{cfqms}
$(\{\ax_n, L_n\}_{n\in\overline\nz}, \sx)$ is a continuous field of compact quantum metric spaces.
\end{prop}

Our next goal is to show that $\ax_n$ converges to $\ax_\8$ in the quantum Gromov--Hausdorff distance. In light of Li's criterion \cite{Li06}, we need to find a `uniform' cover of $\dx_R(\ax_n)$ for $n$ large enough. We will achieve this by using the approximation properties of $\zz$ and going through various estimates in $L_p$ spaces. Recall that a Fourier multiplier $T_\phi$ on $L(\zz_n)$ is defined as
$$T_{\phi}(\sum_{j}a_j\la(j))=\sum_j a_j \phi(j)\la(j).$$

\begin{lemma}\label{cbapz}
Let $\eps>0$ and $k\in \nz$. Then there exist $m=m(k)>k$ and Herz--Schur multipliers $\vph_{k,\eps}^n$ on $\zz_n$ for $n> 2m$ (including $n=\8$) such that
\begin{enumerate}
\item[(i)] $\|T_{\vph_{k,\eps}^n}\|_{\cb}\le 1+\eps$;
\item[(ii)] the image of $T_{\vph_{k,\eps}^n}$ is contained in ${\rm span}\{\la(j): |j|\le m\}$;
\item[(iii)] $|\varphi_{k,\eps}(j) -1|\le \eps$ for $|j|_n\le k$;
\item[(iv)] for $x$ in ${\rm span}\{\la(j): |j|_n\le k\}$ and $\eta\in(0, \frac{\eps}{2(k+1)})$,
\begin{equation}\label{phimap}
\|T_{\vph_{k,\eta}^n}x-x\|_\8\le \eps \|x\|_2.
\end{equation}
\end{enumerate}
\end{lemma}
\begin{proof}
Note that $\#\{j\in \zz_n: |j|_n=k\}\le 2$ for $k\ge1$. Applying Lemma \ref{cbapg} first to $G=\zz$ (so we have $D=2,\ga=0$), we get $m$ and a multiplier $\vph_{k,\eps}$ on $\zz$. Then applying Lemma \ref{cbapg} again to $G=\zz_n$ for $n>2m$, we find multipliers $\vph_{k,\eps}^n$ on $\zz_n$, which satisfy $\vph_{k,\eps}^n(j) = \vph_{k,\eps}(j)$ for $|j|\le m$ because the proof of Lemma \ref{cbapg}  does not depend on $n$ once we choose $m$. The assertion follows by taking $p=2, q=\8, r=1$.
\end{proof}

\begin{lemma}\label{bdd}
Let $T_t^n=e^{-tA_n}$ be the Poisson semigroup associated with $\psi_n$ acting on $L(\zz_n)$ defined in Section \ref{s:cnl}. Then $A_n^{-\alpha}: L_p^0(L(\zz_n))\to L_{\infty}^0(L(\zz_n))$ is completely bounded uniformly in $n\in \overline \nz$ for $\alpha>\frac{1}{p}$.
\end{lemma}
\begin{proof}
  The argument is the same as that of Corollary \ref{cbsob} with $\ga=0$.
\end{proof}

\begin{lemma}\label{eqtai}
Let $\eps>0$. Then there exist $k=k(\eps), m=m(k)$ and Herz--Schur multipliers $\vph_{k,\eta}^n$, $\eta\in(0, \frac{\eps}{2(k+1)})$ on $\zz_n$ for $n> 2m$ (including $n=\8$) such that
\[
 \|x-T_{\vph_{k,\eta}^n}(x)\|_\8\le\eps [\|x\|_2+ L_n(x)]
\]
for $n>2m$ (including $n=\8$).
\end{lemma}
\begin{proof}
Let $k\in\nz$ be a large number which will be determined later. We choose $m$ and $\vph_{k,\eta}^n$ as in Lemma \ref{cbapz}. Since $\|(1-Q_k)x\|_2\le \|x\|_2$, by \eqref{phimap} we have
\[
\|(1-Q_k)(x-T_{\vph_{k,\eta}^n}(x)) \|_\8 = \|(1-Q_k)x - T_{\vph_{k,\eta}^n}((1-Q_k)x) \|_\8 \le \|(1-Q_k)x\|_2 \eps \le {\eps}\|x\|_2.
\]
Note that $Q_k$ and $A_n$ commute. Using Lemma \ref{bdd}, equation \eqref{lpbd}, Lemma \ref{tail} and the boundedness of Riesz transforms \cite{JM10}, we have for $p>1/\al$,
\begin{align*}
&\quad \|A_n^{-\al} A_n^{-\bt} A_n^{1/2} Q_k(x-T_{\vph_{k,\eta}^n}(x))\|_\8 \le c_\al \|A_n^{-\bt}Q_k A_n^{1/2} (x-T_{\vph_{k,\eta}^n}(x)) \|_p \label{qkbd}\\
&\le c_\al C_p k^{-\bt/(p-1)}\|A_n^{1/2}(x-T_{\vph_{k,\eta}^n}(x))\|_p ,\\
&\le c_\al K_p C_p k^{-\bt/(p-1)} (\|\Ga^n(x,x)^{1/2}\|_p+ \|\Ga^n(T_{\vph_{k,\eta}^n}(x),T_{\vph_{k,\eta}^n}(x))^{1/2}\|_p)
\end{align*}
where $c_\al = \|A_n^{-\al}: L_p^0(L(\zz_n)) \to L_\8(L(\zz_n))\|$, $K_p$ is the $L_p$ bound of Riesz transforms, and $C_p k^{-\bt/(p-1)}$ is the bound in Lemma \ref{tail}. By Lemma \ref{cbga}, we have
\[
\|Q_k(x-T_{\vph_{k,\eta}^n}(x)) \|_\8 \le (2+\eps)c_\al K_p C_p k^{-\bt/(p-1)} \|\Ga^n(x,x)^{1/2}\|_\8\le \eps L_n(x),
\]
by choosing $k$ large enough. The claim follows.
\end{proof}

\begin{prop}\label{assum}
Let $0<\eps<1$ and $R\ge 0$. There exist $N>0$ and $x_1,..., x_r$ in $\dx_R(\ax_\8)$ such that the open $\eps$-balls in $\ax_n$ centered at $\pi_n(x_1),..., \pi_n(x_r)$ cover $\dx_R(\ax_n)$ for all $n>N$ (including $n=\8$).
\end{prop}
\begin{proof}
The case $R=0$ is trivial. Assume $R>0$. Let $m$ and $\vph_{k,\eta}^n$ be given by Lemma \ref{eqtai}. For $n>2m$, let us define
\[
\dx_{R}^m(\ax_n)= \{x\in \dx_R(\ax_n): x= \sum_{|j|\le m} a_j \la(j) \}.
\]
Since $\dx_{R}^m(\ax_\8)$ is compact, we can find $x_1,...,x_r\in \dx_R^m(\ax_\8)$ such that for all $y\in \dx_R^m(\ax_\8)$ there exists an $s\in\{1,...,r\}$ with $\|y-x_s\|_\8\le \eps$, i.e. $\{x_1,...,x_r\}$ is an $\eps$-net of $\dx_R^m(\ax_\8)$. 

Let $n>2m, n\in \overline\nz,$ and $y^n\in \dx_R(\ax_n)$. We may write $T_{\vph_{k,\eta}^n}(y^n) = \sum_{|j|\le m} a_j e^{\frac{2\pi \ii j \cdot}{n}}$. Since the coefficients $(a_j)$ are uniquely determined by $y^n$ and $\vph_{k,\eta}^n$, we may define $\hat y = \sum_{|j|\le m} a_j e^{2\pi \ii j \cdot}$ in $ \ax_\8$. Then
\begin{equation}\label{piphi}
  \pi_n(\hat y) = T_{\vph_{k,\eta}^n}(y^n).
\end{equation}
But by choosing $n>\frac{2C_m}\eps$ in Proposition \ref{cont}, we have
\[
0\le \| \hat y\|_\8 - \|\pi_n(\hat y)\|_\8 \le \frac{\eps}2 \|\hat y\|_\8
\]
and
\[
0\le \|\Ga(\hat y,\hat y)\|_\8 -  \| \Ga^n(\pi_n(\hat y),\pi_n(\hat y)) \|_\8\le \frac{\eps}2 \|\Ga(\hat y,\hat y)\|_\8
\]
for any $\hat y$ in the set $\{\hat y: y^n\in\dx_R(\ax_n) \}$. It follows that
\begin{equation}\label{gammay}
  \|\hat y\|_\8 \le (1+\eps) \| \pi_n(\hat y) \|_\8\quad \mbox{and}\quad
 \|\Ga(\hat y,\hat y)\|_\8 \le (1+\eps) \| \Ga^n(\pi_n(\hat y),\pi_n(\hat y)) \|_\8
\end{equation}
for all $y^n\in \dx_R(\ax_n)$ and all $n>N := \max\{ [\frac{2C_m}{\eps}], 2m\}$. We deduce from \eqref{piphi} that
\[
\|\hat y \|_\8 \le (1+\eps) \|T_{\vph_{k,\eta}^n}(y^n)\|_\8 \le (1+\eps )^2 \|y^n\|_\8 \le (1+\eps )^2 R.
\]
By Lemma \ref{cbga},
\[
\|\Ga^n(T_{\vph_{k,\eta}^n}(y^n),T_{\vph_{k,\eta}^n}(y^n))\|_\8\le (1+\eps)^2 \|\Ga^n(y^n,y^n)\|_\8.
\]
Thus by \eqref{piphi} and \eqref{gammay}, $\| \Ga(\hat y,\hat y)\|_\8\le (1+\eps)^3$. We find $\frac{1}{(1+\eps)^2} \hat y\in \dx^m_R(\ax_\8)$. Hence there exists an $x_s$ in the $\eps$-net $(x_i)_{i=1,...,r}$ of $\dx^m_R(\ax_\8)$ such that $\|\frac1{(1+\eps)^2}\hat y - x_s\|_\8 \le \eps$. Then we deduce from \eqref{piphi} that
\begin{align*}
\|T_{\vph_{k,\eta}^n}(y^n)& -\pi_n(x_s)\|_\8 \le \| \hat y -x_s\|_\8 \\
&\le \|\hat y -\frac1{(1+\eps)^2}\hat y \|_\8+ \|\frac1{(1+\eps)^2}\hat y- x_s\|_\8 \le (3R+1) \eps.
\end{align*}
Using Lemma \ref{eqtai}, we have
\[
\|y^n-\pi_n(x_s)\|_\8\le \|y^n- T_{\vph_{k,\eta}^n}(y^n)\|_\8+\|T_{\vph_{k,\eta}^n}(y^n)-\pi_n(x_s)\|_\8 \le (4R+2) \eps.
\]
Replacing $\eps$ with $\frac{\eps}{4R+2}$ in the very beginning, we complete the proof.
\end{proof}

\begin{theorem}\label{ctap}
  $(\ax_n,L_n)$ converges to $(\ax_\8,L)$ in the quantum Gromov--Hausdorff distance.
\end{theorem}
\begin{proof}
By Proposition \ref{cfqms}, $(\{\ax_n, L_n\}_{n\in\overline\nz},\sx)$ is a continuous field of compact quantum metric spaces in the sense of \cite{Li06}. Let $\eps=1/m$. By Proposition \ref{assum}, we find
\[
x_{1}(m),\cdots, x_{r_m}(m)\in \dx_R(\ax_\8)
\]
such that for any $x\in \dx_R(\ax_\8)$ there exists $x_s(m)$ so that $\|x-x_s(m)\|\le 1/m$. Then the set
\[
\Lambda:=\cup_{m=1}^\8\{x_{1}(m),\cdots, x_{r_m}(m)\}
\]
is dense in $\dx_R(\ax_\8)$. Give an ordering on $\La$ as follows: $x_i(m)<x_j(m)$ if $i<j$ and $x_i(m)<x_j(m')$ if $m<m'$. Then $\La$ is totally ordered and we can list the elements of $\La$ according to this ordering. Identify $x\in\dx_R(\ax_\8)$ with a section $x=(\pi_n(x))_{n\in \overline\nz}$ such that $\pi_n(x)\in \ax_n$. By our construction, for any $\eps>0$, there exist $r$ and $N$ such that the open $\eps$-balls in $\ax_n$ centered at $\pi_n(x_1),\cdots,\pi_n(x_r)$ cover $\dx_R(\ax_n)$ for all $n>N$, where $x_i\in \Lambda$ for all $i$. In other words, $\Lambda$ satisfies Condition (iii) in \cite{Li06}*{Theorem 7.1}. Hence $\ax_n$ converges to $\ax_\8$ in ${\rm dist}^R_{\rm oq}$, the $R$-variant order-unit quantum Gromov--Hausdorff distance, by the same theorem. But by Proposition \ref{statebd}, the radii of state spaces of $\ax_n$ are uniformly bounded. The assertion follows from \cite{Li06}*{Theorem 1.1}.
\end{proof}

\begin{rem}
In Theorem \ref{ctap}, we used the Poisson semigroup on $L(\zz_n)$ to define the Lip-norm. In fact, the same approximation result remains true if we use the heat semigroup on $L(\zz_n)$ and the proof is slightly more direct. Indeed, thanks to \eqref{heatas}, we would get $m=1$ in \eqref{rm}, which allows us to choose $p=2$ and $\frac14<\al<\frac12$ to replace Lemma \ref{bdd}. Then certain $L_p$ estimates reduce to $L_2$ estimates. We leave this to the interested reader.
\end{rem}

\section{Matrix algebras converge to noncommutative tori} \label{A_theta}
Similar to the previous section, we need to define a Lip-norm and a semigroup action on $M_n$, and show that the family of matrix algebras together with these Lip-norms form a continuous field of compact quantum metric spaces. Now we have to introduce some notation. Let $n\in \nz$. Then $M_n\simeq \ell_{\8}(n)\rtimes_\al \zz_n=\{u_j(n), v_k(n):0\le j,k\le n-1\}''$,
where $u_j(n)$ is defined in \eqref{ujn} and $v_k(n)=\la_n(k)=\begin{pmatrix}
0& I_k\\
I_{n-k}& 0
\end{pmatrix}
$, the left regular representation of $\zz_n$. The action $\al$ is given by
\[
\al_k(u_j(n)) = v_k(n)^* u_j(n) v_k(n).
\]
Then we have the following relations
\[
u_j(n)e_p=e^{\frac{2\pi \ii jp}{n}}e_p \quad \text{and} \quad v_k(n)e_l=e_{k+l},
\]
where $\{e_j\}_{j=1}^n$ is the standard orthonormal basis of $\cz^n$. It follows that
\[
u_1(n)v_1(n)=e^{\frac{2\pi \ii}{n}}v_1(n)u_1(n).
\]
We expect that $u_j(n)$ and $v_k(n)$ commute in the limit. For convenience, we define $u_{-j}(n) = u_j(n)^*, v_{-k}(n) = v_k(n)^*$ for $0\le j,k\le n-1$. Clearly, $u_j(n)v_k(n)=e^{\frac{2\pi \ii jk}{n}}v_j(n)u_k(n)$ for all $1-n \le j,k\le n-1$.

\subsection{Norm estimates for trigonometric polynomials}
For $n\in\nz$, we define $T^n_t$ to be the semigroup acting on $M_n$ by $T^n_t(u_j(n)v_k(n))=e^{-t(\psi_n(j)+\psi_n(k))}u_j(n)v_k(n)$,
where $\psi_n$ is given by \eqref{cnl1},
\[
 \psi_n(k)=\frac{n^2}{2\pi^2}[1-\cos(\frac{2\pi k}{n})].
\]
Then by Schoenberg's Theorem $T_t^n$ is a completely positive map for each $t\ge0$. Note that $u_j(n) = [u_1(n)]^j$ and $v_k(n) = [v_1(n)]^k$. So here we are using $u_1(n), v_1(n)$ as the generators of $M_n$ when we define the semigroup $T_t^n$. In fact, as we shall see later, we may use any fixed pair of generators of $M_n$ or any prime powers of these generators as the generators of $M_n$, but we always define $T_t^n$ as if they were $u_1(n), v_1(n)$. For example, $u_{p}(n), v_q(n)$ also generate $M_n$ as long as $(pq, n)=1$; see e.g. \cite{Dav}. In this case, we may define
\begin{equation}\label{sgheatmn}
T_t^n([u_{p}(n)]^j [v_q(n)]^k) = e^{-t(\psi_n(j)+\psi_n(k))} [u_p(n)]^j [v_q(n)]^k.
\end{equation}
For simplicity, we may just write $u_1(n)$ and $ v_1(n)$ for $u_p(n)$ and $v_q(n)$, respectively, by abuse of notation. The semigroup we are using should be clear from context. Note that $\psi_n(j)+\psi_n(k)$ on $\zz_n^2$ is conditionally negative. Clearly,
\[
\psi_n(k)\sim \begin{cases} k^2 &\mbox{if } |k|\leq \frac{n}{2}, \\
(n-k)^2 & \mbox{if } |k|> \frac{n}{2}.
\end{cases}
\]
Let $u$ and $v$ be the generators of $M_\8:=\ax_\ta$. Intuitively, since $\lim_{n\to \8}\psi_n(k)=k^2=:\psi_\8(k)$, we would expect the heat semigroup in the limit
\begin{equation}\label{sgheatata}
 T_t(u^jv^k):=T^{\8}_t(u^jv^k)=e^{-t(|j|^2+|k|^2)}u^jv^k
\end{equation}
to act on $\ax_\ta$. We define the gradient form $\Ga^n$ associated to the generators
\[
A_n(u_j(n)v_k(n))=(\psi_n(j)+\psi_n(k))u_j(n)v_k(n)
\]
as in \eqref{carre} for $n\in\overline\nz$. Without loss of generality, from now on we always assume that $n$ is large enough and $|j|,|k|\leq n/2$. For $n\in \overline\nz$, we define $L_n(f)=\|\Gamma^n(f,f)^{1/2}\|_{\8}$. Write $\Ga:=\Ga^\8$ and $L(f):=L_{\8}(f)$. Note that $M_n\simeq C_r^*(\zz_n\rtimes_\al \zz_n)$ for $n\in\overline \nz$. It follows from \cites{JM10,JMP10} that $L_n$ and $L$ are Lip-norms on $M_n$ and $\ax_\ta$, respectively. Since the heat semigroup $T_t$ on $L(\zz_n)\rtimes_\al \zz_n$ is a symmetric Markov semigroup, the following result follows by the same argument as that of Corollary \ref{cbsob} with $\ga=\frac32$.

\begin{lemma}\label{ttbd1}
Let $A_n$ be the generator of the heat semigroup acting on $L(\zz_n)\rtimes_\al\zz_n$ defined as above. Then $A^{-\alpha}_n: L^0_p(M_n)\to L^0_{\8}(M_n)$ is completely bounded uniformly in $n\in\nz$ for $\al>\frac{5}{2p}$.
\end{lemma}

Similar to \eqref{comult}, for $n\in\nz$, we define a $^*$-homomorphism
\[
\pi: M_n\to\ell_\8(\zz_n^2) \otimes M_n ,\quad u_j(n)v_k(n)\mapsto  \la(j,k)\otimes u_j(n)v_k(n), \quad(j,k)\in \zz_n^2.
\]
Here $\la(j,k)$ is the left regular representation of $\zz_n^2$. We also define a $^*$-homomorphism for $0<\ta<1$
\[
\pi: \ax_\ta\to  L(\zz^2) \otimes \ax_\ta, \quad u_\ta^j v_\ta^k \mapsto \la(j,k)\otimes u_\ta^j v_\ta^k, \quad (j,k)\in \zz^2.
\]
Here $u_\ta,v_\ta$ are the generators of $\ax_\ta$. It is easy to check that $\pi$ is trace-preserving. If we understand $M_\8=\ax_\ta$ and $u_\ta^j=u_j(\8), v_\ta^k=v_k(\8)$, we can define the Fourier multipliers for $n\in\overline\nz$ by
\begin{equation}\label{tphi}
\td T_\phi(\la(j,k))=\phi(j,k) \la(j,k), \quad  T_\phi(u_j(n)v_k(n))=\phi(j,k) u_j(n)v_k(n).
\end{equation}
Note that $\pi\circ  T_\phi = (\td T_\phi \otimes \id ) \circ\pi$.
We immediately have the following useful co-representation transference technique.
\begin{lemma}\label{trans}
For any $n\in\overline \nz$ and $1\le p\le \8$, we have
\[
\| T_\phi: L_p(M_n)\to L_p(M_n)\|_{\cb} \le \|\td T_\phi: L_p(\zz_n^2)\to L_p(\zz_n^2)\|_{\cb}.
\]
\end{lemma}

Let us consider $\phi(j,k) = e^{-t\psi(j,k)}$ in \eqref{tphi} for a conditionally negative length function $\psi$ on $\zz_n^2$. For instance, we may take $\psi(j,k)=\psi_n(j)+\psi_n(k)$ on $\zz_n^2$ where $\psi_n$ is defined in \eqref{cnl1}. This gives a symmetric Markov semigroup on $M_n$, which coincides with the semigroup $T_t$ defined in \eqref{sgheatmn} and \eqref{sgheatata}. Again, let $\Ga$ denote the gradient form associated to $T_t$. For the development of next section, we may extend $T_t$ to $M_m\otimes_{\min} M_n$ by $\id_{M_m}\otimes T_t$ for any $m\in\nz$ even though we only need $m=1$ in this section. The following result is a special case of \eqref{cbriesz0}.

\begin{prop}\label{cbriesz}
Let $2\le p<\8$. For any $m\in\nz$, $a_{j,k}\in M_m$ and a finite sum $f=\sum_{j,k} a_{j,k}\otimes u_j(n)v_k(n)$, we have
\begin{align*}
\| &(\id_{M_m}\otimes A)^{1/2}(f)\|_{L_p(M_m(M_n))} \\
&\le C_p \max\{\|\Ga(f,f)^{1/2}\|_{L_p(M_m(M_n))}, ~\|\Ga(f^*,f^*)^{1/2}\|_{L_p(M_m(M_n))}\}
\end{align*}
where $C_p$ is independent of $m\in\nz$ and $n\in\overline \nz$. Therefore, $A^{1/2}: \nabla_\8(M_n) \to L_p^0(M_n)$ is completely bounded.
\end{prop}
\begin{proof}
The conditionally negative length function $\psi$ gives the positive semidefinite Gromov form $K$ on $\zz_n^2$. By the Schur product theorem, we know that $K\bullet K$ is also positive semidefinite, where $\bullet$ denotes the Schur product of matrices. It follows that $\Ga_2\ge 0$ on $L(\zz_n^2)$; see e.g. \cite{JZ12}. This transfers to $\Ga_2\ge0$ on $M_n$ by our definition of $T_t$ on $M_n$, which further extends to $M_m\otimes_{\min} M_n$. Now we can apply \eqref{cbriesz0} and then Corollary \ref{ade8p}.
\end{proof}

Let $Q^1_l, Q^2_l:L_p(M_n)\to L_p(M_n)$, $n>2l$, $n\in \overline\nz$, be the projections defined as
\[
Q^1_l(\sum_{j,k}a_{jk}u_j(n)v_k(n))=\sum_{\substack{|j|>l,\\ k}}a_{jk}u_j(n)v_k(n), \]
\[
Q^2_l(\sum_{j,k}a_{jk}u_j(n)v_k(n))=\sum_{\substack{|k|>l,\\ j}}a_{jk}u_j(n)v_k(n).
\]
Let $\Delta\subset \zz_n^2$. We define
\[
L^{\Delta}_p(M_n)=\{f\in L_p(M_n) : f=\sum_{(j,k)\in \Delta}a_{jk}u_j(n)v_k(n)\}.
\]
Let
\begin{align}\label{lambd2}
\Lambda^2_l=\{0, \pm 1,...,\pm l\}\times\{0,\pm 1,...,\pm l\}.
\end{align}
Observe that $Q_l^1$ and $Q_l^2$ commute and the idempotent $P_{l}$ defined by $P_l=(1-Q^1_l)(1-Q_l^2)$ projects $L_p(M_n)$ onto $L^{\Lambda^2_l}_p(M_n)$.

\begin{lemma}\label{qcbd}
For $1<p<\8$, $n\in\overline\nz$ such that $n>2l$ ,
\[
\|Q^1_l: L_p(M_n)\to L_p(M_n)\|_{\cb}\leq C_p,\quad \|Q^2_l: L_p(M_n)\to L_p(M_n)\|_{\cb}\leq C_p
\]
\[
\|P_l: L_p(M_n)\to L_p(M_n)\|_{\cb}\leq C_p,
\]
for some constant $C_p$ independent of $n,l$.
\end{lemma}
\begin{proof}
{As we proved in Lemma \ref{prbd}, $Q^1_l$ and $Q^2_l$ are completely bounded operators on $L_p(L\zz_n)$. Therefore they are also completely bounded on $L_p(L\zz_n\overline \otimes L\zz_n)$. This implies that $Q^1_l$ and $Q^2_l$ are completely bounded on $L_p(M_n)$ for $n > 2l$ and $n\in\overline\nz$ by Lemma \ref{trans}. Here is another argument for $n=\8$. Note that we have for $a_{j,k}\in M_m$,
\begin{align*}
(\id_{M_m}\otimes Q^1_l\otimes \id_{L(\zz^2)})&(\sum_{j,k} a_{j,k}\otimes u_\ta^j v_\ta^k\otimes \la(j,k)) \\
&= (\id_{M_m}\otimes \id_{\rx_\ta}\otimes Q^1_l)(\sum_{j,k} a_{j,k}\otimes u_\ta^j v_\ta^k\otimes \la(j,k)).
\end{align*}
We deduce that $Q_l^1$ is completely bounded on $L_p(\rx_\ta\otimes L(\zz^2))$ and the assertion for $Q_l^1$ follows. The case of $Q^2_l$ is similar. As a consequence, $P_l$ is also completely bounded in $L_p$.
}
\end{proof}

\begin{prop} \label{I} Let $1<p<\8$, $\beta>0$ and $n>2l$. Let $\psi$ be a conditionally negative length function on $\zz_n$ satisfying $\psi(l)\le \psi(j)$ for $| l |\le |j|$. Then for any $m\in\nz$ and $a_{ij}\in M_m, i,j\in \zz_n$, we have
\begin{align}\label{Ia}
 \Big\|\sum_{\substack{l\le |j|\le n/2\\ \psi(j)>0}}& \psi(j)^{-\beta} a_{jk}\otimes u_j(n)v_k(n) \Big\|_{L_p(M_m(M_n))}\\
 &\leq c_p\psi(l)^{-\beta} \Big\|\sum_{j,k}a_{jk}\otimes u_j(n)v_k(n) \Big\|_{L_p(M_m(M_n))},\nonumber
\end{align}
for some constant $c_p$ independent of $m,n$ and $l$.
\end{prop}

\begin{proof}
Let $2<p<p_0$ be such that $\frac{1}{p}=\frac{1-\theta}{p_0}+\frac{\theta}{2}$ for some $0<\theta<1$. We define $F_j(z)=(\frac{\psi(l)}{\psi(j)})^{z\alpha}e^{(z-\theta)^2}$, for some $\alpha$ large enough so that $\theta\alpha=\beta$. Define a new operator $T$ by
\[
 T(z)(\sum_{j,k}a_{jk}\otimes u_j(n)v_k(n))=\sum_{|j|\geq l}F_j(z)a_{jk}\otimes u_j(n)v_k(n).
\]
Let $z=\ii t$. Consider the Fourier multiplier
$$
A_\psi(\sum_{j,k} a_{jk}\otimes u_j(n) v_k(n))= \sum_{j,k}\psi(j) a_{jk}\otimes u_j(n) v_k(n).
$$
By \cite{JM12}*{Corollary 5.4} (see also \cite{Co83}), we have $\|A^{\ii s}_\psi f\|_{p_0}\le C_{p_0}e^{c_{p_0}|s|} \|f\|_{p_0}$ for all $s\in\rz$ and {$f\in L_{p_0}(M_m(M_n))$}.
Then
\[
\|T(\ii t):L_{p_0}\to L_{p_0}\|\leq C_{p_0}e^{{c}_{p_0}\al |t|-t^2}\|Q^1_l: L_{p_0}\to L_{p_0}\|,
\]
for some constants $C_{p_0}$ and ${c}_{p_0}$ independent of $n$ and $k$. By Lemma \ref{prbd}, $T(\ii t)$ is bounded. Now let $z=1+\ii t$. Since $|\frac{\psi(l)}{\psi(j)}|\leq 1$, we have
\[
\|T(1+\ii t): L_2\to L_2\|\leq |e^{(1+\ii t-\theta)^2}|\leq e^{-t^2+(\theta-1)^2}.
\]
Therefore, $T(1+\ii t)$ is also bounded. For $z=\theta$, the assertion follows from Stein's interpolation theorem \cite{S56}. By duality, the result holds for $1<p\leq 2$ as well.
\end{proof}

\begin{prop}\label{II}
Let $1<p<\8$ and $\beta>0$. For any conditionally negative length function $\psi$ on $\zz_n$, any $m\in\nz$ and $a_{ij}\in M_m, i,j\in \zz_n$, we have
\begin{align}\label{IIa}
\Big\|\sum_{\substack{j,k\\
\psi(j)+\psi(k)>0}}&(\frac{\psi(j)}{\psi(j)+\psi(k)})^{\beta}a_{jk}\otimes u_j(n)v_k(n) \Big\|_{L_p(M_m(M_n))}\\
 &\leq c_p \Big\|\sum_{j,k}a_{jk}\otimes u_j(n)v_k(n) \Big\|_{L_p(M_m(M_n))},
\end{align}
for some constant $c_p$ independent of $m$ and $n$.
\end{prop}

\begin{proof}
It follows from the same argument as that of Proposition \ref{I} applied to $F_{j,k}(z)=(\frac{\psi(j)}{\psi(j)+\psi(k)})^{z\alpha}e^{(z-\theta)^2}$.
\end{proof}

Let $A^{(1)}_n(u_j(n)v_k(n))=\psi_n(j)u_j(n)v_k(n)$ and $A^{(2)}_n(u_j(n)v_k(n)) =\psi_n(k)u_j(n)v_k(n)$. Then $A_n=A_n^{(1)}+A_n^{(2)}$. Here we allow $\psi_n$ to be any conditionally negative length function with $\psi_n(k)\le\psi(l)$ if $|k|\le |l|$. By \eqref{tphi}, $A_n^{(1)}, A_n^{(2)}$ and $A_n$ are all generators of certain semigroups on $M_n$.

%

\begin{cor}\label{Pl}
Let $1< p<\8,\bt>0$ and $n\in\overline \nz$ such that $n>2l$. Then
  \[
  \|A_n^{-\bt}(1-P_l): L_p(M_n)\to L_p(M_n)\|_{\cb}\le C_p\psi_n(l)^{-\bt},
  \]
  where $C_p$ is independent of $n,l\in \nz$.
\end{cor}

\begin{proof}
By (\ref{Ia}) and (\ref{IIa}), we have for any $m\in\nz$ and any finite sum $x=\sum_{j,k} a_{j,k}\otimes u_j(n)v_k(n)\in M_m\otimes M_n$,
\begin{align*}
  \|\id\otimes Q^1_l(x)\|_{L_p(M_m(M_n))}&\leq c_p{\psi_n}(l)^{-\beta}\|\id\otimes( A_n^{(1)})^{\beta}x\|_{L_p(M_m(M_n))}\\
  &\leq c_p{\psi_n}(l)^{-\beta} \|\id\otimes(A_n^{(1)}+A_n^{(2)})^{\beta}x\|_{L_p(M_m(M_n))}.
\end{align*}
Similar inequality holds for $Q_l^2$. Using Lemma \ref{qcbd}, we get
\begin{align*}
  \|\id\otimes (1-P_l)(x)\|_{L_p(M_m(M_n))}& =\|\id\otimes [Q^1_l+Q_l^2(1-Q_l^1)](x)\|_{L_p(M_m(M_n))} \\
  &\leq (c_p{\psi_n}(l)^{-\beta}+\tilde{c}_p{\psi_n}(l)^{-\beta}) \|\id\otimes (A_n)^{\beta}x\|_{L_p(M_m(M_n))} \\ &=C_p{\psi_n}(l)^{-\beta}\|\id\otimes (A_n)^{\beta}(x)\|_{L_p(M_m(M_n))},
\end{align*}
for some constants $c_p$, $\tilde{c}_p$ and $C_p$ independent of $m,n$ and $l$.
\end{proof}

We remark that the previous complete boundedness results for matrix algebras can be alternatively proved using Lemma \ref{trans} in the same way as what we did in Lemma \ref{qcbd}.

\subsection{Continuous fields of compact quantum metric spaces}\label{cts fields 1}
Let $\ax_\ta$ denote the rotation $\mathrm{C}^*$-algebra associated to $\ta\in[0,1)$. It is well known that $\ax_0 = C(\tz^2)$. Let $(M_n)_{sa}$ and $(\ax_\ta)_{sa}$ denote the subspace of self-adjoint elements of $M_n$ and $\ax_\ta$. Let $\mx_n=(M_n)_{sa}$ and $\mx_\8=\ax_\ta^\8\cap (\ax_\ta)_{sa}$. Note that $\overline{\mx_n}=(M_n)_{sa}$ for $n\in\overline\nz$. Let $\sx$ denote a suitable set of continuous sections of the continuous field of order-unit spaces over $\overline\nz$ with fibers $\overline\mx_n$. In this section we show that
\[
(\{(\mx_n, L_n\}_{n\in \overline\nz}, \sx)
\]
is a continuous field of compact quantum metric spaces. In order to establish this, we have to first consider two cases, namely $\theta=0$ and $\theta$ a non-zero rational.

\subsubsection{Approximation in the commutative case}
A key tool is the following map defined by extending comultiplication linearly:
\begin{align}\label{rhon}
\rho_n: \cz(\zz)\otimes \cz(\zz)&\to M_n\\
\lambda_j\otimes\lambda_k&\mapsto u_j(n)v_k(n)\nonumber.
\end{align}
Note that $\rho_n$ is defined for trigonometric polynomials in $C(\tz^2)=C(\tz) \otimes_{\min} C(\tz)$. Also, for a fixed $n$, $\rho_n$ is not a $^*$-homomorphism. Therefore, we need to introduce a $^*$-homomorphism $\rho_{\omega}$ as follows. First we recall the ultraproduct construction; see e.g. \cite{BO08}*{Appendix A}. Let $\om$ be a free ultrafilter on $\nz$. Note that the Banach space $\prod_\om X_n$ is defined as a quotient of $\prod_n X_n$ by the subspace
\[
I_{\omega}=\{(x_n)\in \prod_n X_n: \lim_{n\to\omega}\|x_n\|=0\}
\]
with respect to the norm
\[
\|(x_n)^{\bullet}\|=\lim_{n\to \omega}\|x_n\|_{X_n}.
\]
If $(X_n)$ are C$^*$-algebras, we obtain a new C$^*$-algebra $\prod X_n/{I_{\omega}}$, since $I_\om$ is an ideal. If in addition $(X_n)$ are von Neumann algebras with finite traces, then the von Neumann algebra ultraproduct $(X_n)^{\omega}$ is defined to be $\prod X_n/I_{\tau_{\omega}}$, where
\[
I_{\tau_{\omega}}=\{(x_n)\in \prod_n X_n: \lim_{n\to\om} \tau(x_n^*x_n)=0\}.
\]
Note that $I_{\omega}\subset I_{\tau_{\omega}}$ and we obtain a quotient $^*$-homomorphism
\[
\sigma_w: \prod\nolimits_{\omega} X_n\to (X_n)^{\omega}.
\]
Now we focus on $X_n=M_n$. We define the maps $\pi_1,\pi_2: C(\mathbb{T})\to \prod_{\omega} M_n$ as follows:
\[
\pi_1(\lambda_j)=(\pi_n^{(1)}(\lambda_j))^{\bullet}, \quad \text{where} \quad \pi_n^{(1)}(\lambda_j)=u_j(n),
\]
and
\[
\pi_2(\lambda_k)=(\pi_n^{(2)}(\lambda_k))^{\bullet}, \quad \text{where} \quad \pi_n^{(2)}(\lambda_k)=v_k(n).
\]
Suppose $\sum_k f^k\otimes g^k$ is a tensor of polynomials in $C(\tz)$. Then
\[
\rho_n(\sum_k f^k\otimes g^k)=\sum_k \pi_n^{(1)}(f^k) \pi_n^{(2)}(g^k)
\]
is a densely-defined linear map on $C(\tz^2)$. The maps $\pi_1$ and $\pi_2$ are $^*$-homomorphisms with commuting ranges. In fact we have
 \begin{align*}
\|[\pi_1(\lambda_1),\pi_2(\lambda_1)]\|&=\lim_{n\to\om}\|[u_1(n),v_1(n)]\| =\lim_{n\to\om}\|u_1(n)v_1(n)-v_1(n)u_1(n)\|\\
                                        &=\lim_{n\to\om}\|(e^{\frac{2\pi \ii}{n}}-1)v_1(n)u_1(n)\|=\lim_{n\to \8}|e^{\frac{2\pi \ii}{n}}-1|=0.
 \end{align*}
It follows that the map $\rho_\om:=(\rho_n)^{\bullet}$ extends to the universal $\mathrm{C}^*$-algebra $C(\tz)\otimes_{\max} C(\tz)$ and
\[
\rho_{\omega}: C(\tz^2)=C(\mathbb{T})\otimes_{\min} C(\mathbb{T})=C(\mathbb{T})\otimes_{\max}C(\mathbb{T})\to \prod\nolimits_{\omega} M_n
\]
is a well-defined $^*$-homomorphism. Let $\pi_{\omega}=\sigma_{\omega} \rho_{\omega}$. Then $\pi_\om: C(\tz^2)\to (M_n)^\om$ is also a $^*$-homomorphism.



\begin{lemma}\label{ct2cts}
The maps $\pi_{\omega}$ and $\rho_{\omega}$ are faithful. In particular, $\lim_{n\to \8} \|(\id\otimes \rho_n) (f)\|_{M_r\otimes M_n} = \|f\|_{M_r\otimes C(\tz^2)}$ for $f\in M_r\otimes [C(\zz)\otimes C(\zz)]$ and $r\in \nz$.
\end{lemma}

\begin{proof}
Let $\tau_n$ be the normalized trace on $M_n$ and $\tau_{\omega}=\lim \tau_n$. Then since $u_j(n)$ is a diagonal matrix and $v_k(n)$ is a shift matrix, we have
\[
\tau_{\omega}(\sigma_{\omega} \rho_{\omega}(\lambda(j)\otimes \lambda(k)))=\lim_{n\to\om} \tau_n(u_j(n)v_k(n))= \delta_{j0}\delta_{k0}= (\tau \otimes \tau)(\lambda(j)\otimes \lambda(k)),
\]
where $\tau$ is the canonical trace on $C^*_r(\zz)\simeq C(\tz)$. This proves that $\pi_{\omega}$ is trace-preserving. Now let $x\in C(\tz^2)$, and $\pi_{\omega}(x)=0$. Then since $\pi_{\omega}$ is trace-preserving, we have
\[
 \tau_{\omega}(\pi_\om(x^*)\pi_\om(x))=\tau_{\omega}(\pi_\om(x^*x))=\tau\otimes \tau(x^*x)=0.
\]
Since the trace on $C(\tz^2)$ is faithful, this proves that $\pi_{\omega}$ is faithful and so is $\rho_{\omega}$. It follows that $\rho_\om$ extends to a faithful $^*$-homomorphism $\id\otimes \rho_\om: M_r(C(\tz^2))\to M_r(\prod_\om M_n)$. We deduce that
\[
\lim_{n\to\om} \|(\id\otimes \rho_n)(f)\|_{M_r\otimes M_n}=\|f\|_{M_r\otimes C(\tz^2)}
\]
for any finite linear combination $f=\sum_{j,k}a_{jk}\otimes (\la_j\otimes \la_k)\in M_r\otimes C(\tz^2)$. Here we have used the fact $M_r(\prod_\om M_n) = \prod_\om M_r(M_n)$; see \cite{Pi03}*{(2.8.1)}. But the ultrafilter $\om$ is arbitrary, and so the assertion follows.
\end{proof}

Let us define the vector space
\begin{align}\label{e:polyxy}
{\text{Poly}}(x,y)=\bigcup_{k\ge1}\{p=\sum_{\substack{|i|,|j|\leq k}} a_{ij}x^iy^j: a_{ij} \in \cz\},
\end{align}
where $x,y$ are noncommuting variables.  By definition, ${\text{Poly}}(x,y)$ is a subspace of ordered noncommutative Laurent polynomials in two variables. Let $u, v$ be the unitary generators of $C(\tz^2)$. For instance, we may take $u=\la_1\otimes 1$ and $v=1\otimes \la_1$. We define a linear map
\[
\si: {\text{Poly}}(x,y)\to  \cz(\zz)\otimes \cz(\zz)\subset C(\tz^2), \quad x^j y^k \mapsto u^j v^k.
\]
Since every element $g\in\cz(\zz^2)$ is of the form $\Re(g)+\ii \Im(g)$ for $\Re(g)= \frac12(g+g^*), \Im(g)=\frac12[-\ii g+ (-\ii g)^*]\in \cz(\zz^2)_{sa}$, every element of $\cz(\zz^2)_{sa}$ can be written as $\frac12[\si(f)+\si(f)^*]$ for some $f\in \text{Poly}(x,y)$. We define $\rz$-linear maps for $n\in\overline\nz$,
\begin{align*}
\varpi_n: \text{Poly}(x,y)&\to M_n,\quad f\mapsto \rho_n(\si(f))+\rho_n(\si(f))^*
.
\end{align*}
Here and in the following we understand $\rho_\8 = \id$. Equivalently, we have
\[
\varpi_n(x^j y^k) =  u_j(n)v_k(n)+v_{-k}(n)u_{-j}(n).
\]
We observe that the range of $\varpi_n$ is the space of self-adjoint elements of $M_n$. Also, $\rho_n$ does not preserve the $^*$-involution, which suggests the use of $\rho_n(\si(f))^*$ rather than $\rho_n(\si(f)^*)$. Recall that $\sx$ denote a set of continuous sections of the continuous field of order-unit spaces over $\overline\nz$ with fibers $(M_n)_{sa}$. Similar to Proposition \ref{cfqms}, we will see in the next result that we may choose
$
\sx=\{(\varpi_n(f))_{n\in \overline\nz}: f\in \text{Poly}(x,y)\}.
$
Recall also $\mx_n=(M_n)_{sa}$ and $\mx_\8= \cz(\zz^2)_{sa}$.

\begin{prop}\label{ctqms1}
Let $\Ga^n$ be the gradient form associated to $A_n$ on $M_n$. Then
\begin{align*}
\lim_{n\to\8} \| \rho_n(\si(f)^*) &-\rho_n(\si(f))^*\|_{M_n} =0,\\
\lim_{n\to \8} \|\Gamma^n(\varpi_n(f),\varpi_n(f)) &-\rho_n[\Gamma ( \si(f)+\si(f)^*, \si(f)+\si(f)^* ) ] \|_{M_n} =0,
\end{align*}
for $f\in {\text{Poly}}(x,y)$. {Therefore, $(\{\mx_n, L_n\}_{n\in \overline\nz}, \sx)$ is a continuous field of compact quantum metric spaces.}
\end{prop}
\begin{proof}
Note that we may write $f=\sum_{j,k} a_{jk}x^j y^k$ as a finite sum and that
\begin{align*}
\rho_n(\si(f)^*) -\rho_n(\si(f))^* & = \sum_{j,k} \bar a_{jk} u_{-j}(n)v_{-k}(n) -\sum_{j,k} \bar a_{jk} v_{-k}(n)u_{-j}(n)\\
&=\sum_{j,k} \bar a_{jk}(1-e^{-\frac{2\pi \ii j k}n}) u_{-j}(n)v_{-k}(n).
\end{align*}
By the triangle inequality, we have $\lim_{n\to\8} \| \rho_n(\si(f)^*)-\rho_n(\si(f))^*\|_{M_n} =0$. Together with Lemma \ref{ct2cts} we have
\begin{align*}
\lim_{n\to \8} \|\varpi_n(f)\|_{M_n} &= \lim_{n\to \8}  \|\rho_n[\si(f)+\si(f)^*] \|_{M_n} =  \|   \si(f)+\si(f)^*  \|_{C(\tz^2)}.
\end{align*}
Since $\Ga^n$ is sesquilinear,
\begin{align*}
\Gamma^n(\varpi_n(f),\varpi_n(f)) & =   \Ga^n[\rho_n(\si(f)), \rho_n(\si(f))]+\Ga^n[\rho_n(\si(f)), \rho_n(\si(f))^*] \\
 &\ \ + \Ga^n[\rho_n(\si(f))^*, \rho_n(\si(f))] + \Ga^n[\rho_n(\si(f))^*, \rho_n(\si(f))^*]  .
\end{align*}
We have a similar formula for $\Ga(  \si(f)+\si(f)^*,  \si(f)+\si(f)^*)$. Using the commutation relation, we have
\begin{align*}
  \Ga^n[ \rho_n(\si(f)),\rho_n(\si(f))] &= \frac12\Big[\sum_{j,j',k,k'} e^{\frac{2\pi \ii (j'-j)k}{n}} \bar a_{jk} a_{j'k'} [\psi_n(-j)+\psi_n(-k)+\psi_n(j')+\psi_n(k') \\
  &\ \ -(\psi_n(j'-j)+\psi_n(k'-k))]  u_{j'-j}(n) v_{k'-k}(n)\Big].
\end{align*}
Similarly, by recalling the definition of $\rho_n$ as in \eqref{rhon},
\begin{align*}
\rho_n(\Ga(\si(f),\si(f))) & = \frac12\Big[\sum_{j,j',k,k'}   \bar a_{jk} a_{j'k'} [\psi(-j)+\psi(-k)+\psi(j')+\psi(k') \\
  &\ \ -(\psi(j'-j)+\psi(k'-k))]  u_{j'-j}(n) v_{k'-k}(n)\Big].
\end{align*}
Since $\lim_{n\to \8} | e^{\frac{2\pi \ii (j'-j)k}{n}} -1 | = 0$ and $\lim_{n\to\8}\psi_n(k')=\psi(k')$ for any $j,j',k,k'$, we have
\[
\lim_{n\to \8} \|\rho_n(\Ga(\si(f),\si(f)))  - \Ga^n[ \rho_n(\si(f)),\rho_n(\si(f))] \|_{M_n} =0.
\]
The same conclusion holds for the three other terms in the expression of $\Ga^n(\vpi_n(f),\vpi_n (f))$ and $\rho_n[\Gamma ( \si(f)+\si(f)^* , \si(f)+\si(f)^*)]$.  By the triangle inequality, we have
\[
\lim_{n\to \8}  \|\Gamma^n(\varpi_n(f),\varpi_n(f)) -\rho_n [\Gamma ( \si(f)+\si(f)^* , \si(f)+\si(f)^*  ) ] \|_{M_n} =0.
\]
By Lemma \ref{ct2cts}, we find
\begin{align*}
\lim_{n\to\8} \|\Ga^n(\varpi_n(f),\varpi_n(f))\|_{M_n} &= \lim_{n\to\8}   \|\rho_n [\Gamma (  \si(f)+\si(f)^* ,  \si(f)+\si(f)^*  ) ] \|_{M_n} \\
&=  \|  \Gamma ( \si(f)+\si(f)^* , \si(f)+\si(f)^*  )  \|_{C(\tz^2)}.
\end{align*}
We have proved that $ (\varpi_n(f))_{n\in\overline\nz}$ is a continuous section and $n\mapsto L_n(\vpi_n(f))$ is continuous at $n\in\overline\nz$ for all $f \in \text{Poly}(x,y)$. Note that the set $\{\varpi_n(f): f\in \text{Poly}(x,y) \}$ is dense in $(M_n)_{sa}$ for $n\in\overline\nz$. Using the same argument as for Proposition \ref{cfqms}, we conclude that $(\{\mx_n, L_n\}_{n\in \overline\nz}, \sx)$ is a continuous field of compact quantum metric spaces.
\end{proof}


\subsubsection{Approximation for rational $\theta$} Let $0<\theta<1$ be a rational number. Then $\ax_\ta\simeq C(\tz)\rtimes_{\theta} \zz$. On the other hand $\ax_\ta$ is the universal C$^*$-algebra generated by two unitaries $u$ and $v$, which commute according to the following rule
\[
uv=e^{2\pi \ii\theta}vu.
\]
Now we extend the map $\rho_n$ defined previously, from $\theta=0$ to $\theta$ rational. In the following, we embed $\ax_\ta$ in $M_m(C(\tz^2))$ using the unitaries $u_j(n)$ and $v_k(n)$ which were introduced in the previous section. Since $\theta$ is rational, we can write $\theta=\frac{p}{m}$, such that $(p,m)=1$. Note that $m$ is fixed. We define a $^*$-homomorphism
\begin{align*}
\vsi:\ax_\ta&\to M_m\otimes_{\min} C(\tz) \otimes_{\min} C(\tz)\\
u^j&\mapsto u_j(m)\otimes \lambda_j\otimes 1\\
v^k&\mapsto v_{kp}(m)\otimes 1\otimes \lambda_k.
\end{align*}
Recall that the canonical trace $\tau$ on $\ax_\ta$ is faithful; see e.g. \cite{Ri90, B1}. Since $\vsi$ is trace-preserving, it is injective. Now let $\rho_n^{\theta}= (\id\otimes\rho_n)\circ \vsi |_{\ax_\ta^\8}$, i.e.
\begin{align}\label{rhonta}
\rho_n^{\theta}: \ax_\ta^\8&\to M_m(M_n)\\
u^jv^k&\mapsto u_j(m)v_{kp}(m)\otimes u_j(n) v_k(n)= U_j(n)V_k(n)\nonumber
\end{align}
where $U_j(n) :=u_j(m)\otimes u_j(n)$ and $V_k(n):=v_{kp}(m)\otimes v_k(n)$. Note that
\[
\rho_n^\ta(v^ku^j)=v_{kp}(m) u_j(m)\otimes u_j(n) v_k(n) = e^{-\frac{2\pi\ii jkp}{m}}U_j(n)V_k(n).
\]
It is clear that $\rho_n^\ta$ is well-defined. Moreover, we have $U_j(n)^*  = U_{-j}(n)$ and $V_k(n)^* = V_{-k}(n)$.
We want the image of $\rho_n^\ta$ to generate the full matrix algebra $M_n$.
It suffices to check the commutation relation for $U_1(n)$ and $V_1(n)$. We have
\[
(u_1(m)\otimes u_1(n))\cdot (v_p(m)\otimes v_1(n))= e^{2\pi \ii \theta+\frac{2\pi \ii}{n}} (v_p(m)\otimes v_1(n))\cdot (u_1(m)\otimes u_1(n)).
\]
This means $U_1(n)V_1(n)=e^{2\pi \ii \eta_n}V_1(n)U_1(n)$, where $\eta_n=\theta+\frac{1}{n} = \frac{pn + m}{mn}$. In order for $U_1(n)$, $V_1(n)$ to generate $M_n$, we need to write $\eta_n$ as $\frac{a}{n}$ for some $a$, such that $(a, n)=1$; see e.g. \cite{Dav}. For this, choose a subsequence $n=m^{k_n}$ for some exponents $k_n$. Then $\eta_{n} = \frac{pm^{k_n-1}+1}{m^{k_n}} = \frac{a}{n}$. Suppose $q$ is a prime number which divides both $pm^{k_n-1}+1$ and $m^{k_n}$ for $k_n>1$. So $q$ divides $m$. This implies that $q$ divides $m^{k_n - 1}$ and hence it divides $pm^{k_n -1}$. But $q$ also divides $pm^{k_n -1} + 1$. Hence $q$ divides 1 which is a contradiction. Therefore, $n=m^{k_n}$ does the job, and it suffices to take the subsequence $n_k=m^{k+1}$. Let us state what we have found so far\footnote{In fact, our argument even works when $p$ and $m$ are not coprime. The condition $(p,m)=1$ implies that $u_1(m)$ and $v_p(m)$ generate $M_m$ (instead of a subalgebra of $M_m$). In this way, we get smaller matrix algebras. This is not essential as we will send the dimension of matrix algebras to infinity eventually.}:
\begin{lemma}\label{rhotank}
The map $\rho_{n_k}^{\theta}: \ax_\ta^\8\to M_{n_k}$ is surjective.
\end{lemma}
The Lemma above says that $C^*(\rho_{n_k}^{\theta}(\ax_\ta^\8))= M_{n_k}$, where $C^*(\rho_{n_k}^{\theta}(\ax_\ta^\8))$ denotes the C$^*$-algebra generated by $\rho_{n_k}^{\theta}(\ax_\ta^\8)$. We next check the continuity at infinity. Recall that $\mx_\8=\ax_\ta^\8\cap (\ax_\ta)_{sa}$. Similar to before, we define a linear map
\[
\si: \text{Poly}(x,y) \to \ax_\ta^\8,\quad x^jy^k \mapsto u^jv^k.
\]
Note that every element of $\mx_\8$ can be written as $\si(f)+\si(f)^*$ for some $f\in \text{Poly}(x,y)$. Choose $n_k\in \nz$ as above. We define $\rz$-linear maps for $k\in\overline\nz$
\begin{align*}
  \vpi_{n_k}: &\text{Poly}(x,y)\to M_{n_k}, \quad f\mapsto \rho_{n_k}^\ta (\si(f))+\rho_{n_k}^\ta(\si(f))^*,
\end{align*}
where we understand $\rho_\8^\ta=\id$.  As usual, we define
\[
A_n(U_j(n)V_k(n))=(\psi_n(j)+\psi_n(k))U_j(n)V_k(n)
\]
and $L_n(f)=\|\Ga^n(f,f)^{1/2}\|_\8$ on $M_n$. Let $\sx$ denote a set of continuous sections of the continuous field of order-unit spaces over $\overline\nz$ with fibers $(M_{n_k})_{sa}$. We will see in the following result that we may choose $\sx=\{(\vpi_{n_k}(f))_{k\in\overline\nz}: f\in \text{Poly}(x,y)\}$.
\begin{prop}\label{ctfata}
  Choose $n_k\in \nz$ as above. Then $(\{\mx_{n_k}, L_{n_k}\}_{k\in\overline\nz}, \sx)$ is a continuous field of compact quantum metric spaces.
\end{prop}
\begin{proof}
  We follow an argument similar to that of Proposition \ref{ctqms1}. Note that
  \begin{align}\label{e:rhonta*}
  \rho_n^\ta[(u^jv^l)^*] = e^{-2\pi\ii \ta jl}U_{-j}(n) V_{-l}(n), \quad
  [\rho_n^\ta(u^jv^l)]^* = e^{-2\pi\ii (\ta+\frac1n) jl}U_{-j}(n) V_{-l}(n).
  \end{align}
  From here we deduce that for $f\in \text{Poly}(x,y)$,
  \[
  \lim_{k\to \8} \|\rho_{n_k}^\ta[\si(f)^*] -[\rho_{n_k}^\ta(\si(f))]^* \|_{M_{n_k}}=0.
  \]
 By Lemma \ref{ct2cts}, we have
 \begin{align}\label{e:rhonkta}
 \lim_{k\to \8} \|\rho_{n_k}^\ta [\si(f)]\|_{M_{n_k}}& = \lim_{k\to \8} \|(\id\otimes \rho_{n_k})\circ\vsi [\si(f)]\|_{M_m\otimes M_{n_k}} \\
 &=\|\vsi[\si(f)]\|_{M_m\otimes C(\tz^2)} = \|\si(f)\|_{\ax_\ta}.\notag
 \end{align}
  Combining the above two equations  together, we have
  \begin{align}\label{e:pinkf}
  \lim_{k\to\8} \|\vpi_{n_k}(f)\|_{M_{n_k}}=\lim_{k\to\8}\|\rho_{n_k}^\ta(\si(f)+\si(f)^*)\|_{M_{n_k}}= \|\si(f)+\si(f)^*\|_{\ax_\ta}.
  \end{align}
Since $\lim_{k\to\8}e^{\frac{2\pi \ii (j'-j)l}{n_k}}=1$ and $\lim_{k\to\8} \psi_{n_k}(l')=\psi(l')$ for all $j,j',l,l'$, a direct computation yields
  \begin{align}\label{e:rhonkga}
  &\ \ \lim_{k\to \8} \|\rho_{n_k}^\ta[\Ga(u^jv^l, u^{j'}v^{l'})] -\Ga^{n_k}[\rho_{n_k}^\ta(u^jv^l), \rho_{n_k}^\ta(u^{j'}v^{l'})] \|_{M_{n_k}} \\
  & =\frac12 \lim_{k\to \8} \big\|   \big( [\psi(-j)+\psi(-l)+\psi(j')+\psi(l')-\psi(j'-j)-\psi(l'-l)] \notag \\
  &\ \ \ -e^{\frac{2\pi \ii (j'-j)l}{n_k}} [\psi_{n_k}(-j)+\psi_{n_k}(-l)+\psi_{n_k}(j')+\psi_{n_k}(l')-\psi_{n_k}(j'-j)-\psi_{n_k}(l'-l)]\big) \notag\\
  &\ \ \  \times e^{2\pi\ii \ta(j'-j)l}U_{j'-j}(n_k) V_{l'-l}(n_k)\big\|_{M_{n_k}}\notag \\
  &=0.\notag
  \end{align}
  Together with the sesquilinearity of $\Ga^n$ we deduce that
  \begin{align*}
  \lim_{k\to \8}&  \|\Gamma^{n_k}(\varpi_{n_k}(f),\varpi_{n_k}(f)) -\rho_{n_k}^\ta  [\Gamma ( \si(f)+\si(f)^* , \si(f)+\si(f)^* ) ] \|_{M_{n_k}} =0.
  \end{align*}
  Using Lemma \ref{ct2cts}, we find
  \begin{align}\label{e:gankpif}
  \lim_{k\to\8}\|\Ga^{n_k}(\vpi_{n_k}(f),\vpi_{n_k}(f))^{1/2} \|_{M_{n_k}} &= \|\Ga(\si(f)+\si(f)^*, \si(f)+\si(f)^*)^{1/2}\|_{\ax_\ta}.
  \end{align}
Therefore, $(\vpi_{n_k}(f))_{k\in\overline\nz}$ is a continuous section and $k\mapsto L_{n_k}(\vpi_{n_k}(f))$ is continuous at $k\in\overline\nz$ for all $f\in \text{Poly}(x,y)$. Since the set $\{\vpi_{n_k}(f): f\in \text{Poly}(x,y) \}$ is dense in $(M_{n_k})_{sa}$ for $k\in \overline\nz$, we conclude that $(\{\mx_{n_k}, L_{n_k}\}_{k\in\overline\nz}, \sx)$ is a continuous field of compact quantum metric spaces as in Proposition \ref{cfqms}.
\end{proof}

From now on, with abuse of notation, when we use $\rho_n^{\theta}$ for $\ax_\ta^\8$, we always mean $\rho_{n_k}^{\theta}$. We still need to consider the case when $\ta$ is irrational. In fact, we now deal with a more general situation.


\subsubsection{Continuous field for the higher dimensional case }\label{cts fields 2}
In the following, let $\ax^{d}_{\Theta}$ denote the $d$-dimensional noncommutative torus which was introduced in Section \ref{anaest}. Recall that $\Theta=(\ta_{ij})$ is a $d\times d$ skew symmetric matrix. We will discuss $\ax_\Theta^{d}$ in Section \ref{higher cb qgh} in more depth. In this section we only show that they form a continuous field of compact quantum metric spaces.

Recall that for a compact Hausdorff space $X$, a $C(X)$-algebra is a C$^*$-algebra $A$ endowed with a unital morphism from $C(X)$ of continuous functions on $X$ into the center of the multiplier algebra $M(A)$ of $A$; see \cite{Kas88}. In the following we are going to derive some results about the rotation algebras using the Heisenberg group $\hz_B=\zz^m\times_B\zz^d$, where $m=\frac{d(d-1)}{2}$ and $B: \zz^d\times \zz^d \to \zz^m$ is a skew symmetric bilinear map. For $u=(u_i)_i$, $u'=(u'_j)_j$ in $\zz^d$ and $z, z'$ in $\zz^m$, the multiplication on $\hz_B$ is defined by
\[
(z,u)(z',u')=(z+z'+B(u,u'), u+u')
\]
where $[B(u,u')]_{rs}=u_ru'_s-u'_r u_s$ for $r, s=1,...,d$. Here we have identified $B(u,u')$ with a vector in $\zz^m$. Indeed, since $B$ is skew symmetric, the diagonal of $B(u,u')$ is 0 and the upper triangular submatrix has $\frac{d(d-1)}2$ entries. For definiteness, we use the upper triangular submatrix to represent $B(u,u')$. Let $C^*(\hz_B)$ and $C^*_r(\hz_B)$ be the universal C$^*$-algebra and the reduced C$^*$-algebra of $\hz_B$, respectively.

\begin{lemma}\label{ctmalg}
 C$^*(\hz_B)$ is a $C(\tz^m)$-algebra.
\end{lemma}
 \begin{proof}
Note that $C^*(\hz_B)=C^*_r(\hz_B)$ since $\hz_B$ is amenable. Let $\lambda(k,j)\in C^*(\hz_B)$ be the left regular representation.  The left regular representation on $\zz^m$ induces a representation on $\ell_1(\zz^m)$ given by
\[
\lambda:  \ell_1(\zz^m)\to C(\tz^m), \quad f=\sum_{l\in \zz^m} f(l)e_l \mapsto \la(f)=\sum_l f(l)\lambda(l,0),
\]
where $(e_l)_l$ is the standard orthonormal basis of $\ell_2(\zz^m)$. Let $f\in \ell_1(\zz^m)$. Then we have
 \begin{align*}
 \lambda(f)\lambda(k,j)&=\sum_{l\in\zz^m} f(l)\lambda(l,0)\lambda(k,0)\lambda(0,j)\\
 &= \sum_{l\in\zz^m} f(l)\lambda(l+k,0)\lambda(0,j) = \lambda(k,j) \la(f).
 \end{align*}
By density, this shows that $C(\tz^m)$ is in the center of $C^*_r(\hz_B)$. Since $C^*(\hz_B)$ is unital, $M(C^*(\hz_B))=C^*(\hz_B)$.  Hence $C^*(\hz_B)$ is a $C(\tz^m)$-algebra.
 \end{proof}

Let $I_{\Theta}$ be the closed two-sided ideal of $C^*(\hz_B)$ generated by all $f\in C(\tz^m)$ with $f(\Theta)=0$ and define $C_{\Theta}=C^*(\hz_B)/I_{\Theta}$. More generally, let $\ax$ be a $C(X)$-algebra.  Denote by $\ax_x$ the quotient of $\ax$ by the closed two-sided ideal
$
I_x
$
generated by all $f\in C(X)$ such that $f(x)=0$. Let $a_x$ denote the image of an element $a\in \ax$ in the fiber $\ax_x$. Recall that the $C(X)$-algebra $\ax$ is said to be a continuous field of C$^*$-algebras over $X$ if the function $\pi_a: X \to \cz$ defined by $\pi_a(x)= \|a_x\|$ is continuous for every $a\in \ax$. In fact, the function $x\mapsto \|a_x\|$ is always upper semi-continuous; see \cite{B97,R89,Dix77} and the references therein. Let us define
\begin{align}\label{e:polyxd}
\text{Poly}_{\vta}(x_1, ...,x_{d})=\Big\{\sum_{(k_1,...,k_d)\in\zz^d} a_{k_1,...,k_d}  x_1^{k_1}\cdots x_d^{k_d}: a_{k_1,...,k_d} \in C(\tz^m) \Big\}.
\end{align}
where the sum is over finite indices and $x_1,...,x_d$ are noncommuting variables. In other words, $\text{Poly}_\vta(x_1, ...,x_{d})$ is a subspace of ordered noncommutative Laurent polynomials in $d$ variables with coefficients in $C(\tz^m)$. In contrast to \eqref{e:polyxy}, scalar coefficients are not sufficient here because $\Theta$ is a variable when considering $(\ax_\Theta)_\Theta$ as a continuous field of C$^*$-algebras in the following. Let $u_1(\Theta),...,u_{d}(\Theta)$ be the generators of $\ax^{d}_\Theta$ and define
\begin{align}\label{e:siTa}
\si_\Theta: \text{Poly}_\vta(x_1, ...,x_{d}) \to \ax_\Theta^\8, \quad a_{k_1,...,k_d} x_1^{k_1}\cdots x_d^{k_d} \mapsto a_{k_1,...,k_d}(\Theta) u_1(\Theta)^{k_1}\cdots u_d(\Theta)^{k_d}
\end{align}
Let $f_\Theta=\si_{\Theta}(f)$ for $f\in \text{Poly}_\vta(x_1, ...,x_{d})$.

\begin{lemma}\label{lower d-dim}
$C_{\Theta}\simeq\ax^{d}_{2\Theta}$ and $\{ \ax^{d}_\Theta \}_{\Theta\in \tz^m}$ is a continuous field of C$^*$-algebras. In particular, the function $\Theta\mapsto \|\si_\Theta(f)\|_{\ax_\Theta}$ is continuous for any $f\in\text{Poly}_\vta(x_1,...,x_d)$.
\end{lemma}
\begin{proof}
Let $(e_i)_{i=1}^d$ be the canonical generators of $\zz^d$. Note that for all $r$ and $s$, we have
$
\lambda(0,e_s)\lambda(0, e_r)=\lambda(B(e_s,e_r), e_s+e_r),
$
and
\begin{align}
\lambda(0,e_r)\lambda(0, e_s)&=\lambda(B(e_r,e_s), e_r+e_s)\nonumber\\
&=\lambda(2B(e_r,e_s),0)\lambda(B(e_s,e_r), e_r+e_s)\nonumber\\
&=\lambda(2B(e_r,e_s),0)\lambda(0, e_s)\lambda(0, e_r).\label{heicomm}
\end{align}
Let us fix $\Theta^0\in \tz^m$. Recall that $B(e_r,e_s)=e_{r,s}$ for $r< s$, where $e_{r,s}$ is a vector in $\zz^m$. Then the map $f$ defined by $f(\Theta)= e^{4\pi \ii \ta_{rs}} - e^{4\pi \ii \ta_{rs}^0}$ is in $I_{\Theta^0}$. Note that the image of $\la(2e_{r,s},0)$ in the quotient $C_{\Theta^0}$ is simply $e^{4\pi \ii \ta_{rs}^0}$. Considering \eqref{heicomm} in $C_{\Theta^0}$, we find that the unitaries $\lambda(0,e_r)$ in $C_{\Theta^0}$ satisfy the commutation relations of $\ax_{2\Theta^0}^d$. This means that one can define a $^*$-homomorphism
\[
\sigma: \ax^{d}_{2\Theta^0}\to C_{\Theta^0},\quad \sigma(u_r)=\lambda(0, e_r) + I_{\Theta^0} \in C_{\Theta^0},
 \]
 where $(u_r)_{r=1}^d$ are the generators of $\ax_{2\Theta^0}^d$.

To identify $C_\Theta$ with $\ax_{2\Theta}^d$, we define for $k=(k_1,...,k_d)\in \zz^d$,
\[
\td \la (0,k) = \la(0, k_1e_{1})\cdots \la(0, k_d e_d).
\]
Let $L(\hz_B)$ be the von Neumann algebra of $\hz_B$. For $i,j,k\in \zz^d$ and $f\in C(\tz^m)$, by Lemma \ref{ctmalg} and \eqref{heicomm} we have
\begin{align*}
\lge \td\la(0, k) \td\la(0, i) f, & \ \td \la(0,j)\rge_{L_2(L(\hz_B), \tau)} = \tau\left( \td\la(0,j)^* \td\la(0, k) \td\la(0, i) f\right) \\
&=\tau\left( f \la \left(\sum_{\al <\bt} 2 i_\al k_\bt B(e_\bt, e_\al), 0\right) \td \la(0,j)^* \td \la(0,k+i) \right)\\
&= \de_{j, k+i} \int_{\tz^m} f(\Theta) \exp\left( -4\pi \ii \sum_{\al<\bt} i_{\al} k_\bt \ta_{\al \bt } \right) \dd\mu( \Theta),
\end{align*}
where $\mu$ is the normalized Haar measure on $\tz^m$. Let $f_n^{\Theta^0}\in C(\tz^m), n\ge 1,$ be a sequence of positive functions such that $\int f_n^{\Theta^0} \dd \mu =1$ and $\lim_{n\to \8} \int f_n^{\Theta^0}(\Theta) g(\Theta)\dd \mu ( \Theta)= g(\Theta^0)$ for $g\in C(\tz^m)$. Let $\om$ be a free ultrafilter on $\nz$. We consider the ultrapower of Hilbert spaces
\[
L_2(L(\hz_B),\tau)^\om = \ell_2(\hz_B)^\om =  L_2(\tz^m, \ell_2(\zz^d))^\om
\]
and ultrapower of von Neumann algebra $ [L(\hz_B) ]^\om$ with trace $\tau_\om$. We may regard each element of $C^*(\hz_B)$ as an element of $[L(\hz_B) ]^\om$. Then for $g\in C(\tz^m)$ and $(\sqrt{f_n^{\Theta^0}}1_{L(\hz_B)})^\bullet \in L_2(L(\hz_B),\tau)^\om$, we have
\begin{align}\label{ultragfn}
\lge g  \td\la(0, k) \td\la(0, i)  (\sqrt{f_n^{\Theta^0}})^\bullet ,& \  \td \la(0,j) (\sqrt{f_n^{\Theta^0}})^\bullet\rge_{L_2(L(\hz_B), \tau)^\om} \\
&= \de_{j,k+i} \exp\left( -4\pi \ii \sum_{\al<\bt} i_{\al} k_\bt \ta^0_{\al \bt } \right) g(\Theta^0) \nonumber\\
& = g(\Theta^0) \lge u_1^{k_1}\cdots u_d^{k_d} u_1^{i_1}\cdots u_d^{i_d} , u_1^{j_1}\cdots u_d^{j_d}\rge_{L_2(\ax_{2\Theta^0}, \tau)}.\nonumber
\end{align}
Let $K$ be the closed linear span of $\{(\sqrt{f^{\Theta^0}_n}x)^\bullet: x\in {C}^*(\hz_B)\}$. Then $K$ is a subspace of $L_2(L(\hz_B),\tau)^\om$.
We consider a special representation of $C^*(\hz_B)$ on $K$ defined by
\[
w_{\Theta^0}(x) (\sqrt{f_n^{\Theta^0}}y)^\bullet = (\sqrt{f_n^{\Theta^0}}x y)^\bullet,\quad x,y \in C^*(\hz_B).
\]
Then by \eqref{ultragfn} we have
\[
\tau_\om([w_{\Theta^0}(x)(\sqrt{f_n^{\Theta^0}} y)^\bullet]^* [w_{\Theta^0}(x)(\sqrt{f_n^{\Theta^0}} y)^\bullet]) = \lim_{n\to\om}\tau(y^*x^* xy f_n^{\Theta^0}) = 0,
\]
for $x\in I_{\Theta^0}$. Thus $w_{\Theta^0}$ factors through $C_{\Theta^0}$: If we denote the quotient map by $q_{\Theta^0}: C^*(\hz_B)\to C_{\Theta^0}$ and define $v_{\Theta^0}(x+I_{\Theta^0})= w_{\Theta^0}(x)$, then $w_{\Theta^0}=v_{\Theta^0} q_{\Theta^0}$. We define a linear operator $\al: L_2(\ax_{2\Theta^0})\to K$ by
\[
\al(u_1^{k_1}\cdots u_d^{k_d})= (\sqrt{f^{\Theta^0}_n}\td\la(0,k))^\bullet.
\]
Note that $\al$ has a dense range and preserves the inner product by \eqref{ultragfn}. Then $\al$ is unitary. We define $\phi: \bx(K)\to \bx(L_2(\ax_{2\Theta^0}^d))$ by $\phi(x) = \al^*x \al$. One can directly check that $\phi[w_{\Theta^0}(\td\la(0,k))] = u_1^{k_1}\cdots u_d^{k_d}$. We define $\pi_{\Theta^0}=\phi\circ v_{\Theta^0}$. Then
\[
\pi_{\Theta^0}: C_{\Theta^0} \to \ax_{2\Theta^0}^d,\quad \pi_{\Theta^0}(\td \la(0, k) + I_{\Theta^0}) = u_1^{k_1}\cdots u_d^{k_d}.
\]
It follows that $\pi_{\Theta^0}\circ \si = \id$, and $\si,\pi_{\Theta^0}$ are trace-preserving isomorphisms. We can represent our argument here in a commutative diagram
\[
\xymatrix{
\ax_{2\Theta}^d \ar[r]^\si & C_\Theta \ar[dr]^{v_\Theta} \ar[r]^{\pi_{\Theta}}& \ax_{2\Theta}^d \ar@{^(->}[r]& \bx(L_2(\ax_{2\Theta}^d))\\
&C^*(\hz_B)\ar[u]^{q_\Theta} \ar[r]^{w_{\Theta}}  & \bx(K)\ar[ur]^\phi
}
\]

Now we prove the lower semi-continuity of $\Theta\mapsto \|q_\Theta(x)\|$ for $x\in C^*(\hz_B)$. First note that $\pi_\Theta\circ\sigma = \id$. Hence $\pi_\Theta$ is injective. The map
\[
\ax_{2\Theta}^d\hookrightarrow \bx(L_2(\ax_{2\Theta}^d))
\]
is also injective. Therefore, $v_\Theta$ is an isometry. Note that $\|q_\Theta(x)\|=\|v_\Theta[q_\Theta(x)]\|=\|w_{\Theta}(x)\|$. Let $x,\xi,\eta\in C^*(\hz_B)$. We may write
\[
x=\sum_{k\in \zz^m,\ j\in \zz^d} x_{kj} e^{2\pi \ii \lge k, \cdot\rge} \td\la(0,j)\in C^*(\hz_B),
 \]
\[
\xi= \sum_{k\in \zz^m,\ j\in \zz^d} a_{kj} e^{2\pi \ii \lge k, \cdot\rge} \td\la(0,j),\quad
 \eta= \sum_{k\in \zz^m,\ j\in \zz^d} b_{kj} e^{2\pi \ii \lge k, \cdot\rge} \td\la(0,j).
\]
We deduce from \eqref{ultragfn} that $\lge w_{\Theta}(x)(\sqrt{f_n^\Theta} \xi)^\bullet, (\sqrt{f_n^\Theta} \eta)^\bullet\rge$ is a continuous function of $\Theta$. Hence,
\begin{align*}
|\lge w_{\Theta}(x)(\sqrt{f_n^\Theta} \xi)^\bullet, (\sqrt{f_n^\Theta} \eta)^\bullet\rge| & = \liminf_{\Xi\to \Theta} |\lge w_{\Xi}(x)(\sqrt{f_n^\Xi} \xi)^\bullet, (\sqrt{f_n^\Xi} \eta)^\bullet\rge|\\
&\le  \liminf_{\Xi\to \Theta} \|w_{\Xi}(x)\| \|(\sqrt{f_n^\Xi} \xi)^\bullet\| \|(\sqrt{f_n^\Xi} \eta)^\bullet\|.
\end{align*}
Note that
\begin{align*}
  \|q_\Theta(x)\| = \sup\{|\lge w_{\Theta}(x)(\sqrt{f_n^\Theta} \xi)^\bullet, (\sqrt{f_n^\Theta} \eta)^\bullet\rge|: \|(\sqrt{f_n^\Theta} \xi)^\bullet\|\le 1,\|(\sqrt{f_n^\Theta} \eta)^\bullet\|\le 1\}.
\end{align*}
It follows that $\|q_\Theta(x)\|\le \liminf_{\Xi\to\Theta} \|w_{\Xi}(x)\|$ and the proof is complete.
\end{proof}

Let us come back to the study of two-dimensional noncommutative tori. Since we are now working with a family of noncommutative tori, we need to associate a parameter $\ta$ to our map $\si: \text{Poly}(x,y)\to \ax_\ta$ as in \eqref{e:siTa}. Namely, we consider the following linear map
\begin{align}\label{e:sita}
\si_\ta: {\rm{Poly}}_\vta (x,y)\to \ax_{\theta},\quad a_{jk}  x^j y^k \mapsto a_{jk}(\ta) u_\ta^j v_\ta^k, \quad j,k\in\zz,
\end{align}
where $u_\ta,v_\ta$ are the generators of $\ax_\ta$. Recall the map $\rho_n^\ta$ as defined in \eqref{rhonta}. For consistency, we understand $\rho_n^\ta = \rho_n$ if $\ta=0$, where $\rho_n$ was defined in \eqref{rhon}. Then Lemma \ref{ct2cts}, Lemma \ref{lower d-dim} and \eqref{e:rhonkta} immediately imply the following:

\begin{lemma}\label{atact1}
For any $\theta_0\in[0,1)$ we have
\[
 \lim_{\substack{\theta\to \theta_0\\ \theta \in \qz}}\lim_{j\to \8}\|\rho^{\theta}_{n_j(\theta)}(\si_\ta(f))\|_{M_{n_j(\ta)}}=\|\si_{\ta_0}(f)\|_{\ax_{\theta_0}}
\]
for any $f\in \text{Poly}_\vta(x,y)$. Here $n_j(\ta)$ is chosen according to Lemma \ref{rhotank} for any given rational $\ta$.
\end{lemma}

We remark here if $\ta_0$ is a rational number, we simply choose $\ta=\ta_0$ in Lemma \ref{atact1} and the result degenerates to Lemma \ref{ct2cts} (for $\ta=0$) and \eqref{e:rhonkta}. We also need to show that the same result as above holds for the Lip-norm. The first step is to show the Lip-norm is continuous in $\ta$. This can be done by constructing suitable derivations as in \eqref{derext} for general $d$-dimensional noncommutative tori. Here we give two concrete cases in dimension 2 to motivate the discussion.

\emph{Case 1: Poisson semigroup.}
Consider the Hilbert $\ax_\ta$-module $\mathcal{H}=(\ell_2(\zz)\oplus \ell_2(\zz) )\otimes_{} \ax_\ta$. Let $(e_i)_{i\in \zz}$ be the natural unit vectors of $\ell_2(\zz)$ and $h_k= \sum_{ik> 0, |i|\leq |k| }e_i$ for $k\in\zz$. Here $h_k$ is the sum of unit vectors associated to the integers in $[1,k]$ if $k\ge 0$ and in $[-k,-1]$ if $k<0$.  Recall the Gromov form, denoted by $K_{\zz,P}$ here, associated to the word length on $\zz$ which was defined in Section \ref{s:cnl}. Then by elementary calculation,
\begin{equation*}
K_{\zz,P}(j,k) = \frac12(|j|+|k|-|j-k|) = \begin{cases}
 \min\{|j|, |k| \} & jk\ge 0,\\
0& \text{otherwise},
\end{cases}
\end{equation*}
which coincides with $\lge h_j, h_k\rge_{\ell_2(\zz)}$.
We define a derivation $\delta$ on $\ax_\ta^\8$ by
\[
\delta(u_{\theta}^jv_{\theta}^k)=(h_j\oplus h_k)\otimes u_{\theta}^jv_{\theta}^k.
\]
Then we have
\begin{align*}
\langle \delta(u_{\theta}^jv_{\theta}^k), \delta(u_{\theta}^{j'}v_{\theta}^{k'})\rangle_{\ax_\ta}&=\langle h_j\oplus h_k, h_{j'}\oplus h_{k'}\rangle (u_{\theta}^jv_{\theta}^k)^*(u_{\theta}^{j'}v_{\theta}^{k'})\\
&=(\langle h_j, h_{j'} \rangle+\langle h_k, h_{k'} \rangle)(u_{\theta}^jv_{\theta}^k)^*(u_{\theta}^{j'}v_{\theta}^{k'}),
\end{align*}
showing that
\[
\Gamma(u_{\theta}^jv_{\theta}^k,u_{\theta}^{j'}v_{\theta}^{k'}) =K((j,k),(j',k'))(u_{\theta}^jv_{\theta}^k)^*u_{\theta}^{j'}v_{\theta}^{k'} =\langle \delta(u_{\theta}^jv_{\theta}^k), \delta(u_{\theta}^{j'}v_{\theta}^{k'})\rangle_{\ax_\ta},
\]
where we can directly check $K((j,k),(j',k'))=K_{\zz,P}(j,j')+K_{\zz,P}(k,k')$.

\emph{Case 2: Heat semigroup.} Consider the Hilbert $\ax_\ta$-module $\mathcal{H}=\mathbb{R}^2\otimes \ax_\ta$ and define a derivation $\delta$ on $\ax_\ta^\8$ by
\[
\delta(u_{\theta}^jv_{\theta}^k)=(j,k)\otimes u_{\theta}^jv_{\theta}^k.
\]
Therefore, we get
\begin{align*}
\langle \delta(u_{\theta}^jv_{\theta}^k), \delta(u_{\theta}^{j'}v_{\theta}^{k'})\rangle_{\ax_\ta}&=\langle (j,k)\otimes u_{\theta}^jv_{\theta}^k, (j',k')\otimes u_{\theta}^{j'}v_{\theta}^{k'}\rangle\\
&= \langle (j,k), (j',k')\rangle (u_{\theta}^jv_{\theta}^k)^*(u_{\theta}^{j'}v_{\theta}^{k'})\\
&=(jj'+kk') (u_{\theta}^jv_{\theta}^k)^*(u_{\theta}^{j'}v_{\theta}^{k'}).
\end{align*}
On the other hand, we have
\begin{align*}
\Gamma(u_\ta^jv_\ta^k ,u_\ta^{j'}v_\ta^{k'})&=\frac{1}{2}[(j^2+k^2)+(j')^2+(k')^2-(j-j')^2-(k-k')^2] (u_{\theta}^jv_{\theta}^k)^*u_{\theta}^{j'}v_{\theta}^{k'}\\
&=(jj'+kk')(u_{\theta}^jv_{\theta}^k)^*u_{\theta}^{j'}v_{\theta}^{k'}=\langle \delta(u_{\theta}^jv_{\theta}^k), \delta(u_{\theta}^{j'}v_{\theta}^{k'})\rangle_{\ax_\ta}.
\end{align*}

Note that in both cases $\rz^d$ and $\oplus_{i=1}^d\ell_2(\zz)$ embed into the column space $\ell_2^c$, we may take $\hx=\ell_2^c\otimes_{\min}\ax_{\ta}$. Let $p(x,y)=\sum_{j,k}a_{jk} x^jy^k\in \text{Poly}_\vta(x,y)$ be a noncommutative Laurent polynomial with continuous coefficients. Then by Lemma \ref{lower d-dim},
\begin{align*}
  \lim_{\theta'\to \theta}\|\delta (\si_{\ta'}(p))\|_{\hx}  &=\lim_{\ta'\to \ta} \Big\|\sum_{j,k} a_{jk}(\ta')\xi_{jk}\otimes u_{\theta'}^jv_{\theta'}^k\Big\|_{\hx}\\
  &= \lim_{\ta'\to\ta} \Big\|\sum_{j,k,j',k'} \bar a_{j'k'}(\ta')a_{jk}(\ta') \lge \xi_{j'k'},\xi_{jk}\rge (u_{\theta'}^{j'} v_{\theta'}^{k'})^* u_{\theta'}^jv_{\theta'}^k \Big\|_{\ax_\ta}^{1/2}\\
  &=\Big\|\sum_{j,k} a_{jk}(\ta)\xi_{jk}\otimes u_{\theta}^jv_{\theta}^k \Big\|_{\hx}=\|\delta (\si_\ta (p)) \|_{\hx}
 \end{align*}
for some $\xi_{jk}\in \ell_2^c$ independent of $\theta$.

To state the same result for general noncommutative $d$-tori, let us  fix a conditionally negative length function $\psi$ on $\zz^d$. The typical choice of $\psi$ is $\psi(k_1,...,k_d)=\sum_{i=1}^d k_i^2$ or $\psi(k_1,...,k_d)=\sum_{i=1}^d |k_i|$ as explained in Section \ref{s:cnl}.   Let $\Ga$ be the gradient form defined on $\ax_\Theta^\8$ associated to $\psi$ as in Section \ref{anaest}; see the discussion around \eqref{derext}.  Similar to the two-dimensional case, Lemma \ref{lower d-dim} and \eqref{e:gade} imply the following:
\begin{lemma}\label{atact2}
For any $p \in\text{Poly}_\vta(x_1,...,x_d)$, the function
\[
\Theta\mapsto \|\delta( \si_\Theta(p) )\|_{H_\psi^c\otimes_{\min} \ax_\Theta}=\|\Ga(\si_{\Theta}(p) , \si_\Theta (p) )^{1/2}\|_{\ax_\Theta}
\]
is continuous. Here $H_\psi^c$ is the column operator space of the Hilbert space $H_\psi$ and the latter is determined by $\psi$ as explained in Section \ref{anaest}.
\end{lemma}

%

We denote by $\sx$ a set of continuous sections of the continuous field of order unit spaces over $\tz^{d(d-1)/2}$ with fibers $(\ax_\Theta)_{sa}$. We will see that we may choose $\sx = \{( \si_{\Theta}(f)+\si_\Theta(f)^*)_{\Theta\in \tz^{d(d-1)/2}}: f\in \text{Poly}_\vta(x_1,...,x_d)\}$.
 Let $(\ax^\8_\Theta)_{sa} = \ax^\8_\Theta \cap (\ax_\Theta)_{sa}$. For $f_\Theta\in \ax_\Theta^\8$, we define $L_\Theta(f_\Theta) = \|\Ga(f_\Theta,f_\Theta)^{1/2}\|_{\ax_\Theta}$. 
Given $f\in \text{Poly}_\vta(x_1,...,x_d)$, we can always find an $f^{sa}\in \text{Poly}_\vta(x_1,...,x_d)$ such that $\si_{\Theta}(f^{sa})= \si_{\Theta}(f)+\si_\Theta(f)^*$ for \emph{all} $\Theta$. Then it follows from Lemmas \ref{lower d-dim} and \ref{atact2} that the functions $\Theta\mapsto \|\si_{\Theta}(f)+\si_\Theta(f)^* \|_{\ax_\Theta}= \| \si_{\Theta}(f^{sa})\|_{\ax_\Theta}$ and $\Theta\mapsto L_\Theta(\si_\Theta(f)+\si_\Theta(f)^*)=L_\Theta(\si_{\Theta}(f^{sa}))$ are both continuous. Therefore, $(\si_\Theta(f)+\si_\Theta(f)^*)_{\Theta\in \tz^m}$ is a continuous section for any $f\in{\rm Poly}_{\vta}(x_1,...,x_d)$. Note that the real subspace $\{\si_{\Theta}(f)+\si_\Theta(f)^*: f\in \text{Poly}_\vta(x_1,...,x_d)\}$ is dense in $(\ax_\Theta)_{sa}$. Using the same argument as for Proposition \ref{cfqms}, we conclude the following:

\begin{prop}\label{ctata2}
 $(\{(\ax_{\Theta}^\8)_{sa}, L_\Theta\}_{\Theta\in \tz^{d(d-1)/2}}, \sx)$ forms a continuous field of compact quantum metric spaces. 
\end{prop}

\subsection{Matrix algebras converge to rotation algebras}

Let $u_{\theta}, v_{\theta}$ be the generators of $\ax_{\theta}$. We recall the linear map
$
\sigma_\ta: {\rm{Poly}}_\vta(x,y)\to \ax_{\theta}
$
as in \eqref{e:sita}. Lemma \ref{atact1} suggests the construction of a continuous field of order-unit spaces. To this end, we need to  find sequences $(\ta_j)\subset \qz\cap [0,1)$ and $(n_j)\subset \nz$ such that $\ta_j\to \ta$ and $n_j\nearrow \8$ as $j\to \8$. If $\ta$ is rational, we take $\ta_j\equiv \ta$ and choose $n_j$ as in Proposition \ref{ctfata}. Suppose now we fix the irrational $\ta$. Let $(\ta_j)$ be a sequence of rational numbers such that $\lim_{j\to \8}\ta_j = \ta$. Let $p\in {\rm Poly}(x,y)$. Here we consider only noncommutative polynomials with constant coefficients to be consistent with Section \ref{s:prop}, although the following argument is also true for elements in ${\rm Poly}_\vta(x,y)$. In particular, Lemmas \ref{atact1} and \ref{atact2} still hold for all elements of ${\rm Poly}(x,y)$. Since there exists $p^{sa}\in {\rm Poly}_\vta(x,y)$ such that $\si_{\ta_j}(p^{sa}) = \si_{\ta_j}(p)+\si_{\ta_j}(p)^*$, by the discussion before Proposition \ref{ctata2}, we have for all $p\in \text{Poly}(x,y)$
\begin{align*}
\lim_{j\to \8} &\|\si_{\ta_j}(p)+\si_{\ta_j}(p)^*\|_{\ax_{\ta_j}} = \|\si_\ta(p)+\si_\ta(p)^*\|_{\ax_\ta}, \\
\lim_{j\to\8} &\|\Ga (\si_{\ta_j}(p)+\si_{\ta_j}(p)^*, \si_{\ta_j}(p)+\si_{\ta_j}(p)^*)^{1/2} \|_{\ax_{\ta_j}} \\
&= \|\Ga(\si_\ta(p)+\si_\ta(p)^*, \si_\ta(p)+\si_\ta(p)^*)^{1/2}\|_{\ax_\ta}.
\end{align*}
On the other hand, using \eqref{e:pinkf} and \eqref{e:gankpif} and a standard compactness argument, we may choose an increasing sequence of integers $n_j = n_j(\ta_j)$ such that for all $p=\sum_{|k|\le j, |l|\le j} a_{kl} x^k y^l \in \text{Poly}(x,y)$
\begin{align}
&\left|\|\rho_{n_j}^{\ta_j}(\si_{\ta_j}(p))+ \rho_{n_j}^{\ta_j}(\si_{\ta_j}(p))^*\|_{M_{n_j}} - \|\si_{\ta_j}(p)+\si_{\ta_j}(p)^*\|_{\ax_{\ta_j}} \right|\le \frac1j \|\si_{\ta_j}(p)\|_{\ax_{\ta_j}}, \label{e:521}\\
\Big|&\|\Ga^{n_j}[\rho_{n_j}^{\ta_j}(\si_{\ta_j}(p))+ \rho_{n_j}^{\ta_j}(\si_{\ta_j}(p))^*, \rho_{n_j}^{\ta_j}(\si_{\ta_j}(p))+ \rho_{n_j}^{\ta_j}(\si_{\ta_j}(p))^*]^{1/2}\|_{M_{n_j}}\label{e:522} \\
&\qquad- \|\Ga(\si_{\ta_j}(p)+\si_{\ta_j}(p)^* , \si_{\ta_j}(p)+\si_{\ta_j}(p)^*)^{1/2}\|_{\ax_{\ta_j}}\Big| \le \frac1j\|\Ga(\si_{\ta_j}(p) , \si_{\ta_j}(p))^{1/2}\|_{\ax_{\ta_j}}.\notag
\end{align}
Indeed, since $j$ is fixed, $p$ is in a vector space with dimension at most $(2j+1)^2$ and we have
\[
c_j\|\si_{\ta_j}(p)\|_{\ax_{\ta_j}}\le \|\si_{\ta_j}(p)\|_{L_2(\rx_{\ta_j})} \le \|\si_{\ta_j}(p)\|_{\ax_{\ta_j}},
\] 
\begin{align*}
c_j\|\rho_{n_j}^{\ta_j}(\si_{\ta_j}(p))\|_{M_{n_j}} &\le \|\rho_{n_j}^{\ta_j}(\si_{\ta_j}(p)) \|_{L_2(M_{n_j})} \le \|\rho_{n_j}^{\ta_j}(\si_{\ta_j}(p)) \|_{M_{n_j}}
\end{align*}
for some $0<c_j<1$. By compactness, we may choose polynomials $p_1,...,p_m$ with $\|\si_{\ta_j}(p_i)\|_{\ax_{\ta_j}}\le 1$ for all $i=1,...,m$ such that for any $p=\sum_{|k|\le j, |l|\le j} a_{kl} x^k y^l$ with $\|\si_{\ta_j}(p)\|_{\ax_{\ta_j}}\le 1$ we may find $p_\ell$ so that
\[
\|\si_{\ta_j}(p_\ell) - \si_{\ta_j}(p) \|_{L_2(\rx_{\ta_j})}\le \frac{c_j}{10j}. 
\]
Now using \eqref{e:pinkf} and \eqref{e:gankpif} we choose $N_j$ such that for all $n_j\ge N_j$,
\[
\Big|\|\rho_{n_j}^{\ta_j}(\si_{\ta_j}(p_i))+ \rho_{n_j}^{\ta_j}(\si_{\ta_j}(p_i))^*\|_{M_{n_j}} - \|\si_{\ta_j}(p_i)+\si_{\ta_j}(p_i)^*\|_{\ax_{\ta_j}} \Big|\le \frac{1}{10j} 
\] 
for all $i=1,...,m$. Note that by orthogonality $\|\rho_{n_j}^{\ta_j}(\si_{\ta_j}(p)) \|_{L_2(M_{n_j})} = \|\si_{\ta_j}(p)\|_{L_2(\rx_{\ta_j})}$ for $n_j>2j$. Using the triangle inequality together with the equivalence of norms for finite dimensional Banach spaces, for any $p=\sum_{|k|\le j, |l|\le j} a_{kl} x^k y^l$ with $\|\si_{\ta_j}(p)\|_{\ax_{\ta_j}}\le 1$ we may find $p_\ell$ so that
\begin{align*}
&\Big|\|\rho_{n_j}^{\ta_j}(\si_{\ta_j}(p))+ \rho_{n_j}^{\ta_j}(\si_{\ta_j}(p))^*\|_{M_{n_j}} - \|\si_{\ta_j}(p)+\si_{\ta_j}(p)^*\|_{\ax_{\ta_j}} \Big|\\
&\le  \|\rho_{n_j}^{\ta_j}(\si_{\ta_j}(p))+\rho_{n_j}^{\ta_j}(\si_{\ta_j}(p))^* - [\rho_{n_j}^{\ta_j}(\si_{\ta_j}(p_\ell))+\rho_{n_j}^{\ta_j}(\si_{\ta_j}(p_\ell))^*] \|_{M_{n_j}} \\
&\quad + \Big|\|\rho_{n_j}^{\ta_j}(\si_{\ta_j}(p_\ell))+ \rho_{n_j}^{\ta_j}(\si_{\ta_j}(p_\ell))^*\|_{M_{n_j}} - \|\si_{\ta_j}(p_\ell)+\si_{\ta_j}(p_\ell)^*\|_{\ax_{\ta_j}} \Big| \\
&\qquad + \|\si_{\ta_j}(p_\ell)+\si_{\ta_j}(p_\ell)^* - [\si_{\ta_j}(p)+\si_{\ta_j}(p)^*] \|_{\ax_{\ta_j}}\\
&\le \frac1{c_j}\Big[  \|\rho_{n_j}^{\ta_j}(\si_{\ta_j}(p_\ell-p))\|_{L_2(M_{n_j})} +\|\rho_{n_j}^{\ta_j}(\si_{\ta_j}(p_\ell-p))^* \|_{L_2(M_{n_j})} \Big] + \frac1{10j} \\
&\quad+\frac{1}{c_j}\Big[ \|\si_{\ta_j}(p_\ell-p)\|_{L_2(\rx_{\ta_j})} +\|\si_{\ta_j}(p_\ell-p)^* \|_{L_2(\rx_{\ta_j})} \Big]\\
&\le \frac1{c_j}\frac{2c_j}{10j}+ \frac{1}{10j} +\frac1{c_j}\frac{2c_j}{10j} \le \frac{1}{2j}.
\end{align*}
It follows that for any $p=\sum_{|k|\le j, |l|\le j} a_{kl} x^k y^l$,
\[
\Big|\|\rho_{n_j}^{\ta_j}(\si_{\ta_j}(p))+ \rho_{n_j}^{\ta_j}(\si_{\ta_j}(p))^*\|_{M_{n_j}} - \|\si_{\ta_j}(p)+\si_{\ta_j}(p)^*\|_{\ax_{\ta_j}} \Big|\le \frac{1}{2j} \|\si_{\ta_j}(p)\|_{\ax_{\ta_j}} .
\]
We conclude \eqref{e:521} for $n_j\ge N_j$. The proof of \eqref{e:522} is similar and relies on the fact that  in our setting $\|\Ga(f^*,f^*)^{1/2}\|_{L_2(\rx_{\ta_j})} = \|\Ga(f,f)^{1/2}\|_{L_2(\rx_{\ta_j})}$. Note that $\opnorm{f}_B=\|\Ga(f,f)^{1/2}\|_B$ is a semi-norm where $B=M_{n_j}$ or $\ax_{\ta_j}$. We may consider only polynomials without the constant term. Using the triangle inequality, we have
\begin{align*}
&\quad \Big| \opnorm{\rho_{n_j}^{\ta_j}(\si_{\ta_j}(p))+ \rho_{n_j}^{\ta_j}(\si_{\ta_j}(p))^*}_{M_{n_j}} -\opnorm{\si_{\ta_j}(p)+\si_{\ta_j}(p)^*}_{\ax_{\ta_j}}\Big|\\
&\le \opnorm{\rho_{n_j}^{\ta_j}(\si_{\ta_j}(p-p_\ell))+ \rho_{n_j}^{\ta_j}(\si_{\ta_j}(p-p_\ell))^* }_{M_{n_j}}+\opnorm{\si_{\ta_j}(p-p_\ell)+\si_{\ta_j}(p-p_\ell)^*}_{\ax_{\ta_j}}\\
&\qquad  + \Big| \opnorm{\rho_{n_j}^{\ta_j}(\si_{\ta_j}(p_\ell))+ \rho_{n_j}^{\ta_j}(\si_{\ta_j}(p_\ell))^*}_{M_{n_j}} -\opnorm{\si_{\ta_j}(p_\ell)+\si_{\ta_j}(p_\ell)^*}_{\ax_{\ta_j}}\Big|.
\end{align*}
Choosing $n_j$ large enough such that $\psi_{n_j}(k)\le2\psi(k)$ for $|k|\le j$, we find
\begin{align*}
&c_j\opnorm{\rho_{n_j}^{\ta_j}(\si_{\ta_j}(p-p_\ell))+ \rho_{n_j}^{\ta_j}(\si_{\ta_j}(p-p_\ell))^* }_{M_{n_j}} \le 2 \|\Ga^{n_j}[\rho_{n_j}^{\ta_j}(\si_{\ta_j}(p-p_\ell)),\rho_{n_j}^{\ta_j}(\si_{\ta_j}(p-p_\ell))]^{1/2}\|_{L_2(M_{n_j})} \\
&\le 4 \|\Ga(\si_{\ta_j}(p-p_\ell), \si_{\ta_j}(p-p_\ell))^{1/2}\|_{L_2(\rx_{\ta_j})}.
\end{align*}
The rest of the argument is the same as that of \eqref{e:521}. By making $N_j$ larger if necessary, \eqref{e:522} holds for $n_j\ge N_j$. We remark that an alternative (and concise) way to prove \eqref{e:521} and \eqref{e:522} is to invoke Proposition \ref{cb iso} which will be proved in the next section. The argument we explain here may be conceptually more natural. Note that here $n_j$ depends on $j$ but not on $p$. Combining these facts together, we find for all $p\in\text{Poly}(x,y)$
\begin{align*}
\lim_{j\to\8}& \|\rho_{n_j}^{\ta_j}(\si_{\ta_j}(p))+\rho_{n_j}^{\ta_j}(\si_{\ta_j}(p))^*\|_{M_{n_j}} = \|\si_\ta(p)+\si_\ta(p)^*\|_{\ax_\ta},\\
\lim_{j\to\8}& \|\Ga^{n_j}[\rho_{n_j}^{\ta_j}(\si_{\ta_j}(p))+\rho_{n_j}^{\ta_j}(\si_{\ta_j}(p))^*, \rho_{n_j}^{\ta_j}(\si_{\ta_j}(p))+\rho_{n_j}^{\ta_j}(\si_{\ta_j}(p))^*]^{1/2} \|_{M_{n_j}}\\
& = \|\Ga(\si_\ta(p)+\si_\ta(p)^*,\si_\ta(p)+\si_\ta(p)^* )^{1/2}\|_{\ax_\ta}.
\end{align*}
Put $\ta_\8 = \ta$. Recall $\mx_{n_j} = (M_{n_j})_{sa}$ for $j\in \nz$, $M_\8=\ax_\ta$ and $\mx_\8 = \ax_\ta^\8 \cap (\ax_\ta)_{sa}$. Let $\sx$ denote a set of continuous sections of the continuous field of order-unit spaces over $\overline\nz$ with fibers $(M_{n_j})_{sa}$. We have shown that $j\mapsto \rho_{n_j}^{\ta_j}(\si_{\ta_j}(p))+\rho_{n_j}^{\ta_j}(\si_{\ta_j}(p))^*$  and $j\mapsto L_{n_j}(\rho_{n_j}^{\ta_j}(\si_{\ta_j}(p))+\rho_{n_j}^{\ta_j}(\si_{\ta_j}(p))^*)$ are continuous at $j\in\overline\nz$ for any $p\in {\rm Poly}(x,y)$. So we may choose $\sx = \{(\rho_{n_j}^{\ta_j}(\si_{\ta_j}(p))+\rho_{n_j}^{\ta_j}(\si_{\ta_j}(p))^*)_{j\in \overline\nz}: p\in \text{Poly}(x,y) \}$. Even though we only consider noncommutative polynomials with constant coefficients, the set $\{\rho_{n_j}^{\ta_j}(\si_{\ta_j}(p))+\rho_{n_j}^{\ta_j}(\si_{\ta_j}(p))^*: p\in {\rm Poly}(x,y) \}$ is still dense in $\mx_{n_j}$. We may use the same argument as for Proposition \ref{cfqms} to conclude the following:
\begin{lemma}\label{ctfata2}
  For any $\ta\in [0,1)$, there exists sequences $(\ta_j)\subset \qz \cap [0,1)$ and $(n_j)\subset \nz$ such that
  \begin{enumerate}
    \item[(i)] $\lim_{j\to\8} \ta_j = \ta$; {\rm (ii)}  $(n_j)_{j}$ is increasing to infinity;
    \item[(iii)] $(\{(\mx_{n_j}, L_{n_j}\}_{j\in \overline \nz}, \sx)$ is a continuous field of compact quantum metric spaces.
  \end{enumerate}
\end{lemma}

For our development in later sections, we state the following few formulas for general $d$-dimensional tori.
Note that any $f_\Theta \in \ax_\Theta^\8$ can be written in a \emph{unique} way as a finite sum
\begin{align}\label{e:ftauni}
  f_\Theta = \sum_{(k_1,...,k_d)\in\zz^d} a_{k_1,...,k_d} u_1(\Theta)^{k_1}\cdots u_d(\Theta)^{k_d}
\end{align}
where $a_{k_1,...,k_d}\in\cz$ and the generators on the right-hand side are arranged so that their subindices are in the increasing order.  Let us define
\begin{align*}
\text{Poly}(x_1, ...,x_{d})=\Big\{\sum_{(k_1,...,k_d)\in\zz^d} a_{k_1,...,k_d}  x_1^{k_1}\cdots x_d^{k_d}: a_{k_1,...,k_d} \in \cz \Big\},
\end{align*}
where the sum is over finite indices and $x_1,...,x_d$ are noncommuting variables. In other words, $\text{Poly}(x_1, ...,x_{d})$ is a subspace of ordered noncommutative Laurent polynomials in $d$ variables. We also define a linear map between vector spaces
\[
\si_\Theta: \text{Poly}(x_1, ...,x_{d}) \to \ax_\Theta^\8, \quad \sum_{k_1,...,k_d} a_{k_1,...,k_d} x_1^{k_1}\cdots x_d^{k_d} \mapsto \sum_{k_1,...,k_d} a_{k_1,...,k_d}  u_1(\Theta)^{k_1}\cdots u_d(\Theta)^{k_d}.
\]
Compared with \eqref{e:polyxd} and \eqref{e:siTa}, the elements in $\text{Poly}(x_1, ...,x_{d})$ have scalar coefficients and $\si_\Theta$ here is the restriction of $\si_\Theta$ in \eqref{e:siTa} to $\text{Poly}(x_1, ...,x_{d})$. We define the unique $f\in \text{Poly}(x_1,...,x_d)$ associated to $f_\Theta\in \ax_\Theta^\8$ by
\begin{align}\label{e:funi}
  f=\sum_{(k_1,...,k_d)\in\zz^d} a_{k_1,...,k_d}x_1^{k_1}\cdots x_d^{k_d}.
\end{align}
Note that $\si_\Theta(f) =f_\Theta$. Let us come back to the two-dimensional case. Since for any $y\in \ax_\ta^\8$ there exists a unique $p\in {\rm Poly}(x,y)$ in the sense of \eqref{e:ftauni} and \eqref{e:funi} such that $\si_\ta(p) = y$, we know $\si_\ta:  {\rm Poly}(x,y) \to \ax_\ta^\8$ is invertible. We may define the following map
\begin{align}\label{e:rhotacom}
  \rho_{n_j}^\ta = \rho_{n_j}^{\ta_j}\circ \si_{\ta_j}\circ \si_\ta^{-1}: \ax_\ta^\8&\xrightarrow{\si_\ta^{-1}} {\rm Poly}(x,y)\xrightarrow{\si_{\ta_j}}\ax_{\ta_j}^\8 \xrightarrow{\rho_{n_j}^{\ta_j}} M_{n_j}\\
   \si_\ta(p) &\mapsto \rho_{n_j}^{\ta_j} (\si_{\ta_j}(p)).\notag
\end{align}
Morally speaking, $\rho_{n_j}^\ta$ sends elements in $\ax_\ta^\8$ to matrix algebras by factoring through $\ax_{\ta_j}^{\8}$ where $\ta_j$ and $\ta$ are close. It is analogous to the map $\pi_n$ in the one-dimensional case. 

Let us fix a pair of generators $u_1(n),v_1(n)$ of $M_n$; see the discussion before \eqref{sgheatmn}. Write $u_j(n)=[u_1(n)]^j, v_k(n)=[v_1(n)]^k$. The following is an analog of Lemma \ref{eqtai} for $M_n$ and $\ax_\ta$.

\begin{lemma}\label{tailata}
Let $\eps>0$. Then there exist $k=k(\eps)$, $m=m(k)$, and multipliers $\phi_{k,\eta}^n$, $\eta\in(0, \frac{\eps}{2(2k+1)^2})$, on $M_n$ for $n> 2m$ (including $n=\8$) such that
\[
 \|x- T_{\phi^n_{k,\eta}}(x)\|_{M_n} \le\eps [\|x\|_2+ L_n(x)]
\]
for $n>2m$ (including $n=\8$). Here $ T_{\phi^n_{k,\eta}}$ is induced by $\td T_{\phi_{k,\eta}^n}$ as defined in \eqref{tphi}.
\end{lemma}
\begin{proof}
Let $k\in \nz$ be a large number which will be determined later.  We choose $m$ and $\vph_{k,\eta}^n$ on $\zz_n$ as in Lemma \ref{cbapz} for $n>2m$. Here we actually use the heat length function $\psi_n$ as defined by \eqref{cnl1} in Lemma \ref{cbapz}. But since $\psi_n(j)\sim j^2$ when $j$ is large and
\[
\#\{j: |j|_n^2\le k\}\le \#\{j: |j|_n\le k\},
\]
we may still choose $\eta\in(0, \frac{\eps}{2(2k+1)^2})$ and the conclusion of Lemma \ref{cbapz} remains valid.  Let $\phi_{k,\eta}^n(j,l) = \vph_{k,\eta}^n(j)\vph_{k,\eta}^n(l)$ for $(j,l)\in \zz_n^2$. Note that for the Fourier multiplier $\phi_{k,\eta}^n$,
\[
\|\td T_{\phi_{k,\eta}^n}\|_{\cb}\le \|T_{\vph_{k,\eta}^n}\|_{\cb}^2\le (1+\eps)^2.
\]
By Lemma \ref{trans}, we have
\[
\| T_{\phi^n_{k,\eta}}\|_{\cb}\le (1+\eps)^2.
\]
According to our choice of $\phi_{k,\eta}^n$, we have $\supp\phi_{k,\eta}^n \subset [-m,m]^2$. By choosing $\eta\le \eps/(2k+1)^2$ and using Lemma \ref{cbapg}, we have
\begin{equation}\label{phi1small2}
|\phi_{k,\eta}^n(j,l)-1|\le \frac{\eps}{(2k+1)^2},\quad (j,l)\in [-k,k]^2.
\end{equation}
Then for any $x=\sum_{|j|,|l|\le k} a_{j,l} u_j(n)v_l(n)$, we have
\begin{equation*}
  \| T_{\phi_{k,\eta}^n}(x) -x\|_{M_n}\le \sum_{|j|,|l|\le k} |a_{j,l}| |\phi_{k,\eta}^n(j,l)-1|\le  \|x\|_2\eps.
\end{equation*}
Since $\|P_k y\|_2\le \|y\|_2$ for any $y\in M_n$, $n>2m$ (including $n=\8$), we have
\begin{align}\label{phi1pk}
\|P_k(y-T_{\phi_{k,\eta}^n}(y))\|_{M_n}=\|P_k y - T_{\phi_{k,\eta}^n}(P_k y)\|_{M_n} \le \|y\|_2\eps.
\end{align}
Using Lemma \ref{ttbd1}, Corollary \ref{Pl}, and the boundedness of Riesz transforms
\begin{align*}
  \|A_n^{-\al-\bt} & A_n^{1/2}(1-P_k)(y-T_{\phi_{k,\eta}^n}(y))\|_\8 \\
   &\le c_\al \|A_n^{-\bt}A_n^{1/2} (1-P_k)(y-T_{\phi_{k,\eta}^n}(y))\|_p\\
  &\le c_\al C_p k^{-2\bt} \|A_n^{1/2}(y-T_{\phi_{k,\eta}^n}(y))\|_p\\
  &\le c_\al C_p K_p k^{-2\bt} (\| \Ga^n(y,y)^{1/2}\|_p+\|\Ga^n (T_{\phi_{k,\eta}^n}(y), T_{\phi_{k,\eta}^n}(y))^{1/2}\|_p),
\end{align*}
where $c_\al=\|A^{-\al}_n: L_p^0\to L_\8 \|$, $C_p k^{-2\bt}$ is the bound in Corollary \ref{Pl} and $K_p$ is the bound of the noncommutative Riesz transforms. Using Lemma \ref{cbga} and choosing $k$ large enough in the beginning, we have
\[
\|(1-P_k)(y-T_{\phi_{k,\eta}^n}(y))\|_\8 \le (2+2\eps+\eps^2)c_\al C_p K_p k^{-2\bt} \| \Ga^n(y,y)^{1/2}\|_\8 \le \eps L_n(y).
\]
The proof is complete.
\end{proof}


 Let $M_n^{\Lambda_m^2}$ (resp. $\ax_\ta^{\Lambda_m^2}$) denote the elements of $M_n$ (resp. $\ax_\ta$) which are linear combinations of $u_j(n)v_l(n)$ (resp. $u_\ta^j v_\ta^l$) for $(j,l)\in{\Lambda_m^2}$. Recall that $\mx_{n_j}=(M_{n_j})_{sa}, j\in\nz$, and $\mx_\8 = (\ax_\ta^\8)_{sa}$.
\begin{prop}\label{main2}
Let $\varepsilon>0$ and  $R\geq 0$. Then there exist $N$ and $p_1,...,p_r\in  {\rm Poly} (x,y)$ with the following properties:
\begin{enumerate}
\item[(i)]  $\si_{\ta}(p_j)\in\mathcal{D}_R(\mx_\8)$;
\item[(ii)] for any $j>N$ and any $y\in \mathcal{D}_R(\mx_{n_j})$, there exists $s\in\{1,...,r\}$ such that
\[
\Big\|y-\frac12[\rho^{\ta_j}_{n_j}(\si_{\ta_j}(p_s)) +\rho^{\ta_j}_{n_j}(\si_{\ta_j}(p_s))^* ]\Big\|_{M_{n_j}}\leq \varepsilon.
\]
Here $(n_j)$ are chosen as in Lemma \ref{ctfata2}.
\end{enumerate}
\end{prop}

\begin{proof}
The argument is similar to that of Proposition \ref{assum}. The case $R=0$ is trivial. Let $\min\{R,1\}\gg \varepsilon>0$ and $R> 0$ be given. We choose $m$ and $\phi_{k,\eta}^n$ as in Lemma \ref{tailata} for $n>2m$. We define
\[
B=\{y\in L_2^{\Lambda^2_m}(\ax_\ta) : y=y^*, \|y\|_{\ax_{\ta}}\leq R, \|\Gamma(y,y)^{1/2}\|_{\ax_\ta}\leq 1\}.
\]
Since $B\subset \ell_2([-m,m]^2)$, $B$ is pre-compact. Therefore, there exists an $\varepsilon$-net $\{y_1,...,y_r\} $ of $B$. Since $B$ is contained in the image of $\si_\ta: {\rm Poly}(x,y)\to \ax_\ta$,  we   obtain noncommutative Laurent polynomials $p_1,...,p_r\in {\rm Poly}(x,y)$ such that $\si_\ta(p_i) = y_i$ and $\si_{\ta}(p_i) \in B\subset \mathcal{D}_R (\mx_\8)$ for $i=1,...,r$. 

Recall the map $\rho_{n_j}^\ta$ as defined in $\eqref{e:rhotacom}$. Since $\rho_{n_j}^\ta|_{\ax_{\ta}^{\Lambda_m^2}}$ is injective for $n_j$ large, we introduce a locally defined map $s_{n_j,m}$ as follows
\begin{align*}
s_{n_j,m}: M_{n_j}^{\Lambda_m^2}&\to \ax_\ta^{\Lambda_m^2}, \quad x\mapsto (\rho_{n_j}^{\ta})^{-1}(x).
\end{align*}
Note that $T_{\phi_{k,\eta}^{n_j}}(y)$ is supported in $\Lambda_m^2$ for $y\in M_{n_j}$.
We define $\hat{y}=s_{n_j,m}(T_{\phi_{k,\eta}^{n_j}}(y))$. Then
\begin{equation}\label{rhot}
\rho_{n_j}^{\theta}(\hat{y})=T_{\phi_{k,\eta}^{n_j}}(y).
\end{equation}
By Proposition \ref{cb iso}\footnote{This proposition will be proved in the next section for stronger convergence results. We use it here for convenience. It is not hard to justify the first formula in \eqref{asisom} by observing that $\rho_{n_j}^{\ta}|_{\ax_\ta^{\Lambda_m^2}}$ is a $(1+\eps)$-isometry for $j$ large enough. However, we need $\rho_{n_j}^{\ta}|_{\ax_\ta^{\Lambda_m^2}}$ to be a $1+\eps$ cb-isometry to justify the second formula in  \eqref{asisom}. One can check that $\rho_{n_j}^{\ta}|_{\ax_\ta^{\Lambda_m^2}}$ and the map $\rho_{n_s}^\ta$ in Proposition \ref{cb iso} are the same by their constructions.  Alternatively, the inequalities in \eqref{asisom} may be proved using a compactness argument similar to that of \eqref{e:521} and \eqref{e:522}. }, we can choose $N>2m$ large enough so that for any $j>N$ and $y\in \dx_R(\mx_{n_j} )$
\begin{align}\label{asisom}
  (1+\eps)^{-1}\|\hat y\|_{\ax_\ta} \le &\|\rho_{n_j}^{\ta}(\hat y)\|_{M_{n_j}} \le (1+\eps) \|\hat y\|_{\ax_\ta},\\
  (1+\eps)^{-1}\|\Ga(\hat y,\hat y)\|_{\ax_\ta} \le &\|\Ga^{n_j}(\rho_{n_j}^{\ta}(\hat y), \rho_{n_j}^{\ta}(\hat y))\|_{M_{n_j}} \le (1+\eps) \|\Ga(\hat y,\hat y)\|_{\ax_\ta}.\nonumber
\end{align}
Hence, we have
\[
\|\hat y\|_{\ax_\ta}\le (1+\eps)\|\rho_{n_j}^{\theta}(\hat{y})\|_{M_{n_j}} =(1+\eps)\|T_{\phi_{k,\eta}^{n_j}}(y)\|_{M_{n_j}}\leq (1+\varepsilon)^3\|y\|_{M_{n_j}}\leq (1+\varepsilon)^3R
\]
and by Lemma \ref{cbga} and \eqref{rhot},
\[
\|\Ga(\hat y,\hat y)\|_{\ax_\ta} \le (1+\eps) \|\Ga^{n_j}(\rho_{n_j}^{\theta}(\hat{y}),\rho_{n_j}^{\theta} (\hat{y}))\|_{M_{n_j}} \le (1+\eps)^5\|\Ga^{n_j}(y,y)\|_{M_{n_j}}
\]
for all $y\in \dx_R(\mx_{n_j})$ and $j>N$. Since $\frac{1}{(1+\varepsilon)^3}\hat y\in B$, there exists $p_s\in \{p_1,...,p_r\}$ such that $\|\si_{\ta}(p_s)-\frac{1}{(1+\varepsilon)^3}\hat y\|\leq \varepsilon$. By \eqref{e:rhotacom}, \eqref{rhot} and \eqref{asisom}, we have for $j>N$
\[
\Big\|\rho_{n_j}^{\theta_j}(\si_{\ta_j}(p_s))-\frac{T_{\phi_{k,\eta}^{n_j}}(y)}{(1+\varepsilon)^3}\Big\|_{M_n}=\Big\|\rho_{n_j}^{\theta}(\si_{\ta}(p_s))-\frac{\rho_{n_j}^\ta (\hat y)}{(1+\varepsilon)^3}\Big\|_{M_{n_j}}\leq (1+\varepsilon)\varepsilon,
\]
because $\si_{\ta}(p_s)-\frac{\hat y}{(1+\eps)^3}\in \ax_\ta^{\Lambda_m^2}$.

Finally, for any $y\in \dx_R(\mx_{n_j})$ and $j> N$, we have
\begin{align*}
&\quad\| T_{\phi_{k,\eta}^{n_j}}(y)-\rho_{n_j}^{\ta_j}(\si_{\ta_j}(p_s))\|_{M_{n_j}}\\
&\le \Big\| T_{\phi_{k,\eta}^{n_j}}(y)- \frac1{(1+\eps)^3}T_{\phi_{k,\eta}^{n_j}}(y)\Big\|_{M_{n_j}}+ \Big\|\frac1{(1+\eps)^3}T_{\phi_{k,\eta}^{n_j}}(y)- \rho_{n_j}^{\theta_j}(\si_{\ta_j}(p_s))\Big\|_{M_{n_j}}\\
&\le (4R+2)\eps.
\end{align*}
By Lemma \ref{tailata}, we have
\begin{align*}
\|y-\rho_{n_j}^{\ta_j} (\si_{\ta_j}(p_s))\|_{M_{n_j}} &\le \|y-T_{\phi_{k,\eta}^{n_j}} (y)\|_{M_{n_j}}+\|T_{\phi_{k,\eta}^{n_j}} (y)- \rho_{n_j}^{\ta_j} (\si_{\ta_j}(p_s))\|_{M_{n_j}}\\
&\le (5R+3)\eps.
\end{align*}
So far our argument has not used the self-adjointness of $y$. Now let $y\in \dx_R(\mx_{n_j})$. We deduce from  $y=y^*$ that
\begin{align*}
&\Big\|y -\frac12[\rho_{n_j}^{\ta_j} (\si_{\ta_j}(p_s)) + \rho_{n_j}^{\ta_j} (\si_{\ta_j}(p_s))^*]\Big\|_{M_{n_j}}\\
&\le \frac12\| y-\rho_{n_j}^{\ta_j} (\si_{\ta_j}(p_s))\|_{M_{n_j}} +\frac12\|y^*-\rho_{n_j}^{\ta_j} (\si_{\ta_j}(p_s))^*\|_{M_{n_j}}\le (5R+3)\eps.
\end{align*}
Replacing $\eps$ by $\frac{\eps}{5R+3}$ in the beginning completes the proof.
\end{proof}

\begin{theorem}\label{final}
For $\ta\in[0,1)$, let $(\theta_j)$ and $(n_j)$ be as given by Lemma \ref{ctfata2}. Then $((M_{n_j})_{sa},L_{n_j})$ converges to $((\ax_\ta^\8)_{sa},L)$ in the quantum Gromov--Hausdorff distance.
\end{theorem}
\begin{proof}
In Lemma \ref{ctfata2} we proved that $(\{(\mx_{n_j}), L_{n_j}\}_{n\in\overline\nz},\sx)$ is a continuous field of compact quantum metric spaces in the sense of \cite{Li06}. Let $\eps=1/m$ and $R\geq 0$. By Proposition ~\ref{main2}, we can find $N\in \nz$ and
\[
y_1^m=\si_{\ta}(p_1^m),...,y_{r_m}^m=\si_{\ta}(p_{r_m}^m) \in\mathcal{D}_R((\ax_\ta^{\8})_{sa}),
\]
where $(p_{r_s}^m)_{s=1}^m\subset {\rm Poly}(x,y)$, so that for any $z\in\dx_R(\mx_{n_j})$, $j>N$, there exists a $p_{r_s}^m\in\{p_{r_1}^m,...,p_{r_m}^m\}$ with
\[
\Big\|z-\frac12[\rho_{n_j}^{\theta_j}(\si_{\ta_j}(p_{r_s}^m)) +\rho_{n_j}^{\theta_j}(\si_{\ta_j}(p_{r_s}^m))^* ]\Big\|_{M_{n_j}}\leq \varepsilon.
\]
Note that $\si_\ta(p_{i}^m)^* = \si_\ta(p_{i}^m)$ for all $i,m$. The set
\[
\Lambda:=\cup_{m=1}^\8\{y_1^m,..., y_{r_m}^m\} = \si_{\ta} ( \cup_{m=1}^\8\{p_1^m,..., p_{r_m}^m\})
\]
is dense in $\dx_R(\mx_\8)$. Give an ordering as in Theorem \ref{ctap}. By our construction, for any $\eps>0$, there exist $m$ and $r$ such that the open $\eps$-balls in $\mx_{n_j}$ centered at
\[
\frac12[\rho_{n_j}^{\theta_j}(\si_{\ta_j}(p_1))+\rho_{n_j}^{\theta_j}(\si_{\ta_j}(p_1))^* ], ..., \frac12[\rho_{n_j}^{\theta_j} (\si_{\ta_j}(p_r)) +\rho_{n_j}^{\theta_j} (\si_{\ta_j}(p_r))^*]
\]
cover $\dx_R(\mx_{n_j})$ for all $n_j>m$, where $\si_{\ta}(p_i) \in \Lambda$ for all $i=1,...,r$. In other words, $\Lambda$ satisfies condition (iii) in \cite{Li06}*{Theorem 7.1}. Hence $(M_{n_j})_{sa}$ converges to $(\ax_\ta^\8)_{sa}$ in the order-unit quantum Gromov--Hausdorff distance $d_{\rm oq}^R$ by the same theorem. But by Proposition \ref{statebd}, the radii of state spaces of $M_{n_j}$ are uniformly bounded. The assertion follows from \cite{Li06}*{Theorem 1.1}.
\end{proof}

So far we have dealt with the heat semigroup on $\ax_\ta$.  The following indicates that the approximation can also be done using the Poisson semigroup.

\begin{lemma}\label{h-p}
 Let $B_n$ denote the generator of the Poisson semigroup and $A_n$ denote the generator of the heat semigroup on $M_n$. Then we have
\[
\|A_n^{\beta/2}x\|_p\sim \|B_n^{\beta}x\|_p,
\]
for $1<p<\8$.
\end{lemma}

\begin{proof}
Note that for fixed $j,k$ such that $ |j|,|k|\leq \frac{n}{2}$, $j^2+k^2\sim (|j|_n+|k|_n)^2$, where $|\cdot|_n$ is as defined in Section \ref{s:cnl}. Let $p_0$ be such that $\frac{1}{p}=\frac{1-\theta}{p_0}+\frac{\theta}{2}$ for  $0<\theta<1$ and $\bt=\ta\al$.
Now since the maps
\[
B_n^{\alpha \ii t}A_n^{-2\alpha \ii t}:L_{p_0}\to L_{p_0},
\]
\[
B_n^{\alpha (1+\ii t)}A_n^{-2\alpha (1+\ii t)}:L_2\to L_2
\]
are bounded, the assertion follows from Stein's interpolation theorem in the same way as the proof of Proposition \ref{I}.
\end{proof}

\begin{rem}\label{exotic}
  Lemma \ref{ttbd1} has a variant for the Poisson semigroup. Together with Lemma \ref{h-p}, one can prove Proposition \ref{main2} for the Poisson semigroup, which in turn yields the approximation result. In fact, one can even prove the convergence in quantum Gromov--Hausdorff distance using some exotic semigroups. For example, one may consider the semigroup defined by $T_t(u^jv^k)=e^{-t(|j|+|k|^2)}u^jv^k$.
\end{rem}

\section{Completely bounded quantum Gromov--Hausdorff convergence}\label{cb qgh}

In this section we introduce the notion of completely bounded quantum Gromov--Hausdorff distance. The final goal is to show that the continuous fields of compact quantum metric spaces which we presented earlier in this paper converge in this sense.

\begin{defn}
Let $X$ be an operator space.  We say $(X, L)$ is a \emph{Lip operator space structure}, if
\begin{enumerate}
\item $L\subset X$ is a dense subspace;
\item there exists a subspace $N\subset L$ such that $L/N$ carries an additional operator space structure, which will also be referred to as Lip structure.
\end{enumerate}
By this definition, the completion of $L/N$ is an operator space; see Remark \ref{r:disthmr} for more discussion on the null space $N$.  In particular, on the first matrix level the Lip structure induces a semi-norm on $L$. The matrix semi-norms on $L$ will be denoted by $\opnorm{x}_{M_n(L)}$ or simply $\opnorm{x}$ if it is clear that $x\in M_n(X)$. Here we use the convention that $\opnorm{x}_{M_n(L)}=+\8$ for $x\in M_n(X)\setminus M_n(L)$. We also use the notation $L(x):=\opnorm{x}$, especially when we consider a continuous field of compact quantum metric spaces.  
\end{defn}

We will assume in this section that all operator spaces under consideration have Lip operator space structures. We define the completely bounded quantum Gromov--Hausdorff distance of two operator spaces as follows.

\begin{defn}
Let $X$ and $Y$ be two operator spaces. Let $R>0$ and
\[
\dx_R(M_n(X)) = \{ x\in M_n(X): \opnorm{x}_{M_n(L)}\le 1, \|x\|_{M_n(X)}\le R\}.
\]
We denote the \emph{R-cb-quantum Gromov--Hausdorff distance} of $X$ and $Y$ by $d^{cb}_{oq, R}(X,Y)$, and define it by
\begin{align*}
  d^{cb}_{oq, R}(X,Y)= \inf \sup_{n\in\nz} \{d_H[\id\otimes \iota_X(\dx_R(M_n(X))), &\id\otimes \iota_Y(\dx_R(M_n(Y)))]\},
\end{align*}
where $d_H$ denotes the Hausdorff distance, and the infimum runs over all operator spaces $V$ and completely isometric embeddings $\iota_X: X\to V$ and $\iota_Y: Y\to V$. If in addition $X$ and $Y$ are unital with units $e_X$ and $e_Y$, respectively, we modify the definition as follows:
\begin{align*}
  d^{cb}_{oq, R}(X,Y)= \inf \sup_{n\in\nz} \{\max\{d_H[\id\otimes \iota_X(\dx_R(M_n(X))), &\id\otimes \iota_Y(\dx_R(M_n(Y)))], \\
  &\|\iota_X(R e_X)-\iota_Y (R e_Y)\|\}\}.
\end{align*}
\end{defn}

\begin{rem}
The definition above seems stronger than the one introduced in \cite{Wu06}. 
\end{rem}

\begin{rem}
Let $\kx_1$ denote the unitization of $\kx$, the space of compact operators on the Hilbert space $\ell_2$. For two operator systems $X$ and $Y$, it may be more interesting to  consider $d^{cb}_{oq, R}(\kx_1\otimes X,\kx_1\otimes Y)$.
\end{rem}

Now we prove the triangle inequality. The proof follows the same idea as that of Lemma 4.5 in \cite{Li06}.

\begin{lemma}\label{tri1}
Let $\iota_j: A\to B_j$ be linear completely isometric embeddings of operator spaces for $j\in\{1,2\}$. Then there is an operator space $C$ and linear completely isometric embeddings $\psi_j: B_j\to C$ such that $\psi_1\circ\iota_1=\psi_2\circ\iota_2$.
\end{lemma}
\begin{proof}
Let $\psi_j$ be as defined in the proof of Lemma 4.5 in \cite{Li06}. The same argument extends easily to  matrix levels. Then $\psi_1\circ\iota_1=\psi_2\circ\iota_2$ and $\psi_1$, $\psi_2$ are complete isometries.
\end{proof}

\begin{lemma}\label{tri}
Let $X$, $Y$ and $Z$ be operator spaces. Then the following holds
\[
d^{cb}_{oq, R}(X,Z)\leq d^{cb}_{oq, R}(X,Y)+d^{cb}_{oq, R}(Y,Z).
\]
\end{lemma}

\begin{proof}
The triangle inequality follows immediately from applying Lemma \ref{tri1} with $A=Y$.
\end{proof}

Let $k\geq 0$. Recall the notation $\Lambda_k^2$ in \eqref{lambd2}. Let $x\in M_m(\ax_{\theta}^{\Lambda_k^2})$ and $\delta$ be the derivation of $\ax_{\theta}$ into a Hilbert C$^*$-module $\hx_{\ax_\ta}:=H_\psi\otimes \ax_\ta$ as defined in \eqref{derext} (for the case $m=1$). We extend the Lip-norm  to matrix levels as follows
\[
\opnorm{x}=\max\{\|(\id\otimes \delta)(x)\|_{M_m(\hx_{\ax_\ta})}, \|(\id\otimes \delta)(x^*)\|_{M_m(\hx_{\ax_\ta})}\}.
\]
This is exactly the definition \eqref{lipmat} in the two-dimensional  case restricted to rotation $\mathrm{C}^*$-algebras. Note that if $x$ is self-adjoint, the semi-norm $\opnorm{x}$ introduced here is just the matrix extension of the Lip-norm used in Proposition \ref{ctata2} for $d= 2$. We may write $L_\ta(x)$ or $L_\8(x)$ for $\opnorm{x}$ when considering continuous fields of compact quantum metric spaces.

\begin{lemma}\label{eps iso}
Let $X$ and $Y$ be two operator spaces. Let $0<\eps<\frac12$ and $\varphi: X\to Y$ be a $1+\eps$ cb-isometry and a $1+\eps$ Lip-isometry, i.e. for any $m$ and any $\hat{x}\in M_m(X)$, we have
\[
(1-\eps)\|\hat{x}\|_{M_m(X)}\leq \|(\id\otimes \varphi)(\hat{x})\|_{M_m(Y)}\leq (1+\eps)\|\hat{x}\|_{M_m(X)},
\]
and
\[
(1-\eps)\opnorm{\hat{x}}\leq \opnorm{(\id\otimes \varphi)(\hat{x})}\leq (1+\eps)\opnorm{\hat{x}}.
\]
Then we have
\[
d^{cb}_{oq, R}(X,\varphi(X))\leq 3R\eps.
\]
\end{lemma}
\begin{proof}
Let $N=\{(a, -\varphi(a), \eps a) : a\in X\}$. Then $N\subset X\oplus_1 Y\oplus_1 X$. Here $X\oplus_1 Y\oplus_1 X$ is the $\ell_1$-sum of $X$, $Y$ and $X$ in the sense of operator spaces. Let
\[
V=\{(x,\varphi(x'),0)+N: x, x'\in X\}\subset (X\oplus_1 Y \oplus_1 X)/N.
\]
 Then
\[
\|(x,y,0)+N\|_{(X\oplus_1 Y \oplus_1 X)/N}=\inf\{\|x-a\|+\|y+\varphi(a)\|+\eps\|a\| : a\in X\}.
\]
 Thus $X$ and $Y$ embed isometrically into $V$ (see Lemma 7.2 in \cite{Li06}). We claim that the embeddings are actually completely isometric. Indeed, since $S_1(X)$ is the projective tensor product $S_1\otimes^{\wedge} X$ (see \cite{Pi03}*{Page 142}),  we have $S_1((X\oplus_1 \varphi(X)\oplus_1 X)/N)=(S_1(X)\oplus_1 S_1(\varphi(X))\oplus_1 S_1(X))/S_1(N)$ and
\[
S_1(V)\subset S_1(X)\oplus_1 S_1(\varphi(X))\oplus_1 S_1(X)/S_1(N).
\]
Hence for $\hat{x}\in S_1(X)$, we have
\begin{align*}
\|(\hat{x},0,0)+S_1(N)\|_{S_1(V)}=&\|\hat{x}\|_{S_1(X)},\\
\|(0,(\id\otimes\varphi)\hat{x}, 0) +S_1(N)\|_{S_1(V)} =&\|(\id\otimes\varphi)\hat{x}\|_{S_1(Y)}.
\end{align*}
Note that by a result of Pisier (see \cite{Pi98}*{Lemma 1.7}), if $u: X\to Y$ is a completely bounded map, for every $1\leq p\leq \8$, we have
\[
\|u\|_{\cb}=\sup_m\|\id\otimes u: S_p^m(X)\to S_p^m(Y)\|.
\]
Therefore, by applying the above for $p=1$ and $p=\8$, we find that
\begin{align*}
\iota_1: X\to (X, 0, 0)+N \subset V \quad \text{and}\quad \iota_2: \varphi(X) \to (0, \varphi(X), 0)+N \subset V
\end{align*}
are completely isometric embeddings. Note that the maps
\begin{align*}
  \iota: X &\to X\oplus_1 Y\oplus_1 X,\quad x\mapsto (0,0,x),\\
  q: X\oplus_1 Y\oplus_1 X &\to (X\oplus_1 Y\oplus_1 X)/N,\quad (x,y,z)\mapsto (x,y,z)+N
\end{align*}
are completely contractive. For any $\hat{x}\in M_m(X)$, we have
\begin{align}\label{Mm ball}
&\quad \|(\id\otimes \iota_1)\hat{x} -\id\otimes(\iota_2\circ \varphi)\hat{x}\|_{M_m(V)}\leq \|(\hat x,0,0)-(0,\id\otimes \varphi(\hat x),0)+M_m(N)\|_{M_m(V)}\\
&=\|(\hat x, -\id\otimes \varphi(\hat x), 0)-(\hat x,-\id \otimes \varphi(\hat x),\eps \hat x)+M_m(N)\|_{M_m(V)}\nonumber
\le \eps \|\hat{x}\|_{M_m(X)}. \nonumber
\end{align}
Now let $\hat{x}\in \dx_R(M_m(X))$, i.e. $\hat{x}\in M_m(X)$, $\|\hat{x}\|_{M_m(X)}\leq R$ and $\opnorm{\hat{x}}\leq 1$.
Then by the assumption, we have $\|\frac{1}{1+\eps}(\id\otimes\varphi)\hat{x}\|_{M_m(Y)}\le R$ and $\opnorm{\frac{1}{1+\eps}(\id\otimes\varphi)\hat{x}}\leq 1$. This means $\frac{1}{1+\eps}(\id\otimes\varphi)\hat{x}\in \dx_R(M_m(Y))$. Using \eqref{Mm ball} and the triangle inequality, we have
\[
\|(\id\otimes \iota_1)\hat{x} -(\id\otimes\iota_2)[(1+\eps)^{-1}(\id\otimes \varphi)\hat{x}]\|_{M_m(V)}\le \eps R+ \frac{\eps R}{1+\eps}\|\varphi\|_{\cb}\le 2R\eps.
\]
On the other hand, since $\vph$ is a $1+\eps$ cb-isometry and a $1+\eps$ Lip-isometry, $\vph$ is invertible and for any $\hat y\in M_m(Y)$ we have
\[
(1+\eps)^{-1}\|\hat{y}\|_{M_m(Y)}\leq \|(\id\otimes \varphi^{-1})(\hat{y})\|_{M_m(X)}\leq (1-\eps)^{-1}\|\hat{y}\|_{M_m(Y)},
\]
\[
(1+\eps)^{-1}\opnorm{\hat{y}}\leq \opnorm{(\id\otimes \varphi^{-1})(\hat{y})}\leq (1-\eps)^{-1}\opnorm{\hat{y}}.
\]
It follows that $\vph^{-1}:Y\to X$ is a $1+2\eps$ cb-isometry and a $1+2\eps$ Lip-isometry. Repeating the above argument and swapping $X$ and Y, for any $\hat y\in \dx_R(M_m(Y))$, we have $(1-\eps) (\id\otimes\varphi^{-1})\hat{y}\in \dx_R(M_m(X))$ and
\[
\|(\id\otimes \iota_1)\hat{y} -(\id\otimes\iota_2)[(1-\eps)(\id\otimes \varphi^{-1})\hat{y}]\|_{M_m(V')}\le \eps R+ \eps R\|\varphi^{-1}\|_{\cb}\le 3R\eps
\]
for some $V'$ defined in the same way as $V$. Combining the two inequalities together, we find $ d^{cb}_{oq, R}(X,\varphi(X))\le   3R\eps.$
\end{proof}

\subsection{CB-continuous fields of compact quantum metric spaces} In this section we investigate an operator space version of continuous fields of compact quantum metric spaces, and show that the continuous fields of compact quantum metric spaces which we introduced earlier form cb-continuous fields of compact quantum metric spaces with appropriate operator space Lip-norms defined on them.

\begin{defn}
Let $T$ be a locally compact Hausdorff space and let $(\{\bar \ax_t\}_{t\in T}, \sx)$ be a continuous field of order-unit spaces in the sense of \cite{Li06}, where $\sx$ is the space of continuous sections containing the unit. We say $(\{\bar\ax_t\}_{t\in T}, \sx)$ is a \emph{cb-continuous field of order-unit spaces} if there exists a subspace of $\sx$, denoted by $\sx_0$, such that $\{f(t): f\in \sx_0\}$ is dense in $\bar\ax_t$ for all $t\in T$ and the following holds: For any finite subset $\Delta \subset S_0$, $s_0\in T$ and $\eps>0$, there exists a neighborhood $\ux(s_0)$, such that for any $s\in \ux(s_0)$, $m\ge 1$, $f\in \Delta$ and matrix coefficients $a_f\in M_m$, we have the following
\begin{align}\label{e:cbctsous}
\frac{1}{1+\eps}\|\sum_{f\in \Delta} a_f\otimes f(s)\|_{M_m(\bar\ax_s)}&\leq \|\sum_{f\in \Delta} a_f\otimes f(s_0)\|_{M_m(\bar\ax_{s_0})}\\
&\leq (1+\eps) \|\sum_{f\in \Delta} a_f\otimes f(s)\|_{M_m(\bar\ax_s)}.\notag
\end{align}
We call $(\{\ax_t, L_t\}_{t\in T}, \sx)$ a \emph{cb-continuous field of compact quantum metric spaces} if $(\{\ax_t, L_t\}_{t\in T}, \sx)$ is a continuous field of compact quantum metric spaces in the sense of \cite{Li06} and, in the setting of cb-continuous field of order-unit spaces, we have \eqref{e:cbctsous} and
\begin{align*}
\frac{1}{1+\eps}\opnorm{\sum_{f\in \Delta} a_f\otimes f(s)}_{M_m(\ax_s)}&\leq \opnorm{\sum_{f\in \Delta} a_f\otimes f(s_0)}_{M_m(\ax_{s_0})}\\
&\leq (1+\eps)\opnorm{\sum_{f\in \Delta} a_f\otimes f(s)}_{M_m(\ax_s)}.
\end{align*}
\end{defn}

Recall the map $\rho_n: \cz(\zz^2)\to M_n$ as defined in \eqref{rhon}: $\rho_n(u^jv^k)=u_j(n)v_k(n) $, where $u,v$ are the generators of $C(\tz^2)$ and $u_j(n)v_k(n)$ are defined in Section \ref{A_theta}. Let $C^{\Lambda_k^2}(\tz^2)$ (resp. $M_n^{\Lambda_k^2}$) denote the elements in $C(\tz^2)$ spanned by $u^jv^l$ (resp. $u_j(n)v_l(n)$) for $(j,l)\in\Lambda_k^2$.  Note that $C^{\Lambda_k^2}(\tz^2)$ and $M_n^{\Lambda_k^2}$ are operator spaces.

\begin{prop}\label{theta0 cb}
For any $\eps>0$ and $k\geq 0$, there exists $N>0$ such that for any $n>N$, the map $\rho_n|_{C^{\Lambda_k^2}(\tz^2)}: C^{\Lambda_k^2}(\tz^2)\to M_n^{\Lambda_k^2}$ is a $1+\eps$ cb-isometry.
\end{prop}

\begin{proof}
By Lemma \ref{ct2cts}, the map $(\rho_n)^{\bullet}: C(\tz^2)\to \prod_{\om} M_n$ is a faithful $^*$-homomorphism, and in particular a complete isometry. Therefore, for any $\eps>0$ and $k\ge 0$, there exists $N>0$, such that for $n>N$, $C^{\Lambda_k^2}(\tz^2)$ is $(1+\eps)$-isometric to $M_n^{\Lambda_k^2}$ via $\rho_n$; i.e. we have for scalar coefficients $a_{j,l}$
\begin{align}\label{e:ct2iso}
  (1-\eps) \Big\|\sum_{|j|,|l|\le k} a_{j,l}u^jv^l \Big\|_{C(\tz^2)} \le  \Big\|\sum_{|j|,|l|\le k} a_{j,l}u_j(n)v_l(n)\Big\|_{M_n}\le (1+\eps) \Big\|\sum_{|j|,|l|\le k} a_{j,l}u^jv^l \Big\|_{C(\tz^2)}.
\end{align}
We need to extend this inequality to matrix coefficients. Note that $\rho_n|_{C^{\Lambda_k^2}(\tz^2)}$ is invertible between vector spaces and its inverse has the target space $C^{\Lambda_k^2}(\tz^2)$. Since $C(\tz^2)$ is commutative, we have (see e.g. \cite{Pi03}*{Proposition 1.10})
\[
\|(\rho_n|_{C^{\Lambda_k^2}(\tz^2)})^{-1}: M_n^{\Lambda_k^2} \to  C^{\Lambda_k^2}(\tz^2)\|_{\cb} = \|(\rho_n|_{C^{\Lambda_k^2}(\tz^2)})^{-1}: M_n^{\Lambda_k^2} \to  C^{\Lambda_k^2}(\tz^2)\|\le \frac1{1-\eps}.
\]
It follows that for any $m$ and $a_{j,l}\in M_m$
\begin{align}\label{min Mn}
(1- \eps) \Big\|\sum_{|j|, |l|\le k} a_{j,l}\otimes u^jv^l \Big\|_{M_m(C(\tz^2))}\le \Big\|\sum_{|j|,|l|\le k} a_{j,l}\otimes u_j(n)v_l(n)\Big\|_{M_m(M_n)}  .
\end{align}
The upper estimate of \eqref{e:ct2iso} can be extended to matrix coefficients directly if we use the \emph{min-structure} of $M_n$, denoted by $\min(M_n)$, as $\min(M_n)$ is commutative; see e.g. \cite{Pi03}*{Chapter 3}. However, for the natural operator space structure we consider here, we have to use a result of Haagerup--R\o rdam \cite{HR}. By their theorem, there exists a Hilbert space $\hx$ and $u(\theta), v(\theta)\in \bx(\hx)$ such that the following hold:
\begin{enumerate}
\item For any $\theta$, $C^*(u(\theta), v(\theta))\simeq \ax_\ta$;
\item there is a constant $c>0$, such that for any $\ta,\theta'$, $\max\{\|u(\theta)-u(\theta')\|, \|v(\theta)-v(\theta')\|\} \leq c|\theta-\theta'|^{1/2}$.
\end{enumerate}
This implies that for $|j|\le k, |l|\le k$, we have
\[
\sup_{j,l}\|u^j(\theta)v^l(\theta)-u^j(\theta')v^l(\theta')\|\leq 2ck|\theta - \theta'|^{1/2}.
\]
Let $d_{\cb}$ denote the Banach-Mazur distance of two operator spaces. Then there exists $\delta=\delta(\eps, k)>0$ such that $d_{\cb}(\ax_\ta^{\Lambda_k^2}, \ax_{\ta'}^{\Lambda_k^2})< 1+\eps$ for any $|\theta-\ta'|<\delta$; see \cite{Pi03}*{Section 2.13}. We may find a completely bounded map $\phi$ sending $u^j(\ta)v^l(\ta)$ to $u^j(\ta')v^l(\ta')$ such that
\[
\|\phi: \ax_\ta^{\Lambda_k^2} \to \ax_{\ta'}^{\Lambda_k^2} \|_{\cb} \le 1+\eps.
\]
It follows that for all matrix coefficients $a_{j,l}$ we get
\[
\left|\Big\|\sum_{|j|,|l| \leq k} a_{j,l}\otimes u^j(\theta)v^l(\theta)\Big\| - \Big\|\sum_{|j|,|l| \leq k} a_{j,l}\otimes u^j(\theta')v^l(\theta')\Big\| \right|\leq \eps \Big\|\sum_{|j|,|l| \leq k}a_{j,l}\otimes u^j(\theta)v^l(\theta)\Big\|.
\]
Setting $\theta'=\frac{1}{n}<\delta$ and $\ta=0$, we have for any $m$
\begin{align}\label{plus}
\|\sum_{|j|,|l|\le k } a_{j,l}\otimes u^j(1/n)v^l(1/n)&\|_{M_m( \ax_{1/n})}
\leq (1+\eps)\|\sum_{|j|,|l|\le k} a_{j,l}\otimes u^jv^l\|_{M_m( C(\tz^2))}.
\end{align}
But $u_1(n)$ and $v_1(n)$ verify the commutation relation of $\ax_{1/n}$. By universality of $\ax_{1/n}$ we have for any $m$
\begin{align}\label{unia1n}
\|\sum_{|j|,|l|\le k } a_{j,l}\otimes u_j(n)v_l(n)&\|_{M_m( M_n)}\le \|\sum_{|j|,|l|\le k } a_{j,l}\otimes u^j(1/n)v^l(1/n)\|_{M_m( \ax_{1/n})}.
\end{align}
By combining the estimates (\ref{min Mn}), (\ref{plus}) and \eqref{unia1n}, we complete the proof.
\end{proof}

\begin{prop}\label{cb iso}
For any $\eps>0$ and $k\ge 0$, there exists $N>0$  and a family of maps $\rho_n^\ta:\ax_\ta^{\Lambda_k^2} \to M_n^{\Lambda_k^2}$ such that for $n>N$, $\rho_n^\ta$  is a $1+\eps$ cb-isometry and a $1+\eps$ Lip-isometry.
\end{prop}

\begin{proof}
Note that if we know $\rho_n^\ta$ is a $1+\eps$ cb-isometry then by the same argument as that of Lemma \ref{cbga}, it is also a $1+\eps$ Lip-isometry on $\ax_\ta^{\Lambda_k^2}$. Therefore it suffices to show that $\rho_n^\ta$ is a $1+\eps$ cb-isometry. If $\theta=0$, then the result follows immediately from Proposition \ref{theta0 cb}. Let $\theta=\frac{p}{q}$ be rational. Recall from Lemma \ref{rhotank} that we have a surjective map $\rho_{n_l}^{\theta}: \ax_\ta^{\Lambda_k^2}\to M_{n_l}^{\Lambda_k^2}$ for suitable $n_l$. We show that this map is a $1+\eps$ cb-isometry. As we observed in Section \ref{cts fields 1}, there is a trace-preserving $^*$-homomorphism $\si: \ax_\ta\to M_q\otimes_{\min} C(\tz^2)$. By Proposition \ref{theta0 cb}, there exists $N>0$ such that the map $\rho_m: C^{\Lambda_k^2}(\tz^2)\to M_m^{\Lambda_k^2}$ is a $1+\eps$ cb-isometry for $m>N$. Hence, so is the map $\id\otimes \rho_{n_l}: M_q\otimes_{\min} C^{\Lambda_k^2}(\tz^2)\to M_q\otimes_{\min} M_{n_l}^{\Lambda_k^2}$. Note that the specific choice of the subsequence $n_l=q^{l+1}$ guarantees that $\rho_{n_l}^\ta(\ax_\ta^\8)=M_{n_l}$. Therefore, the restriction of $\rho_{n_l}^\ta=(\id\otimes \rho_{n_l})\circ \vsi$ to $\ax_\ta^{\Lambda_k^2}$ is also a $1+\eps$ cb-isometry. This gives the following diagram
\[
\xymatrix{
   M_q\otimes_{\min} C^{\Lambda_k^2}(\tz^2)\ar [r]^{\quad\id\otimes \rho_{n_l}} & M_q\otimes_{\min} M_{n_l}^{\Lambda_k^2}\\
    \ax_\ta^{\Lambda_k^2}\ar[u]^\vsi \ar [r]^{\rho_{n_l}^{\theta}}& M_{n_l}^{\Lambda_k^2}\ar@ {^{(}->} [u]\\
  }
\]
which proves the rational case. Finally, let $\theta$ be irrational. Then there exists a sequence $\theta_s=\frac{p_s}{q_s}$ of rational numbers converging to $\theta$. We may assume $\ta_s-\ta$ is small enough so that we may apply the result of Haagerup--R{\o}rdam \cite{HR} to get a $1+\frac{\eps}3$ cb-isometry $\phi_s: \ax_\ta^{\Lambda_k^2}\to \ax_{\ta_s}^{\Lambda_k^2}$ in the same way as in the proof of Proposition \ref{theta0 cb}. Then by what we proved above, we may choose $n_s$ large enough such that the map $\rho_{n_{s}}^{\theta_s}: \ax_{\theta_s}^{\Lambda_k^2}\to M_{n_s}^{\Lambda_k^2}$ is a $1+\frac{\eps}3$ cb-isometry. Let $\rho_{n_s}^\ta = \rho_{n_s}^{\ta_s}\circ \phi_s$. Then $\rho_{n_s}^\ta: \ax_\ta^{\Lambda_k^2}\to M_{n_s}^{\Lambda_k^2}$ is a $1+\eps$ cb-isometry. We can illustrate the argument using the following diagram
\[
\xymatrix{
   \ax_\ta^{\Lambda_k^2}\ar [d]^{\phi_s} \ar@{.>}[dr]^{\rho_{n_s}^\ta}\\
   \ax_{\theta_s}^{\Lambda_k^2}\ar [d]^\si \ar [r]_{\rho_{n_{s}}^{\theta_s}}& M_{n_{s}}^{\Lambda_k^2}\ar@ {^{(}->} [d]\\
   M_{q_s}\otimes C^{\Lambda_k^2}(\tz^2)\ar [r]^{\quad\id\otimes \rho_{n_s}} & M_{q_s} \otimes M_{n_s}^{\Lambda_k^2}.\\
  }
\]
\end{proof}

Let $(\{\ax_n, L_n\}_{n\in\overline \nz}, \sx)$ denote either of the two continuous fields of compact quantum metric spaces which were introduced in Sections \ref{ct} and \ref{A_theta}. The following result follows immediately from Proposition \ref{cb iso}.

\begin{prop}\label{cb cts field}
$(\{\ax_n, L_n\}_{n\in\overline \nz}, \sx)$ is a cb-continuous field of compact quantum metric spaces.
\end{prop}

Strictly speaking, the image of self-adjoint elements under the map $\rho_n^\ta$ may not be self-adjoint thus may not lie in $(M_n)_{sa}$. However, this could be easily fixed by considering $\hat\rho^\ta_n(x):=\frac12[\rho_n^\ta(x) +\rho_n^\ta(x)^*]$ for $x\in \ax_\ta^{\Lambda_k^2} \cap (\ax_\ta)_{sa}$ as in the proof of Proposition \ref{main2}. This remark will be in force through the rest of this section: For simplicity, we will only use the map $\rho_n^\ta$ instead of $\hat \rho^\ta_n$. 

\subsection{Approximations for $C(\tz)$ and $\ax_\ta$}
Here we only present a formal proof of the approximation for $\ax_\ta$. The argument modifies easily to the case of $C(\tz)$. Before we prove the main result, we show the following estimate.

\begin{theorem}\label{cb compact}
Let $\eps>0$. Then there exist $k=k(\eps)$, $m=m(k)$ and multipliers $\phi_{k,\eta}^n$, $\eta\in(0,\frac{\eps}{4(2k+1)^2})$ on $M_n$ for $n> 2m$ (including $n=\8$) such that
\[
 \|T_{\phi_{k,\eta}^n}-\id: (M_n, \opnorm{\cdot}) \to (M_n, \|\cdot\|) \|_{\cb} \le\eps.
\]
Here $ T_{\phi^n_{k,\eta}}$ is induced by $\td T_{\phi_{k,\eta}^n}$ as defined in \eqref{tphi}.
\end{theorem}

\begin{proof}
We follow the proof of Lemma \ref{tailata}, but we have to get rid of the $L_2$ norm this time. Let $k$ be a large number which will be determined later. Fix $\alpha, \beta$ such that $\alpha+ \beta = \frac{1}{2}$.
Similar to Lemma \ref{tailata}, we may choose multipliers $\phi^n_{k,\eta}$, $\eta\in(0,\frac{\eps}{D(2k+1)^2})$ for some $D$ to be determined later, such that
\begin{align}\label{phi1small2}
|\phi_{k,\eta}^n(j,l)-1|\le \frac{\eps}{D (2k+1)^2},\quad (j,l)\in [-k,k]^2.
\end{align}
Note that
\begin{align*}
T_{\phi_{k,\eta}^n} -\id &= A^{-\alpha}A^{-\beta}(T_{\phi_{k,\eta}^n} - \id)A^{1/2} \\
&= A^{-\al} A^{-\bt}(T_{\phi_{k,\eta}^n} - \id)P_kA^{1/2} + A^{-\alpha}A^{-\beta}(T_{\phi_{k,\eta}^n} - \id)(\id - P_k)A^{1/2}.
\end{align*}
By Proposition \ref{cbriesz}, we know $\|A^{1/2}: (M_n,\opnorm{\cdot}) \to L_p^0(M_n)\|_{\cb} = K_p<\8$. Using \eqref{phi1small2} and Lemma \ref{qcbd}, we may extend \eqref{phi1pk} to matrix levels as in Lemma \ref{cbapg} (but with $q=p\ge 2$ here) and obtain
\[
\|(T_{\phi_{k,\eta}^n} -\id) P_k: L_p^0(M_n) \to L_p^0 (M_n)\|_{\cb} \le \frac{2\eps}D.
\]
By \eqref{lpbd}, we know $\|A^{-\bt}: L_p^0(M_n)\to L_p^0(M_n)\|_{\cb} =c'_\bt<\8$. And by Lemma \ref{ttbd1}, $\|A^{-\al}: L_p^0(M_n)\to L_\8^0(M_n)\|_{\cb}= c_\al <\8$. Therefore, we find
\begin{align*}
&\|A^{-\al} A^{-\bt}(T_{\phi_{k,\eta}^n} - \id)P_kA^{1/2}: (M_n,\opnorm{\cdot}) \to (M_n,\|\cdot \|)\|_{\cb}  \le \frac{2c_\al c'_\bt K_p \eps}{D} \le \frac{\eps}2
\end{align*}
by choosing $D$ large enough. By Corollary \ref{Pl},  we have $\| A^{-\beta}(1 - P_k): L_p(M_n) \to L_p(M_n)\|_{\cb} = C_pk^{-2\beta}$. It follows that
\begin{align*}
&\quad \|T_{\phi_{k,\eta}^n} -\id: (M_n,\opnorm{\cdot})\to (M_n,\|\cdot\|)\|_{\cb} \\
&\leq \|(T_{\phi_{k,\eta}^n} - \id) P_k: (M_n, \opnorm{\cdot})\to (M_n,\|\cdot\|)\|_{\cb} \\
&\qquad+ \|(T_{\phi_{k,\eta}^n} - \id) A^{-\alpha}A^{-\beta}(1 - P_k) A^{1/2}:  (M_n, \opnorm{\cdot})\to (M_n,\|\cdot\|)\|_{\cb} \\
& \leq \frac{\eps}2 +c_\alpha C_pk^{-2\beta}K_p  \|(T_{\phi_{k,\eta}^n} - \id): L_\8(M_n) \to L_\8(M_n) \|_{\cb}.
\end{align*}
But the cb-norm of $T_{\phi_{k,\eta}^n} - \id: L_\8(M_n) \to L_\8(M_n)$ is less than $2 + \eta$, by the construction of $\phi_{k,\eta}^n$. The assertion follows by choosing $k$ large enough.
\end{proof}



\begin{theorem}\label{cb main}
There exists a sequence $n_j\to \8$ such that $(\ax_{n_j}, L_{n_j})$ converges to $(\ax_{\8}, L_{\8})$ in the $R$-cb-quantum Gromov--Hausdorff distance.
\end{theorem}

\begin{proof}
Let $0<\eps<1, R>0$. In this proof we simply write $n$ for $n_j$. We choose $m$ and $\phi_{k,\eta}^n$ as in Lemma \ref{cb compact}. By Lemma \ref{tri}, we have
\begin{align}\label{main tri}
d^{cb}_{oq, R}(\ax_{\8}, \ax_n)\leq d^{cb}_{oq, R}(\ax_{\8}, \ax_{\8}^{\Lambda_m^2})+d^{cb}_{oq, R}(\ax_{\8}^{\Lambda_m^2}, \ax_n^{\Lambda_m^2})+d^{cb}_{oq, R}(\ax_n^{\Lambda_m^2}, \ax_n).
\end{align}
By Proposition \ref{cb iso}, we may choose $n$ large enough such that the map $\rho_n^{\theta}: \ax_\ta^{\Lambda_m^2}\to M_n^{\Lambda_m^2}$ defined by $u_j^nv_l^n\mapsto u_j(n)v_l(n)$ is a $1+\eps$ cb-isometry and $1+\eps$ Lip-isometry. Hence by Lemma \ref{eps iso},
\[
d^{cb}_{oq, R}(\ax_{\8}^{\Lambda_m^2}, \ax_n^{\Lambda_m^2})\leq 2R\eps.
\]
By Lemma \ref{tailata}, we have $\| T_{\phi_{k,\eta}^n}\|_{\cb}\le(1+\eps)^2$. Together with Lemma \ref{cbga}, we deduce that $\frac1{(1+\eta)^2}(\id\otimes  T_{\phi_{k,\eta}^n})x\in \dx_R(M_p(\ax_{n}^{\Lambda_m^2}))$ for all $x\in \dx_R (M_p(\ax_{n}))$ and $n$ large enough (including $n=\8$). By Theorem \ref{cb compact}, we have $\|x-(\id\otimes  T_{\phi_{k,\eta}^n})x\|< \eps$. This shows that
\[
d_H(\dx_R (M_p(\ax_{n})),\dx_R (M_p(\ax_{n}^{\Lambda_m^2})))<\eps+\Big[1-\frac1{(1+\eps)^2}\Big] R\| T_{\phi_{k,\eta}^n}\|_{\cb}\le (3R+1)\eps.
\]
Hence $d^{cb}_{oq, R}(\ax_{n}, \ax_{n}^{\Lambda_m^2})<(3R+1)\eps$. Hence, by \eqref{main tri}, we conclude that
\[
 d^{cb}_{oq, R}(\ax_{\8}, \ax_n)< 8(R+1)\eps.
\]
This completes the proof.
 \end{proof}

\section{Completely bounded quantum Gromov--Hausdorff distance for higher dimensional tori}\label{higher cb qgh}

In this section we explore the convergence of matrix algebras to the noncommutative tori in higher dimensions. In the following, let $m=\frac{d(d-1)}{2}$ and $\ax^{d}_{\Theta}$ denote the rotation algebra with $d$ generators which was introduced in Section \ref{anaest}.  The following is an analog of Haagerup and R\o rdam's result in higher dimensions.

\begin{theorem}\label{cts n-dim}
There exists a Hilbert space $\hx$, such that for all $\Theta$, there exist unitaries $u_1(\Theta),...,u_d(\Theta)\in \bx(\hx)$ such that
\[
C^*(u_1(\Theta),...,u_d(\Theta))\simeq \ax_\Theta^d\quad \text{and}\quad \lim_{\Theta'\to \Theta} \|u_k(\Theta') - u_k(\Theta)\|_{\bx(\hx)}=0
\]
for $k=1,...,d$.
\end{theorem}

\begin{proof}
We recall the Heisenberg group $\hz_B$ as defined in Subsection \ref{cts fields 1}. To shorten the notation, we will write $\hz$ for $\hz_B$ in the following. Note that since $\hz$ is amenable, $C^*(\hz)$ is a nuclear $C(\tz^m)$-algebra. Therefore, by Theorem 3.2 in \cite{B97} we get a unital monomorphism of $C(\tz^m)$-algebras $\alpha: C^*(\hz)\hookrightarrow \ox_2\otimes C(\tz^m)$ and a unital $C(\tz^m)$-linear completely positive map $E: \ox_2\otimes C(\tz^m)\to C^*(\hz)$ such that $E\circ\alpha=\id_{C^*(\hz)}$. Here $\ox_2$ is the Cuntz algebra with two generators. Let $\ox_2\subset \bx(\hx)$ for some Hilbert space $\hx$. Then for all $x\in C^*(\hz)$, $\alpha(x)\in C(\tz^m, \bx(\hx))$. We define $J_\Theta = \{g\in C(\tz^m): g(\Theta) = 0\}$. Then $J_\Theta$ is a closed ideal of $C(\tz^m)$. Recall from the discussion before Lemma \ref{lower d-dim} the quotient $C_\Theta = C^*(\hz)/I_\Theta$. We consider the following diagram
\[
\xymatrix{
C^*(\hz) \ar[r]^{\al} \ar[d]^{q_{\Theta}} & C(\tz^m, \bx(\hx))\ar[d]^{\td q_\Theta} \\
C_\Theta \ar@{.>}[r] & C(\tz^m, \bx(\hx))/{J_{\Theta}\otimes_{\min} \bx(\hx)}.
}
\]
Since $\al$ is $C(\tz^m)$-linear, the kernel of $q_\Theta$ and that of $\td q_\Theta\circ \al$ coincide. We define  $\pi_\Theta = \td q_\Theta \circ \al \circ q_\Theta^{-1}$. Then $\pi_\Theta$ is a well-defined monomorphism and the above diagram commutes. Note that for $f\in C(\tz^m, \bx(\hx))$, we have $\td q_\Theta(f) = f(\Theta)$ and
\[
\|f(\Theta)\|_{\bx(\hx)}=\|f+ J_{\Theta}\otimes_{\min} \bx(\hx) \|_{C(\tz^m, \bx(\hx))/{J_{\Theta}\otimes_{\min} \bx(\hx)}}.
\]
Then $\lim_{\Theta'\to \Theta} \|\al(x)(\Theta')-\al(x)(\Theta)\|_{\bx(\hx)} = 0$ for $x\in C^*(\hz)$. It follows that
\[
\lim_{\Theta'\to \Theta} \|\pi_{\Theta'}(\hat x+I_{\Theta'}) - \pi_{\Theta}(\hat x+I_{\Theta})\|_{\bx(\hx)} = 0, \quad \hat x+I_{\Theta} \in C_\Theta.
\]
By Lemma \ref{lower d-dim}, $\ax_{2\Theta}^{d} \simeq C_\Theta$ and $\la(0,e_1)+I_\Theta,..., \la(0,e_d)+I_\Theta$ generate $C_\Theta$, where $(e_k)_{k=1}^d$ are the canonical generators of $\zz^d$. Let $u_k(\Theta) = \pi_{\Theta}(\la(0,e_k)+I_\Theta), k=1,..., d$ and note that $\pi_\Theta(C_\Theta) \subset \bx(\hx)$. The proof is complete.
\end{proof}

We now consider approximations of $\ax_{\Theta}^{2d}$ by matrix algebras. We want to use finite dimensional versions of rotation algebras and we have to determine their center. In order to use induction we have to introduce a new form of action. We consider an action $\si$ of $\zz^2$ on a unital C$^*$-algebra $B$. Then we can construct the universal crossed product $B\rtimes_\si \zz^2$. In particular, if $B$ is faithfully represented on $\hx$, we may choose a special representation $\pi $ of $B$ on $\hx\otimes \ell_2(\zz^2)$ such that the left regular representation of $\zz^2$ spatially implements the action $\si$, i.e.
\[
(1\otimes \la_{(j,k)}) \pi(b)(1\otimes \la^*_{(j,k)})=\pi(\si_{(j,k)}(b)), \quad b\in B;
\]
see e.g. \cite{BO08}. Let $u,v$ denote the universal generators of $\ax_\ta$. We define a representation of $\ax_\ta$ by
\[
\ga: \ax_{\ta} \to \bx(\hx)\otimes L(\zz^2)\otimes \ax_\ta, \qquad u^j v^k \mapsto 1\otimes \la_{(j,k)}\otimes u^j v^k.
\]
It follows that for $b\in B$,
\begin{align*}
\ga(u^j v^k) (\pi(b)\otimes 1) \ga(u^jv^k)^* &= [(1\otimes \la_{(j,k)}) \pi(b)(1\otimes \la_{(j,k)})^*] \otimes u^j v^k (u^jv^k)^*\\
& = \pi(\si_{(j,k)}(b))\otimes 1.
\end{align*}
Therefore, the $\zz^2$-action $\si$ and the representations $\pi$ and $\ga$ satisfy
\begin{align}\label{pigasi}
\ga(u)(\pi(b)\otimes 1)\ga(u)^* &= \pi(\si_{(1,0)}(b)), \\
\ga(v)(\pi(b)\otimes 1)\ga(v)^* = \pi(\si_{(0,1)}(b)), \quad &\ga(u)\ga(v) = e^{2\pi\ii \ta} \ga(v)\ga(u).\nonumber
\end{align}

In the following, we use the notation $\lge D: R\rge$ to denote the universal C$^*$-algebra generated by $D$ with relations $R$. We may even ignore $R$ for short if the relations are clear from context. We define
\[
\ax_\ta(n)=\lge U, V: U^n=1=V^n, UV=e^{2\pi \ii \theta}VU, U \text{ and } V \text{ unitaries} \rge,
\]
and
\begin{align}\label{batacross}
B\rtimes_{\sigma}\ax_{\theta}=\lge b,U,V&: b\in B, UbU^*=\sigma_{(1,0)}(b), VbV^*=\sigma_{(0,1)}(b),  \\&UV=e^{2\pi \ii \theta}VU, U \text{ and } V \text{ unitaries}\rge.\nonumber
\end{align}
Note that for $\ax_\ta(n)$ we have necessarily $\ta = \frac{q}n$ for some $q\in \zz$. If $q$ and $n$ are coprime, then it is well known that $\ax_\ta(n)\simeq M_n$.

Similar to the case of $\zz^2$-action on $B$, if we start with a $\zz_n^2$-action $\si$ on $B$ and $\ta=\frac{q}n$, we may find representations $\pi$ of $B$ and $\ga$ of $\ax_\theta(n)$ as follows
\[
\ga: \ax_{\ta}(n) \to \bx(\hx)\otimes L(\zz_n^2)\otimes \ax_\ta(n), \qquad u_j(n) v_{kq}(n) \mapsto 1\otimes \la_{(j,k)}\otimes u_j(n) v_{kq}(n),
\]
where the generators $u_1(n), v_q(n)$ of $\ax_\ta(n)$ are as given in equation \eqref{sgheatmn}. Similarly, it follows that for $b\in B$, the $\zz_n^2$-action $\si$ and the representations $\pi$ and $\ga$ satisfy
\begin{align}\label{pigasin}
\ga(u_1(n))(\pi(b)\otimes 1)\ga(u_1(n))^* &= \pi(\si_{(1,0)}(b)), \\
\ga(v_q(n))(\pi(b)\otimes 1)\ga(v_q(n))^* = \pi(\si_{(0,1)}(b)), \quad &\ga(u_1(n))\ga(v_q(n)) = e^{2\pi\ii \frac{q}n} \ga(v_q(n))\ga(u_1(n)).\nonumber
\end{align}
For simplicity, in the following we will write $\si_{1,0}$ and $\si_{0,1}$ for $\si_{(1,0)}$ and $\si_{(0,1)}$, respectively.

\begin{defn}\label{crossprod}
Suppose $B$ is a unital C$^*$-algebra. We define
\begin{align*}
B\rtimes_{\sigma}\ax_{\theta}(n)&=\lge b,U,V: b\in B, UbU^*=\sigma_{1,0}(b), VbV^*=\sigma_{0,1}(b), \\
&\qquad U^n=1=V^n, UV=e^{2\pi \ii \theta}VU, U \text{ and } V \text{ unitaries}\rge,
\end{align*}
where  $\ta=\frac{q}n$, $q,n\in\nz$, and $\si$ is an action of $\zz_n^2$ on $B$.
\end{defn}

Thanks to \eqref{pigasi} and \eqref{pigasin}, the universal objects defined above exist. By universality and using the notation introduced here, we  can  rewrite the noncommutative torus $\ax_{\Theta}^{2d}$ iteratively as
\begin{align}\label{atacross}
\ax_{\Theta}^{2d} = \ax_{\ta_{12}} \rtimes_{\si^{2}} \ax_{\ta_{34}} \rtimes_{\si^{3}} \cdots \rtimes_{\si^{d}} \ax_{\ta_{2d-1,2d}},
\end{align}
where the $\zz^2$-action $\si^{k}, k=2,...,d,$ is defined by
\begin{align*}
\sigma^k_{1,0}(u_1)  &= e^{-2\pi \ii \ta_{1,2k-1}} u_1,\ ...,\  \sigma^{k}_{1,0}(u_{2k-2}) = e^{-2\pi \ii \ta_{2k-2,2k-1}}u_{2k-2},\\
\sigma^k_{0,1}(u_1) &= e^{-2\pi \ii \ta_{1,2k}} u_1,\ ...,\  \sigma^k_{0,1}(u_{2k-2}) = e^{-2\pi \ii \ta_{2k-2,2k}} u_{2k-2}.
\end{align*}
Indeed, note that by the definition in \eqref{batacross}, we have
\begin{align*}
\sigma^k_{1,0}(u_1) &= u_{2k-1} u_1 u_{2k-1}^*, \ ..., \ \sigma^{k}_{1,0}(u_{2k-2}) = u_{2k-1} u_{2k-2} u_{2k-1}^*, \\
\sigma^k_{0,1}(u_1) &= u_{2k} u_1 u_{2k}^*, \ ..., \ \sigma^{k}_{0,1}(u_{2k-2}) = u_{2k} u_{2k-2} u_{2k}^*.
\end{align*}
Then \eqref{atacross} follows from the universality of $\ax_\Theta^{2d}$.

\begin{prop}\label{untwist}
Let $\ta=\frac{q}n$ and $q, n$ be coprime. Let $\si$ be an action of $\zz_n^2$ on a unital C$^*$-algebra $B$. Then $B\rtimes_{\sigma}\ax_{\theta}(n)  \simeq M_n(B)$.
\end{prop}
\begin{proof}
We first consider the case for which the action $\sigma$ is inner, i.e. there exist unitaries $w_1$ and $w_2$ in $B$ such that $\sigma_{1,0}(x)=w_1xw_1^*$, $\sigma_{0,1}(x)= w_2 xw_2^*$, $[w_1, w_2]=0$ and $w_1^n=1=w_2^n$. Let $u_n$ and $v_n$ be the generators of $\ax_{\theta}(n)$. We consider a special representation $\pi_0$ of $B\rtimes_{\sigma}\ax_{\theta}(n)$ defined by $\pi_0(b)=b\otimes 1$ for $b\in B$, $\pi_0(U)=w_1\otimes u_n$ and $\pi_0(V)=w_2\otimes v_n$. It can be directly checked that $\pi_0$ is indeed a representation. Then we have
\begin{align}\label{inner act}
\pi_0 &(B\rtimes_{\sigma}\ax_{\theta}(n))=\pi_0(\langle b, U, V: b\in B\rangle)\\
&=C^*(b\otimes 1, w_1\otimes u_n, w_2\otimes v_n: b\in B)
= B\otimes_{\min} M_n. \nonumber
\end{align}
Now let $\pi_u : B\rtimes_{\sigma}\ax_{\theta}(n)\to \bx(\hx_u)$ be the universal representation of $B\rtimes_{\sigma}\ax_{\theta}(n)$. We show that in this case, we can also write $\pi_u(U)$ and $\pi_u(V)$ as tensors. Note that $\ax_\ta(n)$ has dimension at most $n^2$. Thanks to the image of $U$ and $V$ under $\pi_0$, we know that $C^*(\pi_u(U), \pi_u(V))= M_n$. Therefore, we may take $\hx_u=\kx\otimes \ell_2^n$ for some Hilbert space $\kx$. Let us define $u= \pi_u(w_1^*)\pi_u(U) $ and $v=\pi_u( w_2^*)\pi_u(V)$. Then for $x\in \pi_u(B)$,
\[
x = \sigma_{1,0}^{-1} [ \sigma_{1,0}(x)]=\pi_u (w_1^*) \pi_u(U)x \pi_u(U)^*\pi_u(w_1) = u x u^*.
\]
Thus $ux=xu$. Similarly, $vx=xv$. We deduce that $\pi_u(B)\subset C^*(u,v)'\cap \bx(\hx_u)$. Since $w_1$ and $w_2$ commute, plugging in $x=\pi_u(w_1),\pi_u(w_2)$, we find $\pi_u(U)\pi_u(w_i)=\pi_u(w_i)\pi_u(U)$ and $\pi_u(V)\pi_u(w_i)=\pi_u(w_i)\pi_u(V)$ for $i=1,2$. It follows that $\pi_u(w_i)\in M_n'\cap \bx(\hx_u)$. Moreover, $u$ and $v$ also satisfy the conditions $uv=e^{2\pi \ii \theta}vu$ and $u^n=1=v^n$. Therefore,
\[
C^*(u,v)\simeq M_n,\quad \pi_u(B)\subset \bx(\kx)\otimes \cz \quad \text{and}\quad u,v \in \pi_u(B)'\cap \bx(\hx_u).
\]
We may write $u = a \otimes \td u$ for some $a \in \pi_u(B)'\cap \bx(\kx)$ and $\td u\in M_n$, and $\pi_u(w_1)=\pi_{\kx}(w_1)\otimes z$ for some $z\in \cz$ where $\pi_{\kx}$ is the restriction of $\pi_u$ on $\kx$. Hence,
\[
\pi_u(U) = \pi_u(w_1)u =  \pi_{\kx} (w_1)a \otimes z \td u .
\]
Similarly, we can write $\pi_u(V)$ as a tensor. By \eqref{inner act}, $B\rtimes_{\sigma}\ax_{\theta}(n)\simeq M_n(B)$.

Now we consider $\sigma$ to be a general action. We define a $\zz^2_n$ action $\hat\si$ on $B\rtimes_\si \ax_\ta(n)$: For $x=\sum_{k,l} b_{kl}U^k V^l\in B\rtimes_\si \ax_\ta(n)$,
\[
\hat\sigma_{1,0}(x)=\sum_{k,l} \sigma_{1,0}(b_{kl})U^k V^l, \quad\hat\sigma_{0,1}(x)=\sum_{k,l} \sigma_{0,1}(b_{kl})U^k V^l.
\]
Similarly, we define a $\zz^2_n$-action, still denoted by  $\hat\si$, on the universal crossed product $B\rtimes_\si \zz^2_n$:
For $x=\sum_{k,l} b_{kl}\la(k,l)\in B\rtimes_\si \zz^2_n$,
\[
\hat\sigma_{1,0}(x)=\sum_{k,l} \sigma_{1,0}(b_{kl})\la(k,l), \quad\hat\sigma_{0,1}(x)=\sum_{k,l} \sigma_{0,1}(b_{kl})\la(k,l).
\]
Then by universality we have $(B\rtimes_{\sigma}\ax_{\theta}(n))\rtimes_{\hat{\sigma}} \zz_n^2= (B\rtimes_{\sigma}\zz_n^2)\rtimes_{\hat\sigma}\ax_{\theta}(n)$. By the crossed product construction, the action $\hat\si$ on $B\rtimes_\si \zz^2_n$ is spatially implemented by $w_1=1\otimes \la(1,0)$ and $w_2=1\otimes \la(0,1)$. More precisely,
\[
\pi(\hat\si_{1,0}(x))= (1\otimes w_1) \pi(x)(1\otimes w_1^*), \quad \pi(\hat\si_{0,1}(x))= (1\otimes w_2) \pi(x)(1\otimes w_2^*),
 \]
where $\pi(x)= \oplus_{g\in \zz_n^2}\si_{g^{-1}}(x)$; see e.g. \cite{BO08} for more details. By what we proved in the first paragraph, we find that $(B\rtimes_{\sigma}\zz_n^2)\rtimes_{\hat\sigma}\ax_{\theta}(n) \simeq M_n(B\rtimes_{\sigma}\zz_n^2)$. But $M_n(B\rtimes_{\sigma}\zz_n^2)=M_n(B)\rtimes_{\sigma}\zz_n^2$ where we have denoted the inflated action $\id\otimes \si$ still by $\si$. It is well known that there exists a faithful conditional expectation $E: M_n(B)\rtimes_{\sigma}\zz_n^2 \to M_n(B)$. Recall that we have the canonical embedding $\iota: B\rtimes_{\sigma}\ax_{\theta}(n) \hookrightarrow (B\rtimes_{\sigma}\ax_{\theta}(n))\rtimes_{\hat{\sigma}} \zz_n^2$. We have the following diagram
\[
\xymatrix{
(B\rtimes_{\sigma}\ax_{\theta}(n))\rtimes_{\hat{\sigma}} \zz_n^2 \ar[r]^{\ \ \simeq} &M_n(B)\rtimes_{\sigma}\zz_n^2 \ar[d]^{E} \\
B\rtimes_{\sigma}\ax_{\theta}(n)\ar[u]^{\iota}\ar@{.>}[r] & M_n(B)
}
\]
Note that the multiplicative domain of $E$ is $M_n(B)$, restricted on which $E$ is a $^*$-homomorphism. Moreover, $B\rtimes_{\sigma}\ax_{\theta}(n)$ is contained in the multiplicative domain of $E$ and clearly $E(B\rtimes_{\sigma}\ax_{\theta}(n)) = M_n(B)$. But $E\circ \iota$ is faithful. Hence, we find that $B\rtimes_{\sigma}\ax_{\theta}(n)\simeq M_n(B)$.
\end{proof}

In the following we show the convergence of the matrix algebras to the rotation algebra $\ax^{2d}_\Theta$.  Similar to the two-dimensional case, we extend the Lip-norms to matrix levels on (a dense subspace of) $\ax_\Theta$ as in \eqref{lipmat}:
\[
\opnorm{x}_m = \max\{\|\id\otimes \de(x)\|_{M_m\otimes_{\min} \ax_\Theta\otimes_{\min}H_\psi^c} ,~\|\id\otimes \de(x)\|_{M_m\otimes_{\min} \ax_\Theta \otimes_{\min}H_\psi^r} \}.
\]
Similarly, by Remark \ref{mndlipemb} we may extend the Lip-norms on $M_{n^d}$ to matrix levels once we choose a set of generators of $M_{n^d}$. The Lip-norms on $\ax_\Theta$ and $M_{n^d}$ will also be denoted by $L_\8(\cdot)$ and $L_n(\cdot)$, respectively, especially when we consider continuous fields of compact quantum metric spaces. We follow the same plan as in Section \ref{cb qgh}. Let $u_1(\Theta),...,u_{2d}(\Theta)$ be the generators of $\ax^{2d}_\Theta$. In particular, $u_1(0),..., u_{2d}(0)$ generate $C(\tz^{2d})$. Following Definition \ref{crossprod}, we consider the C$^*$-algebra
\begin{align}\label{a1n2d}
\ax_{1/n}^{2d}:=\ax_{\theta_{1,2}}(n)\rtimes_{\sigma^{2}} \ax_{\theta_{3,4}}(n)\rtimes_{\sigma^{3}} \cdots \rtimes_{\sigma^{d}} \ax_{\theta_{2d-1,2d}}(n),
\end{align}
where the action $\si^{k} ,k=2,...,d$, is defined by
\begin{align*}
\sigma^k_{1,0}(u_1) = u_{2k-1} u_1 u_{2k-1}^* &= e^{-2\pi \ii \ta_{1,2k-1}} u_1,\ ...,\  \sigma^k_{1,0}(u_{2k-2}) = u_{2k-1} u_{2k-2} u_{2k-1}^* = e^{-2\pi \ii \ta_{2k-2,2k-1}}u_{2k-2},\\
\sigma^k_{0,1}(u_1) = u_{2k} u_1 u_{2k}^* &= e^{-2\pi \ii \ta_{1,2k}} u_1,\ ...,\ \sigma^k_{0,1}(u_{2k-2}) = u_{2k} u_{2k-2} u_{2k}^* = e^{-2\pi \ii \ta_{2k-2,2k}} u_{2k-2},\\
&u_i^{n}=1, \ i=1,...,2d, \qquad \ta_{i,j}=\frac1n, \ 1\le i<j\le 2d.
\end{align*}
Then by Proposition \ref{untwist}, we have $\ax_{1/n}^{2d}\simeq M_{n^d}$. For definiteness, let us fix the generators in the iterated crossed product and define $v_1(n)=u_1,..., v_{2d}(n)=u_{2d}$. Then we have
\begin{align}\label{1ncomm}
v_j(n)v_k(n)= e^{\frac{2\pi \ii}{n} } v_k(n) v_j(n),\quad 1\le j<k \le n.
\end{align}
We define a map $\rho_n: \cz(\zz^{2d})=\cz(\zz)\otimes \cdots \otimes \cz(\zz)\to M_{n^d}$ by
\[
\rho_n(u_1(0)^{i_1}\cdots u_{2d}(0)^{i_{2d}})= v_1(n)^{i_1}\cdots v_{2d}(n)^{i_{2d}}.
\]
Let $\Lambda_k^d=\{0, \pm1,...,\pm k\}^d$ and
\[
C^{\Lambda^{2d}_k}(\tz^{2d}) = \Big\{x\in C(\tz^{2d}): x=\sum_{|i_1|\le k, ... |i_{2d}|\le k} a_{i_1,...,i_{2d}}u_1(0)^{i_1}\cdots u_{2d}(0)^{i_{2d}}, a_{i_1,...,i_{2d}}\in \cz \Big\}.
\]
Similarly, we will consider $M^{\Lambda^{2d}_k}_{n^d}$, $\ax_\Theta^{\Lambda_k^{2d}}$, etc.~in the following. Similar to Proposition \ref{theta0 cb}, for any $\eps>0, k>0$, there exists $N>0$ such that for all $n>N$, $\rho_n |_{C^{\Lambda^{2d}_k}(\tz^{2d}) } : C^{\Lambda^{2d}_k}(\tz^{2d}) \to M^{\Lambda^{2d}_k}_{n^d}$ is a $1+\eps$ cb-isometry.

\begin{lemma}\label{d-cb iso}
For any $\eps>0$ and $k\geq 0$, there exists $N>0$ such that for any $n>N$, the map $\rho_n|_{C^{\Lambda^{2d}_k}(\tz^{2d})}: C^{\Lambda^{2d}_k}(\tz^{2d})\to M^{\Lambda^{2d}_k}_{n^d}$ is a $1+\eps$ cb-isometry and a $1+\eps$ Lip-isometry.
\end{lemma}
\begin{proof}
By the definition of $\rho_n$ and the commutation relations \eqref{1ncomm},  we can generalize directly Lemma \ref{ct2cts} to get a faithful $^*$-homomorphism $(\rho_n)^\bullet: C(\tz^{2d}) \to \prod_{\om} M_{n^d}$, where $\prod_\om M_{n^d}$ is the von Neumann algebra ultraproduct. Now we repeat the proof of Proposition \ref{theta0 cb} with the result of Haagerup--R\o rdam replaced by Theorem \ref{cts n-dim}. The claim of $1+\eps$ Lip-isometry follows from the same argument as that of Lemma \ref{cbga}. We leave the details to the reader.
\end{proof}

Suppose $\ta_{rs}=\frac{p_{rs}}q, 1\le r<s\le q$. We consider the iterated crossed product following the notation introduced in Definition \ref{crossprod}
\begin{align}\label{ata2dq}
\ax^{2d}_{\Theta}(q):=\ax_{\theta_{1,2}}(q)\rtimes_{\sigma^{2}} \ax_{\theta_{34}}(q)\rtimes_{\sigma^{3}} \cdots \rtimes_{\sigma^{d}} \ax_{\theta_{2d-1,2d}}(q),
\end{align}
where the action $\si^{k} ,k=2,...,d$ is defined by
\begin{align*}
\sigma^k_{1,0}(u_1) = u_{2k-1} u_1 u_{2k-1}^* &= e^{-2\pi \ii \ta_{1,2k-1}} u_1,\ ...,\  \sigma^k_{1,0}(u_{2k-2}) = u_{2k-1} u_{2k-2} u_{2k-1}^* = e^{-2\pi \ii \ta_{2k-2,2k-1}}u_{2k-2},\\
\sigma^k_{0,1}(u_1) = u_{2k} u_1 u_{2k}^* &= e^{-2\pi \ii \ta_{1,2k}} u_1,\ ...,\  \sigma^k_{0,1}(u_{2k-2}) = u_{2k} u_{2k-2} u_{2k}^* = e^{-2\pi \ii \ta_{2k-2,2k}} u_{2k-2},\\
&u_i^{q}=1, \ \ i=1,...,2d.
\end{align*}
For definiteness, in the following result, the generators of $\ax_\Theta^{2d}(q)$ will be denoted by $u_j^\Theta(q)$, $j=1,...,2d$.

\begin{prop}\label{cbisom 2d}
Let $\Theta=(\ta_{rs})_{r,s=1}^{2d}$ and $\Theta^n=(\ta^n_{rs})_{r,s=1}^{2d}$ be two skew symmetric matrices such that $\ta_{rs}= \frac{p_{rs}}q$ and $\ta^n_{rs} = \ta_{rs}+\frac1n$ for $r<s$.
For any $\eps>0$ and $k\ge 0$, there exists $N>0$ such that for $n>N$ divisible by $q$ the maps $\rho_{n}^\Theta: \ax_\Theta^{\Lambda_k^{2d}} \to [\ax^{2d}_{\Theta^n}(n)]^{\Lambda_k^{2d}}$ defined by
\[
\rho_n^\Theta (u_1(\Theta)^{i_1}\cdots u_{2d}(\Theta)^{i_{2d}}) = [u_1^{\Theta^n}(n)]^{i_1}\cdots [u_{2d}^{\Theta^n}(n)]^{i_{2d}}, \quad |i_j |\le k, j=1,...,2d
\]
is a $1+\eps$ cb-isometry and a $1+\eps$ Lip-isometry. Moreover, we have $\ax^{2d}_{\Theta^{q^{l+1}}}(q^{l+1})\simeq M_{q^{(l+1)d}}$.
\end{prop}
\begin{proof}
Similar to \eqref{rhonta}, we define
\begin{align*}
  \vsi: \ax_{\Theta}^{2d}& \to \ax_{\Theta}^{2d}(q)\otimes_{\min} C(\tz^{2d})\\
   \vsi(u_1(\Theta)^{k_1}\cdots u_{2d}(\Theta)^{k_{2d}})&= u_1^\Theta(q)^{k_1}\cdots u_{2d}^\Theta(q)^{k_{2d}}\otimes u_1(0)^{k_1}\cdots u_{2d}(0)^{k_{2d}}.
\end{align*}
Since the canonical trace on $\ax_\Theta^{2d}$ is faithful (see e.g. \cite{Ri90}) and $\vsi$ is trace-preserving, $\vsi$ is a faithful $^*$-homomorphism. Recall that by \eqref{a1n2d} and Proposition \ref{untwist}, $\ax_{1/n}^{2d}$ is a matrix algebra. Recall also that $\ax_\Theta^\8$ denotes the algebra of all finite linear combinations of $u_1(\Theta)^{k_1}\cdots u_{2d}(\Theta)^{k_{2d}}, (k_1,...,k_{2d})\in\zz^{2d}$.  Define
\begin{align}\label{rhonTa}
  \rho_n^\Theta &= (\id\otimes \rho_n)\circ \vsi: \ax_\Theta^{\8} \to \ax_{\Theta}^{2d}(q) \otimes_{\min} \ax_{1/n}^{2d}\\
   u_1(\Theta)^{k_1}\cdots u_{2d}(\Theta)^{k_{2d}} &\mapsto u_1^\Theta(q)^{k_1}\cdots u_{2d}^\Theta(q)^{k_{2d}} \otimes v_{1}(n)^{k_1}\cdots v_{2d}(n)^{k_{2d}}=:\hat{u}_1^{k_1}\cdots \hat u_{2d}^{k_{2d}}. \nonumber
\end{align}
By Lemma \ref{d-cb iso}, for any $\eps>0$ and $k\ge0$ there exists $N>0$ such that for all $n>N$
\begin{align*}
\rho_n^\Theta|_{\ax_{\Theta}^{\Lambda^{2d}_k}}: \ax_{\Theta}^{\Lambda^{2d}_k}&\to \ax_{\Theta}^{2d}(q)\otimes \ax_{1/n}^{2d}
\end{align*}
is a $1+\eps$ cb-isometry and a $1+\eps$ Lip-isometry onto its image. To identify the image of $\rho_n^\Theta$, note that by \eqref{1ncomm} we have
\begin{equation}\label{tarsn}
\hat u_r \hat u_s= e^{2\pi \ii (\ta_{rs}+\frac1n)} \hat u_s \hat u_r, \quad 1\le r <s \le 2d.\\
 \end{equation}
If we let $\ta_{rs}^n=\ta_{rs}+\frac1n=\frac{n\ta_{rs}+1}{n}$, since by assumption $n\ta_{rs}$ is an integer, we may define the iterated crossed product $\ax_{\Theta^n}(n)$ in the same way as \eqref{ata2dq}, where the entries of $\Theta^n$ are given by $\ta^n_{rs}$. Although $\ax_{\Theta^n}(n)$ is universally defined, dimension counting shows that we can take $\hat u_1,..., \hat u_{2d}$ as its universal generators. Therefore, we have $\rho_n^\Theta(\ax_\Theta^{\8}) = \ax_{\Theta^n}^{2d}(n)$, and $\hat u_j=u_j^{\Theta^n}(n)$, $j=1,...,2d$.

Similar to the case of the 2-dimensional tori, we can choose a subsequence $n_l$ so that $\hat u_1,...,\hat u_{2d}$ generate $M_{n_l^d}$. Indeed, since $\hat{u}_r\hat{u}_{s}= e^{2\pi \ii(\theta_{rs}+\frac{1}{n_l})}\hat{u}_{s}\hat{u}_r$ for all $r< s$, by Proposition \ref{untwist} we just need $(\ta_{rs} n_l +1, n_l)=1$ for all $r<s$ to verify that $C^*(\rho_{n_l}^{\Theta}(\ax_{\Theta}^{\8}))\simeq M_{n_l^d}$. But $\ta_{rs}=\frac{p_{rs}}q$, it suffices to take $n_l=q^{l+1}$ as in Lemma \ref{rhotank}. Then we find that
\[
\ta_{rs}^{n_l}=\frac{q^l p_{rs}+1}{q^{l+1}}\quad \text{and} \quad  \ax^{2d}_{\Theta^{n_l}}(n_l)\simeq M_{n_l^d}.
\]
Note that the generators of $M_{n_l^d}$ may be different from $v_1(n_l),...,v_{2d}(n_l)$. Thus $M_{n_l^d}^{\Lambda_k^{2d}}$ refers to the subspace of $M_{n_l^d}$ generated by $\hat u_1^{j_1}\cdots \hat u_{2d}^{j_{2d}}$ for $|j_i|\le k, i=1,...,2d$.
\end{proof}

Note that we did not assume $(p_{rs},q)=1, r<s,$ in the above result. With this condition, we would have by Proposition \ref{untwist} $\ax_\Theta^{2d}(q) \simeq M_{q^d}$ as the two-dimensional case in \eqref{rhonta}. This would be a little more explicit about the algebra in which the image of $\rho_n^\Theta$ lies.  Suppose now $\theta_{rs}$ may not be written as $\frac{p_{rs}}{q}$. Note that the set
\[\{(\frac{p_{rs}}{q})_{1\le r<s\le 2d}: 0\neq q\in \zz, p_{rs}\in \zz, |p_{rs}|\le q\}
 \]
is dense in $[-1,1]^{d(2d-1)}$. In fact, the set $\{(\frac{p_{rs}}{q})_{1\le r<s\le 2d}: 0\neq q\in \zz, p_{rs}\in \zz, |p_{rs}|\le q, (p_{rs},q)=1\} $ is still dense in $[-1,1]^{d(2d-1)}$. This is needed in the following argument if we assumed $(p_{rs},q)=1, r<s,$ in Proposition \ref{cbisom 2d}. 

Following the same argument as that of Proposition \ref{cb iso} with the result of Haagerup--R\o rdam replaced by Theorem \ref{cts n-dim}, we find the map
\begin{align}\label{rhonta3}
\rho_n^\Theta: \ax_{\Theta}^{^{\Lambda_k^{2d}}} \to [\ax^{2d}_{\td{\Theta}^n}(n)]^{\Lambda_k^{2d}}
\end{align}
which is a $1+\eps$ cb-isometry and a $1+\eps$ Lip-isometry. Here $\td \Theta^n$ is chosen such that $\td \Theta =(\td \ta_{rs}=\frac{p_{rs}}{q})_{r,s=1}^n$ is close to $\Theta$ and $\td\Theta^n=(\td\ta^n_{rs}=\frac{p_{rs}}{q}+\frac1n)_{r,s=1}^n$. Moreover, we have $\ax^{2d}_{\td\Theta^{q^{l+1}}}(q^{l+1})\simeq M_{q^{(l+1)d}}$.

Let $\mx_{n_l}=(M_{n_l^d})_{sa}$ for $n\in\nz$ and $\mx_\8=\ax_\Theta^\8 \cap (\ax_\Theta)_{sa}$. We choose $\Theta_l$ close to $\Theta$ so that the entries of $\Theta_l$ are of the form $\frac{p_{rs}}q$ as in Proposition \ref{cbisom 2d}. We may proceed in the same way as the proof of Lemma \ref{ctfata2} and check that $\rho_{n_l}^{\Theta}(\si_\Theta(p))+\rho_{n_l}^{\Theta}(\si_\Theta(p))^*$ is continuous in $l$ for all $p\in {\rm Poly}(x_1,...,x_{2d})$. Here $\rho_{n_l}^{\Theta}$ is defined exactly as in \eqref{e:rhotacom} via $\rho_{n_l}^{\Theta_l}\circ \si_{\Theta_l}\circ \si_\Theta^{-1}$ and $\rho_{n_l}^{\Theta_l}$ was defined in \eqref{rhonTa}. One can check that $\rho_{n_l}^\Theta$ thus defined, when restricted to $\ax_{\Theta}^{^{\Lambda_k^{2d}}}$, coincides with \eqref{rhonta3}. Let $\sx$ denote a set of continuous sections of the continuous field of order-unit spaces over $\overline\nz$ with fibers $(M_{n_l^d})_{sa}$, where we understand $M_{\8^{d}} = \ax_\Theta$. We may choose $\sx=\{(\rho_{n_l}^{\Theta}(\si_\Theta(p))+\rho_{n_l}^{\Theta}(\si_\Theta(p))^*)_{l\in \overline\nz}: p\in {\rm Poly}(x_1,...,x_{2d})\}$.

The following is a consequence of these results with the same proofs as that of Lemma \ref{ctfata2} and Proposition \ref{cb cts field}. The remark after Proposition \ref{cb cts field} applies here as well.

\begin{prop}\label{cbcts 2d}
$(\{(\mx_{n_l}, L_{n_l}\}_{l\in \overline\nz}, \sx)$ is a cb-continuous field of compact quantum metric spaces.
\end{prop}

We consider a conditionally negative length function $\phi_n$ on $\zz_n^{2d}$ for $n\in\overline\nz$ as in Section \ref{anaest}. For example, we can take $\phi_n(k_1,...,k_{2d})=\psi_n(k_1)+\cdots +\psi_n(k_{2d})$, where $\psi_n$ is given in \eqref{cnl1}. We find a symmetric Markov semigroup on $L(\zz_n^{2d})$, which induces a symmetric Markov semigroup on $M_{n^d}$ as in \eqref{tphi}. Lemma \ref{trans} and Proposition \ref{cbriesz} extend directly to the current situation.
\begin{theorem}\label{atarcb}
There exists a sequence of matrix algebras $M_{n^d_j}$ converging to $\ax^{2d}_\Theta$ in the $R$-cb quantum Gromov--Hausdorff distance.
\end{theorem}
\begin{proof}
First we need a tail estimate which is an extension of Theorem \ref{cb compact}. This follows the same proof as that of Theorem \ref{cb compact}. Indeed, similar to the proof of Lemma \ref{tailata}, given $\eps>0$, we may choose $k$ and then define $\phi_{k,\eta}^n(j_1,...,j_{2d}) =\varphi_{k,\eta}^n(j_1)\cdots\varphi_{k,\eta}^n(j_{2d})$, where $\vph_{k,\eta}^n(\cdot)$ is the multiplier found in Lemma \ref{cbapz} and this time we take $\eta\in(0, \frac{\eps}{2d(2k+1)^{2d}})$.  Then we use (possibly extended versions of) Lemma \ref{ttbd1}, Lemma \ref{qcbd}, Corollary \ref{Pl},  Proposition \ref{cbriesz} (or Corollary \ref{ade8p}) and Remark \ref{cbga2} as explained above. The rest of the argument is a simple extension of the proof of Theorem \ref{cb main}.
\end{proof}

\section{Application to Gromov--Hausdorff propinquity}\label{s:prop}
In this section we work in the framework of Gromov--Hausdorff propinquity developed by Latr\'{e}moli\`{e}re  recently in \cite{La15, La16,La15a}. We will show that our previous results on convergence of matrix algebras to rotation algebras actually hold in the strong sense of Gromov--Hausdorff propinquity. For a fixed permissible function $F$ (see \cite{La15}*{Definition 2.18}), denote by $\Lambda_F((A, L_A), (B, L_B))$ the Gromov--Hausdorff propinquity between two compact quantum metric spaces, in the sense of \cite{La15}*{Definition 3.54}. Recall that according to \cite{La15}*{Definition 3.42}, if $A$ and $B$ are two unital C$^*$-algebras, a bridge $\gamma = (D,\omega, \pi_A,\pi_B)$ is given by a unital C$^*$-algebra $D$, two unital $^*$-monomorphisms $\pi_A: A\hookrightarrow D$ and $\pi_B: B\hookrightarrow D$ and $\omega \in D$ such that the set $S(A|\omega) := \{\varphi \in S(D) : \forall d\in D, \varphi(d) = \varphi(d\omega) = \varphi(\omega d)\}$ is not empty, where $S(D)$ denotes the state space of $D$. In the following let $F: [0, \8)^4 \to [0, \8)$ be defined by $F(x,y,l_x,l_y) = xl_y + yl_x$, for $x, y, l_x, l_y \in [0, \8)$; see Definition 2.18 in \cite{La15}.

\begin{lemma}\label{l:propinquity}
Let $(A, \opnorm{.}_A)$ and $(B,  \opnorm{.}_B)$ be two $F$-quasi-Leibniz compact quantum metric spaces in the sense of \cite{La15}*{Definition 2.44}. If there exist $\eps>0$ and two unital $^*$-monomorphisms $\pi_A: A\hookrightarrow  \bx(\hx)$ and $\pi_B: B\hookrightarrow \bx(\hx)$ for some Hilbert space $\hx$  such that the following hold:
\begin{enumerate}
\item For all $a\in A$ such that $\opnorm{a}_A \leq 1$, there exists $b\in B$ such that $\opnorm{b}_B \leq 1$ and $\|\pi_A(a)-\pi_B(b)\|_{\bx(\hx)} \le \eps$.
\item For all $b\in B$ such that $\opnorm{b}_B \leq 1$, there exists $a\in A$ such that $\opnorm{a}_A \leq 1$ and $\|\pi_A(a)-\pi_B(b)\|_{\bx(\hx)} \le \eps$.
\end{enumerate}
Then $\Lambda_F((A,  \opnorm{.}_A), (B,  \opnorm{.}_B)) \leq \eps$.
\end{lemma}

\begin{proof}
We refine the proof of Lemma 3.79 in \cite{La15} by taking a \emph{trek} (see Definition 3.49 in \cite{La15}) consisting of a single \emph{bridge} (see Definition 3.42 in \cite{La15}), namely $\gamma = (\bx(\hx),\id, \pi_A,\pi_B)$. Note that in this case, since any state on $A$ or $B$ can be extended to a state on $\bx(\hx)$, and for $\omega = \id$, $S(A |\omega) = S(\bx(\hx))$, with the notation of \cite{La15}, we have the \emph{height} $\zeta(\gamma |  \opnorm{.}_A, \opnorm{.}_B) = 0$ ; see Definition 3.46 in \cite{La15}. On the other hand, if (1) and (2) hold, then by definition, the \emph{reach} $\rho (\gamma |  \opnorm{.}_A, \opnorm{.}_B) \le \eps$ (see Definition 3.45 in \cite{La15}). Now by Definitions 3.47 and 3.54 in the aforementioned paper, we have $\Lambda_F((A,  \opnorm{.}_A), (B,  \opnorm{.}_B)) \leq \eps$.
\end{proof}




In Latr\'{e}moli\`{e}re's original definition, if $(A,\opnorm{.})$ is a quantum compact metric space \cite{La16}, then $A$ is a C$^*$-algebra instead of an order-unit space. If $(A, \opnorm{.}_A)$ is a quantum compact metric space, it is required that $\opnorm{.}$ vanishes exactly on scalars, i.e. $\opnorm{.}_A$ is a norm on $A/\cz 1$. For our later development, it is convenient to write $\opnorm{.}_{M_m\otimes A}$ for the norm $\opnorm{.}$ on $M_m(A/\cz 1)$. Clearly, $\opnorm{.}_{M_m\otimes A}$ vanishes on $M_m\otimes 1$ and thus $(M_m\otimes A, \opnorm{.}_{M_m\otimes A})$ itself is not a quantum compact metric space. In order to work in the framework of quantum Gromov--Hausdorff propinquity with matrix coefficients, however, we have to modify  Latr\'{e}moli\`{e}re's original definition of the quantum Gromov--Hausdorff propinquity. Note that the conditions (1) and (2) in Lemma \ref{l:propinquity} do not require that $\opnorm{.}$ vanishes only on scalars.

\begin{defn}
Let $(A_n, \opnorm{.}_{A_n})$ and $(B, \opnorm{.}_B)$ be $F$-quasi-Leibniz quantum compact metric spaces in the sense of \cite{La15}. Let $\kx$ be the space of compact operators on $\ell_2$. We say $(A_n, \opnorm{.}_{A_n})$ converges to $(B, \opnorm{.}_B)$ in \emph{strong quantum Gromov--Hausdorff propinquity}  if for any $\eps>0$ there exists $N>0$ such that for all $n>N$, there exist unital $^*$-monomorphisms $\pi_{A_n}: A_n \hookrightarrow \bx(\hx)$ and $\pi_B: B\hookrightarrow \bx(\hx)$ for some Hilbert space $\hx$ with the following properties:
\begin{enumerate}
\item For all $a\in \kx \otimes A_n$ with $\opnorm{a}_{\kx\otimes A_n}\le 1$, there exists $b\in \kx\otimes B$ such that $\opnorm{b}_{\kx\otimes B}\le 1$ and  $\|(\id\otimes \pi_{A_n})(a)-(\id\otimes \pi_B)(b)\|_{\bx(\ell_2\otimes \hx)}\le \eps$.
\item For all $b\in \kx \otimes B$ with $\opnorm{b}_{\kx\otimes B} \le 1$, there exists $a_n\in \kx\otimes A_n$ such that $\opnorm{a_n}_{\kx\otimes A_n}\le 1$ and  $\|(\id\otimes \pi_{A_n})(a_n)-(\id\otimes \pi_B)(b)\|_{ \bx(\ell_2\otimes \hx)}\le \eps$.
\end{enumerate}
\end{defn}
Thanks to Lemma \ref{l:propinquity}, the notion of strong quantum Gromov--Hausdorff propinquity is indeed stronger than the original definition of Gromov--Hausdorff propinquity.

Recall the definition of $\ax^{2d}_{\tilde \Theta^n}(n)$ in Proposition \ref{cbisom 2d}.  By Definition 2.21 in \cite{La15} and the existence of a derivation $\delta$ as defined in (\ref{derext}), $(\ax^{2d}_\Theta, \opnorm{.})$ and $(\ax^{2d}_{\tilde \Theta^n}(n), \opnorm{.})$ are \emph{Leibniz pairs}. Indeed, the conditions in the definition were proved in \cite{JM10, JMP10}; see also \cite{Ze13} for more remarks on the Lip-norms. Furthermore, let $\delta^c$ and $\delta^r$ denote the column and row structure derivations, respectively (see Lemma \ref{gahr2}). Note that Remark \ref{mndlipemb} applies to the algebra $\ax^{2d}_{\tilde \Theta^n}(n)$ as well by choosing $2d$ generators. Then by \eqref{lipmat}, for $x\in \ax^{2d}_{\tilde \Theta^n}(n)$, $n\in \overline{\nz}$, $\opnorm{x}=\max\{\|\delta^c(x)\|, \|\delta^r(x)\|\}$. We choose the multiplier $\phi_{k,\eta}^n$ on $\zz_n^{2d}$ as  in the proof of Theorem \ref{atarcb} with $\eta=[2d(2k+1)^{2d+1}]^{-1}$ (or $\eps=(2k+1)^{-1}$). By Lemma \ref{cbga} and the choice of $\phi_{k,\eta}^n$, we have as in the proof of Proposition \ref{tailata} $\|\de^c(T_{\phi_{k,\eta}^n}(x))\|\le (1+\eps)^{2d}\|\de^c (x)\|$ for $x\in \ax_{\tilde \Theta^n}^{2d}(n), n\in\nz$ (if $n=\8$, $x\in \ax_{\Theta}^\8$ ). It follows that $$\|\de^c([1+(2k+1)^{-1}]^{-2d}T_{\phi_{k,\eta}^n}(x))\| \le \|\de^c (x)\|.$$ Since $x$ is a finite linear combination and $\lim_{k\to \8} | \phi_{k,\eta}^n(g) -1|=0$ for any $g\in \zz_n^{2d}$ as in the proof of Lemma \ref{cbapg}, we have
\begin{align*}
\lim_{k\to\8}  \|\de^c(T_{\phi_{k,\eta}^n}(x))\| =\| \de^c (x)\|.
\end{align*}
We deduce that
\begin{align*}
\|\delta^c(x)\| = \sup\{ \|\delta^c([1+(2k+1)^{-1}]^{-2d} T_{\phi_{k, \eta}^n}(x))\|: k\ge 1, \eta=[2d(2k+1)^{2d+1}]^{-1} \}.
\end{align*}
But $T_{\varphi_{k, \eta}^n}$ is a finite rank map and $\delta^c$ is continuous on a fixed finite-dimensional space. A similar argument holds true for $\|\delta^r(x)\|$. It follows that $\opnorm{\cdot}$ is a lower semicontinuous Lip-norm. Therefore, by Definition 2.44 in \cite{La15} and by the choice of $F$, $\ax_{\tilde \Theta^n}^{2d}(n)$ and $\ax_\Theta^{2d}$ are $F$-quasi-Leibniz quantum compact metric spaces. Here, in fact they are Leibniz quantum compact metric spaces. For notational convenience, we will write $\ax_{\Theta}^{2d}(n)$ or even $\ax_{\Theta}(n)$ for $\ax^{2d}_{\tilde \Theta^n}(n)$ in the following by abuse of notation.

Let $u_1^{\Theta}(n), ... , u_{2d}^{\Theta}(n)$ denote the generators of $\ax_{\Theta}^{2d}(n)$ and $u_1^\Theta, ... , u_{2d}^\Theta$ denote the generators of $\ax_\Theta^{2d}$. In the following let
\[
l = (l_1, ... , l_{2d})\in \zz^{2d}, \quad \lambda_n^\Theta(l) = u_1^{\Theta}(n)^{l_1}...u_{2d}^{\Theta}(n)^{l_{2d}}\quad\text{and}\quad \lambda^\Theta(l) = (u_1^{\Theta})^{l_1}...(u_{2d}^{\Theta})^{l_{2d}}.
\]
We understand that $\ax_{\Theta}^{2d}(\8) = \ax_\Theta^{2d}$ and $u_i^{\Theta}(\8) = u_i^\Theta$, $1\leq i\leq 2d$, are the generators of $\ax_\Theta^{2d}$. We will also use frequently the following convention: Let $B$ be a unital nuclear C$^*$-algebra equipped with a tracial state $\tau$. Suppose a Lip-norm $\opnorm{.}$ is defined on a dense subalgebra of $B$. For a finite linear combination $x=\sum_{k} a_k\otimes x_k\in \kx\otimes B$, we write $\mathring{x} = \sum_{k} a_k\otimes \mathring{x_k}$ where $\mathring{x_k}$ is the mean-zero part of $x_k$. Since the Lip-norms used here are defined via ergodic semigroups, $\mathring{x}= x-(\id\otimes\tau) (x)$. The set of mean-zero elements of $B$ will be denoted by $\mathring B$.

\begin{lemma}\label{lip-equiv}
Let $m>0$ and $\psi$ be the length function associated with the heat semigroup that was introduced previously. There exists a constant $C= C(m, \psi)$ such that for $n > 2m$ (including $n=\8$) and all $y\in \kx\otimes \ax_\Theta^{\Lambda_m^{2d}}(n)$, we have $\|\mathring{y}\|_{\kx\otimes  \ax_\Theta^{\Lambda_m^{2d}}(n)} \leq C\opnorm{\mathring{y}}_{\kx\otimes  \ax_\Theta^{\Lambda_m^{2d}}(n)}$.
\end{lemma}

\begin{proof}
Recall from Section \ref{anaest} the definition of $\nabla_p(\ax_{\Theta}^{2d}(n))$. For $x\in S_p(\nabla_p(\ax_{\Theta}^{2d}(n)))$, by \eqref{gahr2-2} we have
\begin{align*}
\|x\|_{S_p(\nabla_p(\ax_{\Theta}^{2d}(n)))} = \max\{\|\Gamma(x,x)^{1/2}\|_p,~ \|\Gamma(x^*,x^*)^{1/2}\|_p\}.
\end{align*}
Let $p = 2$ and $x=\sum_k a_k\otimes \lambda_n^\Theta(k)\in S_2(\nabla_2(\ax_{\Theta}^{2d}(n)))$. Then we have
\[
\|x\|^2 _{S_2(\nabla_2(\ax_{\Theta}^{2d}(n)))}= \sum_k\|a_k\|_{S_2}^2 \psi(k).
\]
Similar to \eqref{axpqr}, we define for fixed $k$,
\[
\phi: \nabla_2(\ax_\Theta^{2d}(n)) \to \cz, \quad \sum_{l} a_l\la_n^\Theta(l) \mapsto a_k\psi(k).
\]
Then we have $\|\phi \|_{\cb} = \|\phi\| \le 1$. Note that by \cite{Pi98}*{Lemma 1.7}, we have
\[
\|\phi\|_{\cb} = \| \id_{S_2}\otimes \phi: S_2(\nabla_2(\ax_{\Theta}^{2d}(n)))\to S_2 \| = \|\id_\kx\otimes \phi: \kx\otimes_{\min} \nabla_2(\ax_{\Theta}^{2d}(n)) \to \kx\|.
\]
Hence, we have for $x=\sum_k a_k\otimes \lambda_n^\Theta(k)\in \kx\otimes_{\min}\nabla_2(\ax_{\Theta}^{2d}(n))$,
\begin{equation}\label{psikak}
\sup_k\psi(k)^{1/2} \|a_k\|_\kx \leq \|x\|_{\kx\otimes_{\min} \nabla_2(\ax_{\Theta}^{2d}(n))}.
\end{equation}
Let $y = \sum_{k \in\Lambda_m^{2d}} b_k \lambda_n^\Theta(k) \in \nabla_2(\ax_\Theta^{\Lambda_m^{2d}}(n))$. Define a map
\[
\nu: \nabla_2(\ax_\Theta^{\Lambda_m^{2d}}(n)) \to \ell_\8(\Lambda_m^{2d})
\]
by $\nu(y) = (b_k)_{k\in \Lambda_m^{2d}}$ and let $\mu$ be the inverse of $\nu$. We have the following chain of maps
\[
\xymatrix{
(\ax^{\Lambda_m^{2d}}_\Theta(n) \cap \mathring{\ax}_\Theta(n), \opnorm{\cdot})\ar[r]^{\qquad  \id} & \nabla_2(\ax_\Theta^{\Lambda_m^{2d}}(n)) \ar[r]^{\nu} &\ell_\8(\Lambda_m^{2d}) \ar[r]^{\mu\quad} & (\ax^{\Lambda_m^{2d}}_\Theta(n), \|\cdot\|).
}
\]
By Proposition \ref{dep8}, we have $\|\id: (\ax^{\Lambda_m^{2d}}_\Theta(n), \opnorm{\cdot})\to \nabla_2 (\ax_\Theta^{\Lambda_m^{2d}}(n)) \|_{\cb} \leq c$ for some constant $c$. We deduce from \eqref{psikak} that
\[
\| \nu: \nabla_2(\ax_\Theta^{\Lambda_m^{2d}}(n)) \to \ell_\8(\Lambda_m^{2d}) \|_{\cb}\leq \frac{1}{\inf_{\substack{k\in \Lambda_m^{2d} \\ \psi(k)\neq 0} }\psi(k)^{1/2}}.
\]
Moreover, $\|\mu: \ell_\8(\Lambda_m^{2d}) \to \ax^{\Lambda_m^{2d}}_\Theta(n)\|_{\cb}\leq (2m+1)^{2d}$, since the cardinality of $\Lambda_m^{2d}$ is $(2m+1)^{2d}$.  This proves that $\|\mu\circ\nu\circ\id\|_{\cb} \leq C$ for some $C = C(m,\psi)$ and all $n > 2m$.
\end{proof}

Suppose $\eps>0$, $k\in \nz$ and $\varphi_{k,\eta}^n$ is the multiplier on $\zz_n^{2d}$ chosen as $\phi_{k,\eta}^n$ in the proof of Theorem \ref{atarcb}, which is supported on $\Lambda_m^{2d}$. We define the following multipliers for $n>2m, n\in \overline \nz$
\[
T_{\varphi_{k,\eta}^n}(\lambda_n^\Theta(l)) = \varphi_{k,\eta}^n(l) \lambda_n^\Theta(l),
\]
such that for $n > 2m$ we have
\[
\|T_{\varphi_{k,\eta}^n}: (\ax^{2d}_\Theta(n), \|.\|) \to (\ax^{2d}_\Theta(n), \|.\|)\|_{\cb} \leq 1 + \eta,
\]
and
\[
\|T_{\varphi_{k,\eta}}: (\ax^{2d}_\Theta(n), \opnorm{.}) \to (\ax^{2d}_\Theta(n), \opnorm{.})\|_{\cb} \leq 1 + \eta,
\]
for all $n>2m, n\in \overline{\nz}$.

\begin{cor}\label{uniform cb}
There exists $N>0$ such that the identity map $(\mathring\ax_{\Theta}^{2d}(n), \opnorm{.}) \to (\mathring\ax_{\Theta}^{2d}(n), \|.\|)$ is completely bounded uniformly for $n > N$ including $n = \8$.
\end{cor}

\begin{proof}
Let $\eps>0$. From the proof of Theorem \ref{atarcb}, we know that there exist $\eta < \eps$, $k = k(\eps)$,  $N>0$ and a multiplier $\varphi_{k,\eta}^n$ on $\ax_\Theta^{2d}(n)$  such that $\|T_{\varphi_{k,\eta}^n} \|_{\cb} \leq 1 + \eps$ for $n>N$. Using the same argument as that of Theorem \ref{cb compact}, we have for $n > N$,
\[
\| \id - T_{\varphi_{k,\eta}^n}: (\ax_{\Theta}^{2d}(n), \opnorm{.}) \to (\ax_{\Theta}^{2d}(n), \|.\|)\|_{\cb} \leq \eps.
\]
Let $\eps = 1$ and supp$(\varphi_{k,\eta}^n) = \Lambda_m^{2d}$, for some $m= m(k, \eta)$ independent of $n$. Then by Lemma \ref{lip-equiv}, we have
\[
\|\id: (\ax_\Theta^{\Lambda_m^{2d}}(n)\cap \mathring{\ax}_\Theta(n), \opnorm{.}) \to (\ax_\Theta^{\Lambda_m^{2d}}(n)\cap \mathring{\ax}_\Theta(n), \|.\|)\|_{\cb} \leq C(m,\psi),
\]
where $C(m, \psi)$ is the constant in Lemma \ref{lip-equiv}. Using these facts together with Lemma \ref{cbga} for $x=\sum_{k} a_k \otimes x_k\in \kx \otimes \ax_{\Theta}^{2d}(n)$, where $x_k$'s are mean-zero elements,  we have
\begin{align*}
\|x \|_{\kx \otimes \ax_{\Theta}^{2d}(n)} &\le \|x -T_{\varphi_{k,\eta}^n}(x) \|_{\kx \otimes \ax_{\Theta}^{2d}(n)} + \| T_{\varphi_{k,\eta}^n}(x)  \|_{\kx \otimes \ax_{\Theta}^{2d}(n)}\\
&\le \opnorm{x}_{\kx \otimes \ax_{\Theta}^{2d}(n)} + C(m, \psi) \opnorm{ T_{\varphi_{k,\eta}^n}(x)  }_{\kx \otimes \ax_{\Theta}^{2d}(n)}\le (1+2C(m, \psi) ) \opnorm{x}_{\kx \otimes \ax_{\Theta}^{2d}(n)}.
\end{align*}
Hence
\[
\sup_n\|\id: (\mathring\ax_{\Theta}^{2d}(n), \opnorm{.}) \to (\mathring\ax_{\Theta}^{2d}(n), \|.\|)\|_{\cb} \leq c,
\]
for some constant $c$ independent of $n$.
\end{proof}

Let $n_j$ be the subsequence we found in the proof of Proposition \ref{cbisom 2d}. Then we have $C^*(\rho_{n_j}^{\Theta}(\ax_{\Theta}^{\8})) = M_{n_j^d}$, and $\ax^{2d}_{\Theta}(n_j) \simeq M_{n_j^d}$. Let $B_{n_j}$ and $B_\8$ denote the spaces $\ax_{\Theta}^{2d}(n_j)$ and $\ax_{\Theta}^{2d}$, respectively. In the following we use the index $n$ instead of $n_j$ for simplicity. For any $m>0$, let $B_n^m$ and $B_\8^m$ denote the subspaces $\ax_{\Theta}^{\Lambda_m^{2d}}(n)$ and $\ax_{\Theta}^{\Lambda_m^{2d}}$, respectively.

Recall the subspace of noncommutative Laurent polynomials with constant coefficients
\[
\text{Poly}(x_1,...,x_{2d}) = \bigcup_{k\geq 1}\{p =\sum_{|i_1|,...,|i_{2d}| \leq k}a_{i_1...i_{2d}} x^{i_1}_1...x^{i_{2d}}_{2d}: a_{i_1...i_{2d}}  \in \cz\}.
\]
To ease the notation we denote an element $x = \sum_{|i_1|,...,|i_{2d}| \leq k}a_{i_1...i_{2d}} x^{i_1}_1...x^{i_{2d}}_{2d}\in \text{Poly}(x_1,...,x_{2d})$ by $\sum_{i\in \Lambda_k^{2d}} a_i x(i)$.  Note that the space $\text{Poly}(x_1,...,x_{2d})$ is a subspace of ${\rm Poly}_\vta(x_1,...,x_{2d})$ given in \eqref{e:polyxd}. Moreover, $\text{Poly} (x_1,...,x_{2d})$  is a subspace in the full group C$^*$-algebra $C^*(\fz_{2d})$ while ${\rm Poly}_\vta(x_1,...,x_{2d})$ is not. We will write Poly for $\text{Poly}(x_1,...,x_{2d})$ throughout the rest of the paper for simplicity. Since we are now considering C$^*$-algebras instead of order-unit spaces, it is more convenient to work with $C^*(\fz_{2d})\supset {\rm Poly}$. Let $g_1$,..., $g_{2d}$ denote the generators of $C^*(\fz_{2d})$. Define the following $^*$-homomorphisms
\[
\sigma^n_\Theta:    C^*(\fz_{2d})\to B_n, \quad \sigma_\Theta:   C^*(\fz_{2d}) \to B_\8
\]
by $ \sigma_\Theta(g_i) = u_i^\Theta$ and $ \sigma_\Theta^n(g_i) = u_i^\Theta(n)$, for $1\leq i\leq 2d$ and $n\in \nz$. Then we get a $^*$-homomorphism
\[
\sigma_\Theta^{\bullet} :  C^*(\fz_{2d})\to  \prod_n B_n = M_\8
\]
defined by $\sigma_\Theta^{\bullet} = (\sigma_\Theta^n)_n$. Note that $I = c_0(\{B_n\})$ is an ideal in $M_\8$. Hence we get the quotient map $q: M_\8\to M_\8/I$.
Since by Proposition \ref{cbcts 2d}, $\{(B_n)_{n\in \overline{\nz}}\}$ is a continuous field of C$^*$-algebras\footnote{More precisely, Proposition \ref{cbcts 2d} handles the self-adjoint elements of these algebras; however, the same argument shows that $\{(B_n)_{n\in \overline{\nz}}\}$ is a continuous field of C$^*$-algebras over $\overline\nz$ with fibers $B_n$.} over $\overline{\nz}$, we have
\begin{align}\label{bhat isom}
\|q\circ \sigma_\Theta^{\bullet}(x)\| = \|\sigma_\Theta(x)\|.
\end{align}
Define $\hat{B} = q\circ \sigma_\Theta^{\bullet}(C^*(\fz_{2d}))$. By (\ref{bhat isom}), since norms on $\hat{B}$ and $B_\8$ coincide, $\hat{B}$ is isomorphic to $B_\8$. Let
\[
B =  q^{-1}(\hat{B}) = \{\sigma_\Theta^{\bullet}(a) + z : a\in C^*(\fz_{2d}), z\in I\}.
\]
$B$ is the C$^*$-algebra generated by $c_0(\{B_n\})$ and $\sigma_\Theta^{\bullet}(C^*(\fz_{2d}))$. Then $B$ is a $C(\overline{\nz})$-algebra with fiber maps
\[
\eta_n: B\to B_n \quad\text{and} \quad \eta_\8 = q|_B: B\to B_\8,
\]
where $\eta_n$ is the projection of $B$ onto $B_n$. That is, $\eta_n((x_j)_j+ z) = x_n$, for $(x_j)_j \in M_\8$ and $z\in I$, and $\eta_\8(\sigma_\Theta^{\bullet}(x) + y) = \sigma_\Theta(x)$, for $x\in C^*(\fz_{2d})$, $y\in I$. Then the following sequence is exact
\[
0\to I\to  B\to \hat{B}\cong B_\8\to 0.
\]
Note that both $I$ and $B_\8$ are nuclear C$^*$-algebras (recall that $B_\8$ is an iterative crossed product). Hence $B$ is nuclear (see Proposition 10.1.3 in \cite{BO08}). Therefore, similar to Theorem \ref{cts n-dim}, by the aforementioned result of Kirchberg and Blanchard (see Theorem 3.2 in \cite{B97}), there exist a Hilbert space $\hx$ and a faithful $^*$-homomorphism $\pi: B\to C(\overline{\nz})\otimes \bx(\hx)$. Note that $\pi$ maps $C(\overline{\nz})$ to $C(\overline{\nz})$ canonically. Let $\iota_n: B_n\to B$ be defined by $\iota_n(y) = (y_l)_{l\in \nz}$, where
\[
y_l = \begin{cases}
y \quad l = n,\\
0 \quad \text{else}\\
\end{cases}
\]
for $y\in B_n$. Let ${\rm ev}_n: C(\overline{\nz})\otimes \bx(\hx)\to \bx(\hx)$ be defined by ${\rm ev}_n(f)=f(n)$ for $n\in \overline \nz$. Then we get the following $^*$-monomorphism
\[
\pi_n: B_n\to  \bx(\hx),\quad x \mapsto {\rm ev}_n\circ \pi\circ \iota_n(x)
\]
for $n\in \nz$. Let $\tilde \pi_\8 = {\rm ev}_\8\circ \pi: B\to \bx(\hx)$. Note that $B/I\simeq B_\8$ and that the kernel of $\tilde \pi_\8$ is $I$. Then there exists a $^*$-monomorphism $\pi_\8: B_\8 \to \bx(\hx)$ such that the following diagram commutes:
\[
\xymatrix{
B \ar[r]^{\tilde \pi_\8} \ar[d]_q &  \bx(\hx)\\
B/I\simeq B_\8\ar@{.>}[ur]_{\pi_\8} &
}
\]
In particular, for any $a\in {\rm Poly}$ we have
\begin{align}\label{e:pi8sita}
\pi_\8(\si_\Theta(a)) &= \pi_\8(q(\si_\Theta^\bullet(a)))={\rm ev}_\8(\pi(\si_\Theta^\bullet(a))) \\
&= \lim_{n\to \8} {\rm ev}_n( \pi (\si_\Theta^\bullet(a)1_{\{n\}}  ))= \lim_{n\to \8} \pi_n(\si_\Theta^n(a)),\nonumber
\end{align}
where $1_{\{n\}}$ denotes the indicator function of $\{n\}$.  The following diagram summarizes this argument
\[
\xymatrix{
&&M_\8 \ar[rr]^{q} && M_\8 / I &&\\
&&& C^*(\fz_{2d}) \ar[ul]_{\sigma_\Theta^{\bullet}} \ar[ur]^{q\circ \sigma_\Theta^{\bullet}} \ar[dl]_{\sigma_\Theta^{\bullet}} \ar[dr]^{q\circ \sigma_\Theta^{\bullet}}&&&\\
0\ar [r] & I \ar[r] &B\ar[rr] \ar@{^(->}[uu] && \hat{B} \ar@{=}[r]\ar@{^(->}[uu] & B_\8\ar[r] & 0\\
}
\]
where the last row is short exact.

\begin{lemma}\label{uniform sup}
With the notation above, the following hold:
\begin{enumerate}
\item $\lim_{n\to \8}\|\pi_n(\lambda_n^\Theta(l)) - \pi_\8(\lambda^\Theta(l))\| = 0$.
\item Let $\eps > 0$ and $m\in \nz$. There exists $N\in \nz$ such that for all $n > N$ and $x = \sum_{l\in \Lambda_m^{2d}} a_l\otimes  x(l) \in \kx\otimes \text{Poly}$, we have
\[
\|\id\otimes (\pi_n\circ \sigma_\Theta^n)(x) - \id\otimes (\pi_\8\circ \sigma_\Theta)(x)\|_{\bx(\ell_2 \otimes \hx)} \le \eps\sup_{l\in \Lambda_m^{2d}}\|a_l\|_\kx.
\]
\end{enumerate}
\end{lemma}

\begin{proof}
Note that (1) follows from \eqref{e:pi8sita}. To prove (2), let $\eps > 0$ and $\delta = \frac{\eps}{(2m+1)^{2d}}$. Using (1) and the triangle inequality, there exists $N\in \nz$ such that for all $n> N$, we have
\begin{align*}
\|\id\otimes (\pi_n\circ \sigma_\Theta^n)(x) &- \id\otimes (\pi_\8\circ \sigma_\Theta)(x)\| = \|\sum_{l\in \Lambda_m^{2d}} a_l\otimes (\pi_n(\lambda_n^\Theta(l)) - \pi_\8(\lambda^\Theta(l))) \| \\
&\leq (\sum _{l\in \Lambda_m^{2d}} \|a_l\|)\delta \leq (2m+1)^{2d} \sup_{l\in \Lambda_m^{2d}}\|a_l\| \frac{\eps}{(2m+1)^{2d}} = \eps\sup_{l\in \Lambda_m^{2d}}\|a_l\|,
\end{align*}
which proves the assertion.
\end{proof}

In the following, let $\pi_n$ and $\pi_\8$ be the faithful $^*$-homomorphisms as defined above.

\begin{theorem}\label{propinquity}
Let $\ax_\Theta$ be an even-dimensional noncommutative torus. Then there exist a sequence of matrix algebras with suitable Lip-norms converging to $\ax_\Theta$ in the sense of strong Gromov--Hausdorff propinquity.
\end{theorem}

\begin{proof}
Consider the bridge $\gamma = (\bx(\ell_2\otimes\hx), \id\otimes \pi_\8,\id\otimes \pi_n, \id)$. Then by Lemma \ref{l:propinquity}, it suffices to show that there exists a sequence $n_j$ such that for any $\eps > 0$ the following hold:
\begin{enumerate}
\item For all $a\in \kx\otimes B_\8$ such that $\opnorm{a} \leq 1$, there exists $b\in \kx\otimes B_n$ such that $\opnorm{b} \leq 1$ and $\|\id\otimes\pi_\8(a)-\id\otimes\pi_{n_j}(b)\| \le \eps$.
\item For all $b\in \kx\otimes B_n$ such that $\opnorm{b} \leq 1$, there exists $a\in \kx\otimes B_\8$ such that $\opnorm{a} \leq 1$ and $\|\id\otimes\pi_\8(a)-\id\otimes\pi_{n_j}(b)\| \le \eps$.
\end{enumerate}

Let $\eps > 0$. By Theorem \ref{cb compact}, there exist $0<\eta <\eps$, $k= k(\eps)$ and multipliers $\varphi_{k,\eta}^n$ on $B_n$ supported on $\Lambda_m^{2d}$, for some $m = m(k, \eta)$ independent of $n$, such that $\|T_{\varphi_{k,\eta}^n}\|_{\cb} \leq 1 + \eta$ and for all $n > 2m$ we have
\begin{align}\label{cb est}
\|\id - T_{\varphi_{k,\eta}^n}: (B_n, \opnorm{.}) \to (B_n, \|.\|)\|_{\cb} \leq \frac{\eps}{4}.
\end{align}
In the following, by abuse of notation, for all $n\in \overline{\nz}$, we denote $\id \otimes T_{\varphi_{k,\eta}^n} : (\kx\otimes B_n, \opnorm{\cdot}) \to (\kx\otimes B_n, \opnorm{\cdot})$ by $T_{\varphi_{k,\eta}^n}$. For any $x$ in $\kx\otimes B_n^m$ or $\kx\otimes B_\8^m$, let $\hat{x}$ denote the corresponding element in $\kx\otimes \text{Poly}$. Let $\delta < \frac{\eps}{4C(m,\psi)}$, where $C(m,\psi)$ is the constant from Lemma \ref{lip-equiv}. Using Lemma \ref{uniform sup} (2) and Lemma \ref{lip-equiv}, we can choose a subsequence $n_j$ such that for all $x \in \kx\otimes B_{n_j}$, by denoting $\tilde{x} = T_{\varphi_{k,\eta}^{n_j}}(\mathring{x})$, we have
\begin{align}\label{pi8pi}
\|\id\otimes(\pi_{n_j}\circ\sigma_\Theta^{n_j})(\hat{\tilde{x}})& -\id\otimes(\pi_\8\circ\sigma_\Theta)(\hat{\tilde{x}})\|_{\bx(\ell_2\otimes \hx)} \\\nonumber
&\leq \delta \sup_{l\in \Lambda_m^{2d}}\|a_l\|_\kx\le \de \|\tilde x\|_{\kx \otimes B_{n_j}^m} \leq \frac{\eps}{4}\opnorm{\tilde{x}}_{\kx\otimes B_{n_j}^m},
\end{align}
where $(a_l)$ are the coefficients of $T_{\varphi_{k,\eta}^{n_j}}(\mathring{x}) = \sum_{l\in \Lambda_m^{2d}} a_l\otimes \lambda_{n_j}^\Theta(l)$. From now on we abuse the notation and drop the index $j$ of $n_j$.

To prove (1), let $a\in \kx\otimes B_\8$ such that $\opnorm{a} \leq 1$. Let $x =T_{\varphi_{k,\eta}}(\mathring{a})\in \kx\otimes B_\8^m$. Hence $\hat{x}\in \kx\otimes \text{Poly}$. Let $b' = \id\otimes \sigma_\Theta^n(\hat{x})\in \kx\otimes B_n$. Clearly, $x$ and $b'$ are mean-zero. Then by (\ref{cb est}), (\ref{pi8pi}) we have
\begin{align*}
\|\id\otimes\pi_\8(\mathring{a})-\id\otimes\pi_n(b')\|&=\|\id\otimes\pi_\8(\mathring{a})-\id\otimes(\pi_n\circ\sigma_\Theta^n)(\hat{x})\|\\
&\leq \|\id\otimes \pi_\8 (\mathring{a})-\id\otimes(\pi_\8\circ\sigma_\Theta)(\hat{x})\|\\
&\ \ +\|\id\otimes(\pi_\8\circ\sigma_\Theta)(\hat{x})-\id\otimes(\pi_n\otimes\sigma_\Theta^n)(\hat{x})\|\\
&\leq \|\mathring{a}-T_{\varphi_{k,\eta}}(\mathring{a})\| + \frac{\eps}{4}\opnorm{T_{\varphi_{k,\eta}}(\mathring{a})} \\
&\leq \frac{\eps}{4} \opnorm{a}+ \frac{\eps}{4}(1 + \eta)\opnorm{a}
 \leq \frac{3\eps}{4}\opnorm{a}.
\end{align*}
By Corollary \ref{uniform cb}, $\|b'\|=\|\mathring{b'}\|\leq K\opnorm{b'}$ for some absolute constant $K$. If $\opnorm{b'}=0$, then $b'=0$ and $b=(1\otimes\tau)(a)1$ satisfies (1). Suppose $\opnorm{b'}\neq0$. Let $b = \frac{\opnorm{x}}{1+\eta}\frac{b'}{\opnorm{b'}}$. Then $\opnorm{b} = \frac{\opnorm{x}}{1+\eta}\le 1$. Recall from the discussion after Proposition \ref{cbisom 2d} that for $\eps'>0$, the map $\rho_n^\Theta: B_\8^m\to B_n^m$ is a $1 + \eps'$ Lip-isometry. Let $\eps' = \eta$. Note that $b' = \id\otimes \rho_n^\Theta(\si_{\Theta}(\hat{x}))$.  Hence, $(1 - \eta)\opnorm{x}\le \opnorm{b'} \leq (1 + \eta)\opnorm{x}$. Combining with $\opnorm{x}\le 1+\eta$ and $\|b'\| \leq K\opnorm{b'}$, we have
 \begin{align*}
 \|b' - b\| &=  \frac{| (1+\eta)\opnorm{b'} - \opnorm{x}|}{(1+\eta)\opnorm{b'}} \|b'\| \leq \frac{(\eta^2 + 2\eta)\opnorm{x}}{1+\eta} \frac{\|b'\|}{\opnorm{b'}} \le K (\eta^2 + 2\eta).
 \end{align*}
Therefore, if we choose $\eta$ small enough, we have
 \[
 \|\id\otimes\pi_n(b)-\id\otimes\pi_n(b')\| \leq  \|b - b'\| \leq K (\eta^2 + 2\eta) \leq \frac{\eps}{4}.
\]
By the triangle inequality, $b+(1\otimes \tau) (a)$ verifies (1).

To prove (2), let $b\in \kx\otimes B_n$ be such that $\opnorm{b} \leq 1$ and
\[
b' = T_{\varphi_{k,\eta}^n}(\mathring{b}) = \sum_{0\neq l\in \Lambda_m^{2d}} a_l\otimes \lambda^\Theta_n (l)\in \kx\otimes B_n.
\]
Then by (\ref{cb est}), we have $\|\mathring{b} - b'\| \leq \frac{\eps}{4} \opnorm{b}$. Let $\hat{b'} = \sum_{l\in \Lambda_m^2}a_l\otimes x(l)$ be the corresponding element in $\kx\otimes \text{Poly}$. Choose $a' = \sigma_\Theta(\hat{b'}) \in \kx\otimes B_\8$. Clearly, $b'$ and $a'$ are mean-zero. Then using (\ref{pi8pi}), we get
\begin{align*}
\|\id\otimes\pi_\8(a')-\id\otimes\pi_n(\mathring{b})\|& \leq \|\id\otimes(\pi_\8\circ\sigma_\Theta)(\hat{b'})-\id\otimes(\pi_n\circ\sigma_\Theta^n)(\hat{b'})\|\\
&\ \ +\|\id\otimes(\pi_n\circ\sigma_\Theta^n)(\hat{b'})- \id\otimes\pi_n(\mathring{b})\| \\
& \leq\frac{\eps}{4}\opnorm{b'} + \|\mathring{b} - b'\| \leq\frac{\eps(1 + \eta) +\eps}{4}\opnorm{b}.
\end{align*}
If $\opnorm{a'}=0$, then $a'=0$ and $a = (1\otimes \tau)(b)$ verifies (2). Suppose $\opnorm{a'}\neq 0$. Let $a =\frac{\opnorm{b'}}{1+\eta} \frac{a'}{\opnorm{a'}}$. Similar to (1), using the fact that $(\rho^\Theta_n)^{-1}$ is a $1 + \eta$ Lip-isometry, we get $\|a' - a\| \leq K(\eta^2 + 2\eta)$. Therefore, choosing $\eta$ small enough, we get
\[
\|\id\otimes\pi_{\8}(a) - \id\otimes\pi_\8(a')\|\leq \|a - a'\| \leq K(\eta^2 + 2\eta) \leq \frac{\eps}{4}.
\]
By the triangle inequality, $a+(1\otimes\tau)(b)$ verifies (2).

(1) and (2) together with Lemma \ref{l:propinquity} prove  the assertion.
\end{proof}

\begin{rem}
In Sections \ref{ct} and \ref{A_theta}, we chose $p>2$ for our estimates. Note that in the higher-dimensional case, the choice of $p$ depends on the dimension of the rotation algebra and the choice of the semigroup.
\end{rem}

Let us conclude this paper by describing the common metric space in which the Gromov--Hausdorff distance is realized. It turns out that this metric space is the dual unit ball of $B_n\dash B_\8$ bimodule given by a derivation.

\begin{lemma}\label{deriv oplus}
Let $\delta_1$ and $\delta_2$ be two derivations on the C$^*$-algebras $A$ and $B$, respectively, and $\pi_1: A\to \bx(\hx)$, $\pi_2: B\to \bx(\hx)$ be two $^*$-homomorphisms. Let $\lambda > 0$. Then
\[
\delta(a,b) =
  \begin{pmatrix}
    \delta_1(a) & 0 & 0\\
    0 & \delta_2(b) & 0\\
    0 & 0 &  \lambda(\pi_1(a)- \pi_2(b))
  \end{pmatrix}
\]
is a derivation on $A\oplus B$ with left and right actions respectively given by
\[
\varrho_l(a,b) = \begin{pmatrix}
a& 0 &0\\
0& b &0\\
0& 0& \pi_1(a)
\end{pmatrix},\quad
\varrho_r(a,b) = \begin{pmatrix}
a& 0&0\\
0&b&0\\
0&0& \pi_2(b)
\end{pmatrix}.
\]
\end{lemma}
\begin{proof}
It is a direct calculation. Let $a, a'\in A$ and $b,b'\in B$. Then
\begin{align*}
\de[(a,b)\cdot(a',b')]&= \de(aa', bb') \\
&= \begin{pmatrix}
a\de_1(a')+ \de(a) a'& 0&0\\
0& b\de_2(b') +\de_2(b)b'&0\\
0&0&\la(\pi_1(a)\pi_1(a')- \pi_2(b)\pi_2(b'))
\end{pmatrix}.
\end{align*}
This clearly coincides with $\varrho_l(a,b) \de(a',b')+ \de(a,b)\varrho_r(a',b')$ by the definition of left and right actions.
\end{proof}

Let us recall that the matrix Lip norms are given by two derivations $\delta^r$ and $\delta^c$ (see Lemma \ref{gahr2}). Let $\si$ (resp. $\id$) denote the left (resp. right) action of $B_n$ on $H_\psi^c \otimes B_n$ as given in \eqref{leftact}. Note that $\si(x^*)^*=\si(x)$. Hence,
\[
v_n(x) =   \begin{pmatrix}
     0 & \delta^c(x^*)^*\\
     \delta^c(x) & 0 \end{pmatrix}
\]
is a derivation with actions given by $\tilde{\sigma}(x) = \begin{pmatrix}
     x & 0\\
     0 &\sigma(x)
  \end{pmatrix}$. The same is true for $n= \8$.

For $a\in B_n$ and $b\in B_\8$, by Lemma \ref{deriv oplus}
\[
\delta(a,b) =
  \begin{pmatrix}
    v_n(a) & 0 & 0\\
    0 & v_\8(b) & 0\\
    0 & 0 &  \frac{1}{\eps}(\pi_n(a)- \pi_\8(b))
  \end{pmatrix}
\]
defines a derivation on $B_n\oplus B_\8$. We use the operator space
\[
X = \{\delta(a,b) : \ a\in B_n, b\in B_\8\}.
\]

Inspired by \cite{La16}, it is now easy to describe the Lipschitz embeddings of the state spaces. Recall the faithful $^*$-homomorphisms $\pi_n$ and $\pi_\8$ defined as above. By abuse of notation, let $\hx$ denote the Hilbert space on which both $B_n$ and $B_\8$ are faithfully represented by $\pi_n$ and $\pi_\8$, respectively.

\begin{theorem}\label{t:cbembed}
The maps $q_n: X\to (B_n, \opnorm{\cdot})$ and $q_\8: X\to (B_\8, \opnorm{\cdot})$ defined by
\[
q_n(\delta(a,b)) = \mathring{a}, \ q_\8(\delta(a,b)) = \mathring{b},
\]
respectively, are complete quotient maps. Moreover, we have
\[
\text{dist}_{H}(q_n^*(S(B_n)), q_\8^*(S(B_\8)))\leq\eps.
\]
\end{theorem}

\begin{proof}
Note that if $\de(a,b)=\de(a',b')$, then $\de^c(a)=\de^c(a')$ and $\de^c(b)=\de^c(b')$. It follows that $\mathring{a}=\mathring{a'}$ and $\mathring{b}=\mathring{b'}$. Hence, $q_n$ and $q_\8$ are well-defined. By Theorem \ref{propinquity}, for all $a\in B_n$ with $\opnorm{a}\leq 1$, there exists $b\in B_\8$ with $\opnorm{b} \leq 1$ such that $\|\pi_\8(a)-\pi_n(b)\|\le \eps$. Hence $\|\delta(a,b)\| \leq 1$. Similarly for all $b\in B_\8$ with $\opnorm{b}\leq 1$, there exists $a\in B_n$ with $\opnorm{a} \leq 1$ such that $\|\pi_\8(a)-\pi_n(b)\|\le \eps$. Moreover, these hold in the matrix-valued case as well. Hence $q_n$ and $q_\8$ are complete surjections.

Let $\phi \in S(B_n)$ be a state. Since $\pi_n(B_n)\subset \bx(\hx)$ and $\pi_n$ is faithful, we can find a Hahn--Banach extension $\hat{\phi}\in \bx(\hx)^*$ of $\phi$. For $b\in B_\8$, we may define $\psi(b)\lel \hat{\phi}(\pi_\8(b))$. Then we have
   \begin{align*}
    |q_n^*(\phi)(\de(a,b))-q_\8^*(\psi)(\de(a,b))| &\lel |\phi(\mathring{a}) - \psi (\mathring{b})| \\
  &  \lel |\hat{\phi}(\pi_n(\mathring{a}))-\hat{\phi}(\pi_\8(\mathring{b}))| \\
  & \kl \eps\|\hat{\phi}\|  \frac{\|\pi_n(a)-\pi_\8(b)\|}{\eps} \\
  & \kl  \eps \|\delta(a,b)\|_{X} \pl .
   \end{align*}
By interchanging the roles of $B_n$ and $B_\8$, we get $\text{dist}_{H}(q_n^*(S(B_n)),q_\8^*(S(B_\8)))\le \eps$.
\end{proof}

\begin{rem}\label{r:disthmr}
The conclusion of Theorem \ref{t:cbembed} can be extended to state spaces on $M_r(B_n)$ and $M_r(B_\8)$.  However, for matrix-valued quantum compact metric spaces, some adjustments are in order. We are grateful to one of the referees for emphasizing this point. In the scalar valued case it is usually assumed that only for multiple of the identity the Lipschitz norm vanishes. Let $A$ be a C$^*$-algebra and $\ax$ a dense subalgebra of $A$ on which the Lip-norm $\| . \|_{Lip}$ is defined. Following the Banach space tradition, we shall then consider $(\ax,\|. \|_{Lip})$ as the quotient space $\ax/N$, assuming that
 \[ N \lel \{x\in\ax : \|x\|_{Lip}\lel 0\} \]
 is a linear subspace. However, if we understand $(\ax,\|.\|_{Lip})$ as such a quotient space, we no longer have an inclusion $\ax\subset A$, and we may no longer follow Rieffel's definition of compactness. This can be easily corrected. Let $E(x)\lel \phi_0(x)1$ be a completely positive conditional expectation onto $\cz 1$, where $\phi_0$ is a state. Then the map $\id-E$ vanishes on $\cz 1$ and hence we obtain a lift
 \[
 \xymatrix{
 & A \ar[d] \ar[r]^{\id -E} &A\\
 \ax/N~ \ar@{^(->}[r]& A/N\ar[ru]_{\widehat{\id-E}} &
 }
 \]
In our case the map is the one sending mean-zero elements into $A$ and $A/N$ is the closure of the linear span of elements not involving $1$. The map $E$ depends on the choice of $\phi_0$, and induces the modified metric
 \[ d_{Lip}(\phi,\psi)\lel \|(\phi-\phi \circ E)-(\psi-\psi\circ E)\|_{ (\ax/N,\|.\|_{Lip})^*} .
 \]
Note however, that in the weak$^*$ topology the affine map $T(\phi)=\phi-\phi \circ E=\phi-\phi_0$ is a homeomorphism.

For matrix-valued quantum metric spaces, we expect that elements $M_m\ten 1\subset M_m(\ax)$ have Lipschitz norm $0$. And for compatible matrix-valued spaces, in an `ergodic situation', we should expect that the null space  $N(M_m(\ax))=M_m(N(\ax))=M_m\ten 1$ is exactly given by this null space. In this situation the operator space associated to the `operator-Lipschitz norm' is the \emph{quotient space} $(M_m(\ax/N),\opnorm{.}_m)$, \emph{not the subspace}. Here $\opnorm{.}_m$ is the sequence of matrix norms which define the operator space structure of $A/N$ such that $\opnorm{.}_1=\|.\|_{Lip}$. Of course, we may still have another copy of $(M_m(\ax/N),\opnorm{.}_m)$ inside $M_m(A)$ via the map $\widehat{\id-E}$. Strictly speaking, without such a choice of $E$ it is even problematic to say what it means to restrict a state to the Lipschitz normed space $\ax$. However, $(\widehat{\id-E})^*$ is always well-defined in Rieffel's setting. Thus for matrix-valued states $\phi,\psi \in (M_m(A))^*$ we should consider the Lipschitz metric
 \[
 d_{Lip}(\phi,\psi)
 \lel \|(\phi-  E^m(\phi))-(\psi- E^m(\psi))\|_{(M_m(\ax/N),\opnorm{.}_m)^*}  .
 \]
Here $E^m=(\id_{M_m}\ten E)^*$ is the adjoint of the amplification. This operator space modification, could also serve as a model for quantum metric spaces with a non-trivial null space for the Lipschitz norm and it seems that the `inclusion map' $\iota:\ax/N \to A$ should be part of the given data. Note that here we have $[\id_{M_m}\otimes (\id-E)]^*(\phi) = \phi-\phi\circ(\id_{M_m}\otimes E) = \phi-E^m(\phi)$.

Let us return to the setting of Theorem \ref{t:cbembed}. Let $E(x)=\tau(x)$ for $x\in B_n$, $n\in\overline \nz$, where $\tau$ is the tracial state on $B_n$. Let $\tilde{q}_n=(\id-E)q_n:X\to B_n$. Using conditions (1) and (2) in Theorem \ref{propinquity} and following the argument of Theorem \ref{t:cbembed}, we have
\[
\text{dist}_{H}((\id_{M_r}\otimes\tilde q_n)^*[S(M_r(B_n))], (\id_{M_r}\otimes\tilde q_\8)^*[S(M_r(B_\8)] ) \leq \eps,
\]
where $S(M_r(B_n))$ denotes the state space of $M_r(B_n)$.
\end{rem}

\begin{rem}

For the case $d=1$, we can explicitly construct the $^*$-homomorphisms $\pi_n$ and $\pi_\8$ of Theorem \ref{propinquity}; see Section \ref{ct}.

Usually in the commutative case, derivations and their modules are given by vector bundles and therefore they have a geometric meaning. It would be interesting to understand the geometric interpretation of the distance estimates obtained above.
\end{rem}

\bibliographystyle{abbrv}
\bibliography{qgh}

\begin{bibdiv}
\begin{biblist}

\bib{BL76}{book}{
      author={Bergh, J{\"o}ran},
      author={L{\"o}fstr{\"o}m, J{\"o}rgen},
       title={Interpolation spaces. {A}n introduction},
   publisher={Springer-Verlag, Berlin-New York},
        date={1976},
        note={Grundlehren der Mathematischen Wissenschaften, No. 223},
      review={\MR{0482275 (58 \#2349)}},
}

\bib{B97}{article}{
      author={Blanchard, Etienne},
       title={Subtriviality of continuous fields of nuclear {$C^*$}-algebras},
        date={1997},
        ISSN={0075-4102},
     journal={J. Reine Angew. Math.},
      volume={489},
       pages={133\ndash 149},
         url={http://dx.doi.org/10.1515/crll.1997.489.133},
      review={\MR{1461207}},
}

\bib{BO08}{book}{
      author={Brown, Nathanial~P.},
      author={Ozawa, Narutaka},
       title={{$C^*$}-algebras and finite-dimensional approximations},
      series={Graduate Studies in Mathematics},
   publisher={American Mathematical Society},
     address={Providence, RI},
        date={2008},
      volume={88},
        ISBN={978-0-8218-4381-9; 0-8218-4381-8},
      review={\MR{2391387 (2009h:46101)}},
}

\bib{B1}{book}{
      author={Boca, Florin-Petre},
       title={Rotation {$C^{*}$}-algebras and almost {M}athieu operators},
      series={Theta Series in Advanced Mathematics},
   publisher={The Theta Foundation, Bucharest},
        date={2001},
      volume={1},
        ISBN={973-99097-7-9},
      review={\MR{1895184 (2003e:47063)}},
}

\bib{Bou}{incollection}{
      author={Bourgain, Jean},
       title={Vector-valued singular integrals and the {$H^1$}-{BMO} duality},
        date={1986},
   booktitle={Probability theory and harmonic analysis ({C}leveland, {O}hio,
  1983)},
      series={Monogr. Textbooks Pure Appl. Math.},
      volume={98},
   publisher={Dekker, New York},
       pages={1\ndash 19},
      review={\MR{830227 (87j:42049b)}},
}

\bib{CDS}{article}{
      author={Connes, Alain},
      author={Douglas, Michael~R.},
      author={Schwarz, Albert},
       title={Noncommutative geometry and matrix theory: compactification on
  tori},
        date={1998},
        ISSN={1029-8479},
     journal={J. High Energy Phys.},
      number={2},
       pages={Paper 3, 35 pp. (electronic)},
         url={http://dx.doi.org/10.1088/1126-6708/1998/02/003},
      review={\MR{1613978}},
}

\bib{Co94}{book}{
      author={Connes, Alain},
       title={Noncommutative geometry},
   publisher={Academic Press, Inc., San Diego, CA},
        date={1994},
        ISBN={0-12-185860-X},
      review={\MR{1303779}},
}

\bib{Co83}{article}{
      author={Cowling, Michael~G.},
       title={Harmonic analysis on semigroups},
        date={1983},
        ISSN={0003-486X},
     journal={Ann. of Math. (2)},
      volume={117},
      number={2},
       pages={267\ndash 283},
         url={http://dx.doi.org/10.2307/2007077},
      review={\MR{690846 (84h:43004)}},
}

\bib{CS03}{article}{
      author={Cipriani, Fabio},
      author={Sauvageot, Jean-Luc},
       title={Derivations as square roots of {D}irichlet forms},
        date={2003},
        ISSN={0022-1236},
     journal={J. Funct. Anal.},
      volume={201},
      number={1},
       pages={78\ndash 120},
         url={http://dx.doi.org/10.1016/S0022-1236(03)00085-5},
      review={\MR{1986156 (2004e:46080)}},
}

\bib{CXY}{article}{
      author={Chen, Zeqian},
      author={Xu, Quanhua},
      author={Yin, Zhi},
       title={Harmonic analysis on quantum tori},
        date={2013},
        ISSN={0010-3616},
     journal={Comm. Math. Phys.},
      volume={322},
      number={3},
       pages={755\ndash 805},
         url={http://dx.doi.org/10.1007/s00220-013-1745-7},
      review={\MR{3079331}},
}

\bib{Dav}{book}{
      author={Davidson, Kenneth~R.},
       title={{$C^*$}-algebras by example},
      series={Fields Institute Monographs},
   publisher={American Mathematical Society, Providence, RI},
        date={1996},
      volume={6},
        ISBN={0-8218-0599-1},
      review={\MR{1402012 (97i:46095)}},
}

\bib{Dix77}{book}{
      author={Dixmier, Jacques},
       title={{$C\sp*$}-algebras},
   publisher={North-Holland Publishing Co., Amsterdam-New York-Oxford},
        date={1977},
        ISBN={0-7204-0762-1},
        note={Translated from the French by Francis Jellett, North-Holland
  Mathematical Library, Vol. 15},
      review={\MR{0458185}},
}

\bib{GKP}{article}{
      author={Grosse, H.},
      author={Klim{\v{c}}{\'{\i}}k, C.},
      author={Pre{\v{s}}najder, P.},
       title={Field theory on a supersymmetric lattice},
        date={1997},
        ISSN={0010-3616},
     journal={Comm. Math. Phys.},
      volume={185},
      number={1},
       pages={155\ndash 175},
         url={http://dx.doi.org/10.1007/s002200050085},
      review={\MR{1463037}},
}

\bib{Hor90}{book}{
      author={H\"ormander, Lars},
       title={The analysis of linear partial differential operators. {I}},
     edition={Second},
      series={Grundlehren der Mathematischen Wissenschaften [Fundamental
  Principles of Mathematical Sciences]},
   publisher={Springer-Verlag, Berlin},
        date={1990},
      volume={256},
        ISBN={3-540-52345-6},
         url={http://dx.doi.org/10.1007/978-3-642-61497-2},
        note={Distribution theory and Fourier analysis},
      review={\MR{1065993}},
}

\bib{HL01}{article}{
      author={Ho, Pei-Ming},
      author={Li, Miao},
       title={Fuzzy spheres in {A}d{S}/{CFT} correspondence and holography from
  noncommutativity},
        date={2001},
        ISSN={0550-3213},
     journal={Nuclear Phys. B},
      volume={596},
      number={1-2},
       pages={259\ndash 272},
         url={http://dx.doi.org/10.1016/S0550-3213(00)00594-0},
      review={\MR{1814253}},
}

\bib{HR}{article}{
      author={Haagerup, Uffe},
      author={R{\o}rdam, Mikael},
       title={Perturbations of the rotation {$C^\ast$}-algebras and of the
  {H}eisenberg commutation relation},
        date={1995},
        ISSN={0012-7094},
     journal={Duke Math. J.},
      volume={77},
      number={3},
       pages={627\ndash 656},
         url={http://dx.doi.org/10.1215/S0012-7094-95-07720-5},
      review={\MR{1324637 (96e:46073)}},
}

\bib{JLMX06}{article}{
      author={Junge, Marius},
      author={Le~Merdy, Christian},
      author={Xu, Quanhua},
       title={{$H^\infty$} functional calculus and square functions on
  noncommutative {$L^p$}-spaces},
        date={2006},
        ISSN={0303-1179},
     journal={Ast\'erisque},
      number={305},
       pages={vi+138},
      review={\MR{2265255}},
}

\bib{JM10}{article}{
      author={Junge, Marius},
      author={Mei, Tao},
       title={Noncommutative {R}iesz transforms---a probabilistic approach},
        date={2010},
        ISSN={0002-9327},
     journal={Amer. J. Math.},
      volume={132},
      number={3},
       pages={611\ndash 680},
         url={http://dx.doi.org/10.1353/ajm.0.0122},
      review={\MR{2666903 (2012b:46136)}},
}

\bib{JM12}{article}{
      author={Junge, M.},
      author={Mei, T.},
       title={B{MO} spaces associated with semigroups of operators},
        date={2012},
        ISSN={0025-5831},
     journal={Math. Ann.},
      volume={352},
      number={3},
       pages={691\ndash 743},
         url={http://dx.doi.org/10.1007/s00208-011-0657-0},
      review={\MR{2885593}},
}

\bib{JMP10}{article}{
      author={Junge, Marius},
      author={Mei, Tao},
      author={Parcet, Javier},
       title={Smooth {F}ourier multipliers on group von {N}eumann algebras},
        date={2014},
        ISSN={1016-443X},
     journal={Geom. Funct. Anal.},
      volume={24},
      number={6},
       pages={1913\ndash 1980},
         url={http://dx.doi.org/10.1007/s00039-014-0307-2},
      review={\MR{3283931}},
}

\bib{J3P13}{article}{
      author={{Junge}, M.},
      author={{Palazuelos}, C.},
      author={{Parcet}, J.},
      author={{Perrin}, M.},
       title={{Hypercontractivity in group von Neumann algebras}},
        date={2013-04},
     journal={ArXiv e-prints},
      eprint={1304.5789},
}

\bib{JX07}{article}{
      author={Junge, Marius},
      author={Xu, Quanhua},
       title={Noncommutative maximal ergodic theorems},
        date={2007},
        ISSN={0894-0347},
     journal={J. Amer. Math. Soc.},
      volume={20},
      number={2},
       pages={385\ndash 439},
         url={http://dx.doi.org/10.1090/S0894-0347-06-00533-9},
      review={\MR{2276775}},
}

\bib{JZ12}{article}{
      author={Junge, Marius},
      author={Zeng, Qiang},
       title={Noncommutative martingale deviation and {P}oincar\'e type
  inequalities with applications},
        date={2015},
        ISSN={0178-8051},
     journal={Probab. Theory Related Fields},
      volume={161},
      number={3-4},
       pages={449\ndash 507},
         url={http://dx.doi.org/10.1007/s00440-014-0552-1},
      review={\MR{3334274}},
}

\bib{JZ13}{article}{
      author={Junge, Marius},
      author={Zeng, Qiang},
       title={Subgaussian 1-cocycles on discrete groups},
        date={2015},
        ISSN={0024-6107},
     journal={J. Lond. Math. Soc. (2)},
      volume={92},
      number={2},
       pages={242\ndash 264},
         url={http://dx.doi.org/10.1112/jlms/jdv025},
      review={\MR{3404023}},
}

\bib{Kas88}{article}{
      author={Kasparov, G.~G.},
       title={Equivariant {$KK$}-theory and the {N}ovikov conjecture},
        date={1988},
        ISSN={0020-9910},
     journal={Invent. Math.},
      volume={91},
      number={1},
       pages={147\ndash 201},
         url={http://dx.doi.org/10.1007/BF01404917},
      review={\MR{918241}},
}

\bib{KL09}{article}{
      author={Kerr, David},
      author={Li, Hanfeng},
       title={On {G}romov-{H}ausdorff convergence for operator metric spaces},
        date={2009},
        ISSN={0379-4024},
     journal={J. Operator Theory},
      volume={62},
      number={1},
       pages={83\ndash 109},
      review={\MR{2520541 (2010h:46087)}},
}

\bib{Kr99}{thesis}{
      author={Krogh, Morten},
       title={{Noncommutative geometry and twisted little string theories}},
        type={Ph.D. Thesis},
        date={1999},
         url={http://alice.cern.ch/format/showfull?sysnb=0322384},
}

\bib{KS00}{article}{
      author={Konechny, Anatoly},
      author={Schwarz, Albert},
       title={Moduli spaces of maximally supersymmetric solutions on
  non-commutative tori and non-commutative orbifolds},
        date={2000},
        ISSN={1029-8479},
     journal={J. High Energy Phys.},
      number={9},
       pages={Paper 5, 24},
         url={http://dx.doi.org/10.1088/1126-6708/2000/09/005},
      review={\MR{1789106}},
}

\bib{Lan95}{book}{
      author={Lance, E.~C.},
       title={Hilbert {$C^*$}-modules},
      series={London Mathematical Society Lecture Note Series},
   publisher={Cambridge University Press},
     address={Cambridge},
        date={1995},
      volume={210},
        ISBN={0-521-47910-X},
         url={http://dx.doi.org/10.1017/CBO9780511526206},
        note={A toolkit for operator algebraists},
      review={\MR{1325694 (96k:46100)}},
}

\bib{La05}{article}{
      author={Latr{\'e}moli{\`e}re, Fr{\'e}d{\'e}ric},
       title={Approximation of quantum tori by finite quantum tori for the
  quantum {G}romov-{H}ausdorff distance},
        date={2005},
        ISSN={0022-1236},
     journal={J. Funct. Anal.},
      volume={223},
      number={2},
       pages={365\ndash 395},
         url={http://dx.doi.org/10.1016/j.jfa.2005.01.003},
      review={\MR{2142343 (2006d:46092)}},
}

\bib{La15}{article}{
      author={{Latremoliere}, F.},
       title={{Quantum Metric Spaces and the Gromov-Hausdorff Propinquity}},
        date={2015-06},
     journal={ArXiv e-prints},
      eprint={1506.04341},
}

\bib{La15a}{article}{
      author={Latr{\'e}moli{\`e}re, Fr{\'e}d{\'e}ric},
       title={The dual {G}romov-{H}ausdorff propinquity},
        date={2015},
        ISSN={0021-7824},
     journal={J. Math. Pures Appl. (9)},
      volume={103},
      number={2},
       pages={303\ndash 351},
         url={http://dx.doi.org/10.1016/j.matpur.2014.04.006},
      review={\MR{3298361}},
}

\bib{La16}{article}{
      author={Latr{\'e}moli{\`e}re, Fr{\'e}d{\'e}ric},
       title={The quantum {G}romov-{H}ausdorff propinquity},
        date={2016},
        ISSN={0002-9947},
     journal={Trans. Amer. Math. Soc.},
      volume={368},
      number={1},
       pages={365\ndash 411},
         url={http://dx.doi.org/10.1090/tran/6334},
      review={\MR{3413867}},
}

\bib{Li06}{article}{
      author={Li, Hanfeng},
       title={Order-unit quantum {G}romov-{H}ausdorff distance},
        date={2006},
        ISSN={0022-1236},
     journal={J. Funct. Anal.},
      volume={231},
      number={2},
       pages={312\ndash 360},
         url={http://dx.doi.org/10.1016/j.jfa.2005.03.016},
      review={\MR{2195335 (2006k:46119)}},
}

\bib{Ma99}{book}{
      author={Madore, John},
       title={An introduction to noncommutative differential geometry and its
  physical applications},
     edition={Second},
      series={London Mathematical Society Lecture Note Series},
   publisher={Cambridge University Press, Cambridge},
        date={1999},
      volume={257},
        ISBN={0-521-65991-4},
         url={http://dx.doi.org/10.1017/CBO9780511569357},
      review={\MR{1707285}},
}

\bib{Pi03}{book}{
      author={Pisier, Gilles},
       title={Introduction to operator space theory},
      series={London Mathematical Society Lecture Note Series},
   publisher={Cambridge University Press},
     address={Cambridge},
        date={2003},
      volume={294},
        ISBN={0-521-81165-1},
      review={\MR{2006539 (2004k:46097)}},
}

\bib{Pi98}{article}{
      author={Pisier, Gilles},
       title={Non-commutative vector valued {$L_p$}-spaces and completely
  {$p$}-summing maps},
        date={1998},
        ISSN={0303-1179},
     journal={Ast\'erisque},
      number={247},
       pages={vi+131},
      review={\MR{1648908}},
}

\bib{PX03}{incollection}{
      author={Pisier, Gilles},
      author={Xu, Quanhua},
       title={Non-commutative {$L^p$}-spaces},
        date={2003},
   booktitle={Handbook of the geometry of {B}anach spaces, {V}ol.\ 2},
   publisher={North-Holland, Amsterdam},
       pages={1459\ndash 1517},
         url={http://dx.doi.org/10.1016/S1874-5849(03)80041-4},
      review={\MR{1999201 (2004i:46095)}},
}

\bib{PX97}{article}{
      author={Pisier, Gilles},
      author={Xu, Quanhua},
       title={Non-commutative martingale inequalities},
        date={1997},
        ISSN={0010-3616},
     journal={Comm. Math. Phys.},
      volume={189},
      number={3},
       pages={667\ndash 698},
         url={http://dx.doi.org/10.1007/s002200050224},
      review={\MR{1482934 (98m:46079)}},
}

\bib{Rie}{incollection}{
      author={Rieffel, Marc~A.},
       title={Compact quantum metric spaces},
        date={2004},
   booktitle={Operator algebras, quantization, and noncommutative geometry},
      series={Contemp. Math.},
      volume={365},
   publisher={Amer. Math. Soc.},
     address={Providence, RI},
       pages={315\ndash 330},
         url={http://dx.doi.org/10.1090/conm/365/06709},
      review={\MR{2106826 (2005h:46099)}},
}

\bib{Ri04}{article}{
      author={Rieffel, Marc~A.},
       title={Gromov-{H}ausdorff distance for quantum metric spaces},
        date={2004},
        ISSN={0065-9266},
     journal={Mem. Amer. Math. Soc.},
      volume={168},
      number={796},
       pages={1\ndash 65},
         url={http://dx.doi.org/10.1090/memo/0796},
        note={Appendix 1 by Hanfeng Li, Gromov-Hausdorff distance for quantum
  metric spaces. Matrix algebras converge to the sphere for quantum
  Gromov-Hausdorff distance},
      review={\MR{2055927}},
}

\bib{R89}{article}{
      author={Rieffel, Marc~A.},
       title={Continuous fields of {$C^*$}-algebras coming from group cocycles
  and actions},
        date={1989},
        ISSN={0025-5831},
     journal={Math. Ann.},
      volume={283},
      number={4},
       pages={631\ndash 643},
         url={http://dx.doi.org/10.1007/BF01442857},
      review={\MR{990592 (90b:46120)}},
}

\bib{Ri90}{incollection}{
      author={Rieffel, Marc~A.},
       title={Noncommutative tori---a case study of noncommutative
  differentiable manifolds},
        date={1990},
   booktitle={Geometric and topological invariants of elliptic operators
  ({B}runswick, {ME}, 1988)},
      series={Contemp. Math.},
      volume={105},
   publisher={Amer. Math. Soc., Providence, RI},
       pages={191\ndash 211},
         url={http://dx.doi.org/10.1090/conm/105/1047281},
      review={\MR{1047281}},
}

\bib{Sc98}{article}{
      author={Schwarz, Albert},
       title={Morita equivalence and duality},
        date={1998},
        ISSN={0550-3213},
     journal={Nuclear Phys. B},
      volume={534},
      number={3},
       pages={720\ndash 738},
         url={http://dx.doi.org/10.1016/S0550-3213(98)00550-1},
      review={\MR{1663471}},
}

\bib{S56}{article}{
      author={Stein, Elias~M.},
       title={Interpolation of linear operators},
        date={1956},
        ISSN={0002-9947},
     journal={Trans. Amer. Math. Soc.},
      volume={83},
       pages={482\ndash 492},
      review={\MR{0082586 (18,575d)}},
}

\bib{SW99}{article}{
      author={Seiberg, Nathan},
      author={Witten, Edward},
       title={String theory and noncommutative geometry},
        date={1999},
        ISSN={1029-8479},
     journal={J. High Energy Phys.},
      number={9},
       pages={Paper 32, 93 pp. (electronic)},
         url={http://dx.doi.org/10.1088/1126-6708/1999/09/032},
      review={\MR{1720697}},
}

\bib{Wu06}{article}{
      author={Wu, Wei},
       title={Quantized {G}romov-{H}ausdorff distance},
        date={2006},
        ISSN={0022-1236},
     journal={J. Funct. Anal.},
      volume={238},
      number={1},
       pages={58\ndash 98},
         url={http://dx.doi.org/10.1016/j.jfa.2005.02.017},
      review={\MR{2234123 (2007h:46088)}},
}

\bib{XXX}{article}{
      author={Xia, Runlian},
      author={Xiong, Xiao},
      author={Xu, Quanhua},
       title={Characterizations of operator-valued {H}ardy spaces and
  applications to harmonic analysis on quantum tori},
        date={2016},
        ISSN={0001-8708},
     journal={Adv. Math.},
      volume={291},
       pages={183\ndash 227},
         url={http://dx.doi.org/10.1016/j.aim.2015.12.023},
      review={\MR{3459017}},
}

\bib{XXY}{article}{
      author={{Xiong}, X.},
      author={{Xu}, Q.},
      author={{Yin}, Z.},
       title={{Sobolev, Besov and Triebel-Lizorkin spaces on quantum tori}},
        date={2015-07},
     journal={ArXiv e-prints},
      eprint={1507.01789},
}

\bib{Ze13}{article}{
      author={Zeng, Qiang},
       title={Poincar\'e type inequalities for group measure spaces and related
  transportation cost inequalities},
        date={2014},
        ISSN={0022-1236},
     journal={J. Funct. Anal.},
      volume={266},
      number={5},
       pages={3236\ndash 3264},
         url={http://dx.doi.org/10.1016/j.jfa.2013.12.005},
      review={\MR{3158723}},
}

\end{biblist}
\end{bibdiv}
\end{document}